\documentclass[a4paper,11pt,oneside]{article}
\usepackage[margin=3cm]{geometry}

\usepackage{latexsym,amsmath,amsthm,amssymb,calc}
\usepackage[utf8]{inputenc}
\usepackage{dsfont}
\usepackage{xcolor}
\usepackage{graphicx}
\usepackage{mdwlist}
\usepackage{enumerate}
\usepackage{setspace}
\usepackage{multicol}
\usepackage{multirow}
\usepackage{cite}

\newtheorem{theorem}{Theorem}[section]
\newtheorem{lemma}[theorem]{Lemma}
\newtheorem{corollary}[theorem]{Corollary}
\newtheorem{proposition}[theorem]{Proposition}

\theoremstyle{definition}
\newtheorem{definition}[theorem]{Definition}
\newtheorem{example}[theorem]{Example}
\newtheorem{assumptions}[theorem]{Assumptions}

\theoremstyle{remark}
\newtheorem{remark}[theorem]{Remark}

\numberwithin{equation}{section}

\usepackage{algorithm}
\usepackage{algorithmicx}
\usepackage{algpseudocode}
\algrenewcommand\algorithmicrequire{\makebox[32pt][l]{\textrm{input}}}
\algrenewcommand\algorithmicensure{\makebox[32pt][l]{\textrm{output}}}
\algrenewcommand\algorithmicfunction{\textrm{function}}
\algrenewcommand\algorithmicwhile{\textrm{while}}
\algrenewcommand\algorithmicdo{}
\algrenewcommand\algorithmicend{\textrm{end}}
\algrenewcommand\algorithmicforall{\textrm{for all}}
\algrenewcommand\algorithmicfor{\textrm{for}}
\algrenewcommand\algorithmicrepeat{\textrm{repeat}}
\algrenewcommand\algorithmicuntil{\textrm{until}}

\DeclareFontFamily{U}{mathx}{\hyphenchar\font45}
\DeclareFontShape{U}{mathx}{m}{n}{
      <5> <6> <7> <8> <9> <10>
      <10.95> <12> <14.4> <17.28> <20.74> <24.88>
      mathx10
      }{}
\DeclareSymbolFont{mathx}{U}{mathx}{m}{n}
\DeclareMathSymbol{\bigtimes}{1}{mathx}{"91}

\newcommand{\N}{\mathds{N}}
\newcommand{\R}{\mathds{R}}
\newcommand{\Z}{\mathds{Z}}
\newcommand{\C}{\mathds{C}}

\DeclareMathOperator{\supp}{supp}

\DeclareMathOperator*{\argmin}{arg\,min}

\DeclareMathOperator{\range}{range}
\DeclareMathOperator{\ops}{ops}

\newcommand{\id}{{\rm id}}

\newcommand{\spl}[1]{{\rm \ell}_{#1}}

\newcommand{\spH}[1]{{\rm H}^{#1}}

\newcommand{\spL}[1]{{\rm L}_{#1}}

\newcommand{\zinterval}[2]{\{{#1},\ldots,{#2}\}}

\newcommand{\Ocal}{{\mathcal{O}}}

\newcommand{\Hcal}{{\mathcal{H}}}
\newcommand{\Rcal}{{\mathcal{R}}}

\newcommand{\Ical}{{\mathcal{I}}}

\newcommand{\Acal}{{\mathcal{A}}}

\newcommand{\Ncal}{{\mathcal{N}}}
\newcommand{\Rank}{\Rcal}

\renewcommand{\Re}{\operatorname{Re}}
\renewcommand{\Im}{\operatorname{Im}}

\DeclareMathOperator{\rank}{rank}

\DeclareMathOperator{\apply}{\textsc{apply}}
\DeclareMathOperator{\coarsen}{\textsc{coarsen}}
\DeclareMathOperator{\rhs}{\textsc{rhs}}
\DeclareMathOperator{\recompress}{\textsc{recompress}}
\DeclareMathOperator{\solve}{\textsc{solve}}

\providecommand{\abs}[1]{\lvert#1\rvert}
\providecommand{\bigabs}[1]{\bigl\lvert#1\bigr\rvert}

\providecommand{\biggabs}[1]{\biggl\lvert#1\biggr\rvert}

\providecommand{\norm}[1]{\lVert#1\rVert}
\providecommand{\bignorm}[1]{\bigl\lVert#1\bigr\rVert}
\providecommand{\Bignorm}[1]{\Bigl\lVert#1\Bigr\rVert}
\providecommand{\biggnorm}[1]{\biggl\lVert#1\biggr\rVert}

\providecommand{\ceil}[1]{\lceil#1\rceil}

\newcommand{\UU}{\mathbb{U}}

\def\bu{{\bf u}}
\def\bU{{\bf U}}

\newcommand{\bv}{{\bf v}}
\newcommand{\bw}{{\bf w}}

\def\e2{\spl{2}(\nabla^d)}
\def\cA{{\mathcal{A}}}
\def\ga{\gamma}
\def\garatio{{\rho_\ga}}

\newcommand{\As}{{\Acal^s}}

\newcommand{\AH}[1]{{\Acal_\Hcal({#1})}}

\newcommand{\kk}[1]{{\mathsf{#1}}}
\newcommand{\rr}[1]{{\mathsf{#1}}}

\newcommand{\KK}[1]{{\mathsf{K}_{#1}}}

\newcommand{\Psvd}[2]{\operatorname{P}_{\UU({#1}),{#2}}}
\newcommand{\hatPsvd}[1]{\operatorname{\hat P}_{#1}}

\newcommand{\rsvd}{\operatorname{r}}
\newcommand{\Cctr}{\operatorname{C}}
\newcommand{\hatCctr}[1]{\operatorname{\hat C}_{#1}}
\newcommand{\Restr}[1]{\operatorname{R}_{#1}} 

\newcommand{\constsvd}{\kappa_{\rm P}}
\newcommand{\constcrs}{\kappa_{\rm C}}

\newcommand\eref[1]{(\ref{#1})}
\newcommand{\beqn}{\begin{equation}}
\newcommand{\eeqn}{\end{equation}}

\newcommand{\bA}{\mathbf{A}}
\newcommand{\bT}{\mathbf{T}}
\newcommand{\bbf}{\mathbf{f}}
\newcommand{\bg}{\mathbf{g}}
\newcommand{\cL}{{\mathcal{L}}}
\newcommand{\cD}{{\mathcal{D}}}

\newcommand{\cN}{{\mathcal{N}}}
\newcommand{\bB}{\mathbf{B}}

\newcommand{\hatbS}{\mathbf{\hat S}}
\newcommand{\bS}{\mathbf{S}}

\newcommand{\tbS}{{\tilde{\mathbf{S}}}}
\newcommand{\tbT}{{\tilde{\mathbf{T}}}}

\newcommand{\hdimtree}[1]{\mathcal{D}_{#1}}
\newcommand{\hroot}[1]{{0_{#1}}}
\newcommand{\leaf}[1]{\cL(\hdimtree{#1})}
\newcommand{\nonleaf}[1]{\cN(\hdimtree{#1})}

\newcommand{\leftchild}{{{\rm c}_1}}
\newcommand{\rightchild}{{{\rm c}_2}}
\newcommand{\child}[1]{{{\rm c}_{#1}}}
\newcommand{\hsum}[1]{\mathrm{\Sigma}_{#1}}

\newcommand{\dd}{{\rm rank}}

\newcommand{\om}[1]{{\omega_{#1}}}
\newcommand{\omi}[2]{{\hat\omega_{#1,#2}}}

\newcommand{\Sc}{{\bS}}

\newcommand{\Sci}[1]{{\hat\bS_{#1}}}
\newcommand{\Sr}{{\tbS}}
\newcommand{\Sa}[1]{\tbS_{#1}}

\title{Adaptive Low-Rank Methods: Problems on Sobolev Spaces\thanks{This work has been supported in
part by the DFG Special Priority Program 1324, by the DFG SFB-Transregio 40, by the DFG Research Group 1779,
the Excellence Initiative of the German Federal and State Governments (RWTH Aachen  Distinguished Professorship,
Graduate School AICES, GSC 111),
 and NSF grant DMS 1222390.}
}

\author{Markus Bachmayr and Wolfgang Dahmen}
\date{\today}

\begin{document}

\maketitle

\begin{abstract}
This paper is concerned with the development and analysis of an iterative solver for high-dimensional second-order elliptic problems 
based on subspace-based low-rank tensor formats. Both the subspaces giving rise to low-rank approximations and
corresponding sparse approximations of lower-dimensional tensor components are determined adaptively. A principal obstruction to
a simultaneous control of rank growth and accuracy turns out to be the fact that the underlying elliptic operator is an isomorphism 
only between spaces that are not endowed with cross norms. Therefore,
as central part of this scheme, we devise a method for preconditioning low-rank tensor representations of operators.
Under standard assumptions on the data, we establish convergence to the solution of the continuous problem with a guaranteed
error reduction. Moreover, for the case that the solution exhibits a certain low-rank structure and representation sparsity,
we derive bounds on the computational complexity, including in particular bounds on the tensor ranks that can arise during the iteration. We emphasize that
such assumptions on the solution do not enter in the formulation of the scheme, which in fact is shown to detect them
automatically.  Our findings are illustrated by numerical experiments that demonstrate the practical efficiency of the method in high 
spatial dimensions.


\end{abstract}

\section{Introduction}\label{sec:intro}

The approximate solution  of {\em high-dimensional linear diffusion problems} is not only of intrinsic interest, but occurs also frequently  as a subproblem in solvers for other classes of high-dimensional problems, e.g. via operator splitting.
Written as operator equations, such diffusion problems are of the form
\begin{equation}
\label{1.1}
A u = f,
\end{equation}
where the exact solution $u$ belongs to  some {\em energy space} $V$, comprised of functions of $d \gg 1$ variables, and $f$ is a given 
element in the normed dual $V'$ of $V$. 
A basic model problem of this type is the high-dimensional Poisson problem with $A= -\Delta$ and $V = H^1_0((0,1)^d)$.

Such spatially high-dimensional problems have been investigated in different communities from rather different perspectives. 
One can roughly distinguish the following groups:

(a) A rich theoretical foundation exists for methods based on variants of hyperbolic cross approximations and sparse grids, where approximability can indeed be directly related to the regularity of certain high-order mixed derivatives. Rigorous adaptive methods for this type of approximations are available, for instance the one proposed in \cite{Dijkema:09}.  However, such approaches turn out to be feasible only for moderate values of $d$.

(b) Very promising concepts of low-rank tensor approximation have been developed, for instance, in the works \cite{Andreev:12,Ballani:13,Beylkin:02,Beylkin:05-1,Hackbusch:05,Khoromskij:11,Khoromskij:11-1,Kressner:11}. These tools have been successfully applied in high-dimensional regimes. However, to our knowledge,
rigorous error and complexity bounds for relevant norms are not yet available in this context.

(c) The intrinsic {\em tractability} of high-dimensional diffusion problems has been addressed from the viewpoint of {\em Information Based
Complexity}, see  \cite{WW:14} and the literature cited there.
The central issue there is to determine under which circumstances  the {\em curse of dimensionality} can be 
broken, that is, whether one can find an algorithm whose complexity does {\em not} scale  exponentially in the spatial dimension $d$
when realizing a given target accuracy. In this latter case the problem is called {\em tractable}.
Favorable rigorous complexity bounds have been obtained for elliptic Neumann problems
under various assumptions on the right hand side which constrain the dependence on the different variables and ensure the availability of simple (diagonal) solution operators.  However,  it is not clear how to translate these findings into a realistic
computational scenario.

The present paper is an attempt to offer a synthesis between (a), (b) and (c). However, we emphasize from the start
that, in contrast to (c), our focus is on the the complexity of  the {\em inversion process}---diagonal operator representations {\em not} being available---to find approximations to the solution $u$, {\em given} appropriate approximations to the data $f$. The rationale is that
even for the simplest type of data, such as a constant function $f$, the inversion is completely infeasible for increasing $d$ when using standard
techniques under realistic regularity assumptions.

The approaches listed under (b) can be viewed as seeking  suitable {\em solution-dependent} but computationally accessible bases, with respect to which
the solution permits good approximations with relatively few terms. The identification of such bases becomes then part of the solution
process and the resulting parametrizations of approximate solutions are highly nonlinear, much more so than, for instance, 
best $n$-term approximations with respect to an {\em a priori} given fixed {\em background basis} as in (a).

The rationale in (b) as well as in the present work  for employing dictionaries with tensor structure   is that
the Laplacian is a sum of rank-one operators and   the problem  is formulated
on a product domain. Thus one hopes that functions with tensor structure can best exploit structural properties of $u$, while
separation of variables is known to help in computationally dealing with a large number of variables.
The adaptive method we put forward in this work iteratively finds  basis functions with tensor structure that are adapted to the approximand $u$. In the simplest case $d=2$, for instance, the algorithm yields univariate basis functions $U^{(1)}_k$, $U^{(2)}_k$ and coefficients $a_k$ such that
\begin{equation}\label{simpleexample} 
 u(x_1,x_2) \approx \sum_{k=1}^r a_k\, U^{(1)}_k(x_1) \,U^{(2)}_k(x_2)   
\end{equation}
where the value of $r$ is near-minimal---in a sense to be made precise later---for achieving a certain error tolerance in the $V$-norm by a tensor expansion of this form.
To achieve a similar result for large $d$, we build on recent progress in high-dimensional tensor representations, and find approximations in the \emph{hierarchical tensor format} \cite{Hackbusch:09-1}.
The iterative scheme used to find these approximations is based on a perturbed Richardson iteration that works directly on the continuous problem, but approximates all quantities by finite approximations with suitable error tolerances.
Our objective is to control the solution error in an appropriate norm---here, the $V$-norm---and at the same time to control the complexity of the complete numerical scheme.

\subsection{State of the Art and Main Obstructions}
A first question is why one would expect a substantial gain in making the additional effort of finding, as part of the solution process, a suitable dictionary for representing approximations. Indeed, many well-studied techniques for approximating  high-dimensional functions rely on {\em sparsity} with respect to a judiciously chosen but \emph{fixed} tensor product background basis for  the spatially high-dimensional space. However, under realistic 
assumptions the resulting methods usually cannot avoid an exponential scaling of the computational complexity in $d$.
For instance, the adaptive solver for certain problems of the type \eqref{1.1} constructed in \cite{Dijkema:09} builds on anisotropic tensor product wavelet bases, and is shown to have optimal complexity (also with respect to its $d$-dependence) in relation to the corresponding best $n$-term approximation of $u$. But, as the results for the Poisson problem given there demonstrate, even the best $n$-term approximations in such bases become infeasible in high dimensions. This  indicates that, in order to arrive at a feasible scheme under realistic regularity assumptions, one has to give up on
$n$-term approximations in terms of {\em  fixed} background bases and needs to modify the type of approximation.

As mentioned earlier, this is indeed the common theme in the works grouped under (b) above. However, an essential distinction
from the present work is that---except for \cite{BD}---all methods known to us require as a first step the a priori choice of a {\em fixed} discretization of the
continuous problem, and subsequently aim at solving this discrete problem approximately in an efficient way using tensor formats for high-dimensional Euclidean spaces. In many cases of interest, e.g.\ for the Laplacian, the corresponding discretizations of the underlying operator have simple explicit representations in such tensor formats. However, to motivate the subsequent developments, it is important to understand the shortcomings of such a strategy. 

First, accuracy considerations are detached from the 
underlying continuous problem. In fact, since accuracy is measured in terms of the Euclidean norm of discretization coefficients, it is unclear what this means
for the computed approximation in a function space norm such as the {\em energy norm}. Second, since the resolution is fixed for
each variable, even if the discretized problem was solved exactly, the spatial resolution of the tensor factors may
be insufficient for warranting a desired target accuracy. Furthermore, in the case of non-zero order operators such as the Laplacian, this cannot be controlled by a posteriori error indicators: due to the mapping properties of such operators, Euclidean residuals do not faithfully reflect solution accuracy. Moreover, refinement of the discretization renders the discrete problem more and more ill-conditioned. 

This also becomes apparent in the upper bounds for tensor approximation ranks for solutions of linear systems obtained in \cite{KressnerUschmajew}. These are applicable, in particular, to \emph{discretizations} of second-order elliptic operators, but \emph{not} to the corresponding continuous problems: although the bounds depend only weakly on $d$, they may grow strongly with discretization refinement due to the influence of condition numbers. Since this leads to gross overestimates of the increase of ranks relative to the total solution error (compared e.g.\ to the numerical results in Section \ref{sec:num-res}), this underscores the necessity of \emph{preconditioning} in the context of low-rank approximations.

Preconditioning means to approximate the inverse as a mapping from
$V'$ to $V$. Unfortunately, when $A$ has non-zero order neither $V$ nor $V'$ are endowed with \emph{cross norms}, that is, norms with the property that the norm of a rank-one function equals the product of the norms of the lower-dimensional factors.
As a mapping between such spaces $V$, $V'$ without simple tensor product structure, the inverse of $A$ has {\em infinite rank}, which intrinsically obstructs the control of rank growth 
when increasing accuracy.  As illustrated in Section \ref{ssec:example}, this is an inherent consequence of the spectral properties of such elliptic operators.

In the method studied in \cite{BNZ:14}, this problem manifests itself in applying the inverse of a certain Riesz map. However, again only the case of $V$ endowed with a cross norm, where both this Riesz map and its inverse are of rank one, is considered in detail. Although in other works, preconditioners for low-rank tensor methods for second-order problems have been proposed, e.g.\ in \cite{Khoromskij:09,Kressner:11,Ballani:13,Andreev:12}, these have not been analyzed in their overall effect on the complexity of the solution process. The central objective of the present work is to
put forward several new conceptual ingredients to address these intrinsic obstructions.  

\subsection{New Conceptual Ingredients}
To overcome the above obstructions one has to account for the following points.
First, to be able to achieve arbitrarily good approximations to the solution of the continuous problem, one has to intertwine finding good low-rank approximations with finding sufficiently accurate basis expansions for lower-dimensional tensor components. In the example \eqref{simpleexample} for $d=2$ this means to keep, for a given target accuracy $\varepsilon$,
the rank $r=r(\varepsilon)$ as small as possible, while the  involved low-dimensional tensor factors $U^{(i)}_k(x_i)$ need to be resolved with an increasingly better accuracy as well.
Second,  to properly balance both levels of approximation as well as monitor the deviation from the continuous solution,
we need to relate solution errors to {\em residuals}.
This inevitably requires taking into account that the operator $A$ is an isomorphism from $V$ onto its normed dual $V'$. 
Third, we need to use tensor formats with similar stability properties as the singular value decomposition, while respecting the norms imposed on us by the spaces $V$ and $V'$.

This has led to the framework proposed in \cite{BD}. With the aid of a suitable {\em background basis} such as a tensor product
wavelet basis on $\Omega =(0,1)^d$ the problem \eqref{1.1} is transformed into an {\em equivalent} problem on the infinite dimensional
sequence space $\ell_2(\nabla^d)$ with entries indexed by elements of the Cartesian product $\nabla^d$ of low-dimensional
wavelet index sets.  Hence, sequences can be viewed as {\em tensors of order $d$}, and the spectral theorem allows one to carry over
the results on stable tensor formats to $\ell_2(\nabla^d)$. Moreover, when $A$ is a zero-order operator or when $A$ acts on only 
a fixed small number of variables as an operator of nonzero order, as in the case of parametric PDEs, suitable spaces $V$
are tensor product Hilbert spaces with tensor product Riesz bases. 
As a consequence,  the wavelet representation $\bA$ still has low rank and the transformed problem is well-conditioned on
$\ell_2(\nabla^d)$, so that solution errors indeed become equivalent to residuals. It is shown in \cite{BD} how to formulate under these circumstances
an iterative scheme that approximates the true solution with near-optimal complexity. Note that the resulting tensor expansions as in (\ref{simpleexample}) can then still be interpreted as an expansion with respect to a tensor product wavelet basis $\{ \psi^{(1)}_{\nu_1} \otimes \cdots \otimes \psi^{(d)}_{\nu_d} \}$, but whereas, for example in \cite{Dijkema:09}, the coefficients for such a basis are represented directly as a sparse vector, in our setting these coefficients are now in turn expanded into sums of tensor products of sparse vectors.

In the present work we build on the concepts in \cite{BD}, but focus on the essential obstructions encountered when $V$ and $V'$
are {\em not} endowed with cross norms. Specifically, we consider second order elliptic equations as a prototypical scenario, but
remark that the results carry over to more general situations of analogous nature. In accordance with the previously mentioned
problems with preconditioning discretizations of elliptic operators, the necessary rescaling of an $L_2$-orthonormal tensor product wavelet basis
for  the corresponding representation $\bA$ to be well conditioned on $\ell_2(\nabla^d)$ causes $\bA$ to have {\em infinite rank}.
A major contribution of this work is an {\em adaptive rescaling scheme} embedded in a perturbed
Richardson iteration that, depending on the current approximate solution, causes only a moderate controllable rank growth.
It is based on a refined result on the relative accuracy of exponential sum approximations  derived from sinc quadrature for the function $t\mapsto t^{-1/2}$.
In particular, using the mapping properties of $A$ in this manner allows us to adjust error tolerances for the iteration in such a way that tensor ranks---which have a strong impact on numerical efficiency---grow only gradually as the scheme progresses.
We eventually arrive at a solver that performs well also for large $d$, and---under model assumptions that hold, in particular, for the high-dimensional Poisson problem---can be proven to produce approximate solutions with an overall complexity that grows sub-exponentially in $d$. 
We invest a considerable effort in analyzing the influence of the spatial dimension $d$, %
and a number of resulting findings  are perhaps of interest in their own right.
Our numerical experiments for a high-dimensional Poisson problem show that the complexity of the method exhibits  in fact only a low-degree polynomial growth in $d$.

The proposed scheme and its analysis apply also to problems with a more general structure than such Poisson problems, e.g.\ to elliptic operators with non-diagonal diffusion matrices. 
Even when considering finite-dimensional discretized problems, in such cases methods based on approximating the inverse by exponential sums as in \cite{Grasedyck:04}
are not applicable, since the operator then no longer has a suitable structure. In fact, since the variables are now coupled more strongly,
one expects a somewhat stronger rank growth with increasing accuracy. We quantify this by some first experiments.

The paper is organized as follows. In Section \ref{sec:road} we sketch a road map for the subsequent developments and explain in more detail
the issue of the interaction of mapping properties on Sobolev spaces and low-rank structure. In Section \ref{sec:prereq}, for the convenience of the reader we collect
some prerequisites needed for the remainder of the paper. This includes a short introduction
to the hierarchical Tucker format and near-optimal recompression and coarsening concepts, which are crucial for the iterative scheme
outlined already in Section \ref{sec:road}. Section \ref{sec:apply} is devoted to the central task, namely the adaptive application
of rescaled low-rank operators.  A precise formulation of the adaptive solver is given in Section \ref{sec:adalg} along with
the main convergence and complexity results. This theorem is proved in Section \ref{sec:proofs}. We conclude with some
numerical experiments in Section \ref{sec:num-res}.

We shall use the notation $a\lesssim b$  to express that $a$ is bounded by a constant times $b$, where this constant is independent of any parameters $a$ and $b$ may depend on, unless such dependencies are explicitly stated; moreover, $a\sim b$ means that $a\lesssim b$ and $b\lesssim a$.
\section{The Road Map}\label{sec:road}
In this section, we give an overview of our basic strategy. To this end, we also recapitulate for the convenience of the reader a few relevant facts from \cite{BD}.
\subsection{An Equivalent $\ell_2$-Problem}\label{ssec:l2}

We consider an operator equation 
\begin{equation}
\label{opeq}
Au =f,
\end{equation}
where $A: V\to V'$ is an isomorphism of some Hilbert space $V$ onto its dual $V'$.
We shall always assume that we have a Gelfand triplet
\begin{equation*}
V\subset  H \equiv H' \subset V',
\end{equation*}
in the sense of dense continuous embeddings, where we assume that $H$ is a {\em  tensor product Hilbert space}, that is,
\begin{equation}
\label{eq:htensorspace}
 H = H_1 \otimes \cdots \otimes H_d ,\quad \| g_1\otimes \cdots \otimes g_d\|_H = \prod_{i=1}^d\|g_j\|_{H_i},
\end{equation}
with lower-dimensional Hilbert spaces $H_i$. 
In this paper we focus on the case 
$$ 
H = \spL{2}(\Omega) = \spL{2}(\Omega_1)\otimes \cdots\otimes \spL{2}(\Omega_d)\,, 
$$
i.e., for $\Omega_i \subseteq \R^{d_i}$, for some $d_i\in\N$, the high-dimensional domain $\Omega$ is
a product domain $\Omega := \Omega_1\times \cdots\times \Omega_d$ and $\spL{2}(\Omega)$ is a
tensor product Hilbert space. When $A$ stands for an elliptic   operator of non-zero order the corresponding
{\em energy space} $V$ is typically of the form $V \subseteq \spH{s}(\Omega)$, $s\neq 0$, where the case of a 
strict subspace is given when certain essential  homogeneous boundary conditions are imposed on the trial space.
Note that for $s>0$,
\begin{equation*}
V=\spH{s}(\Omega) = \bigcap_{i=1}^d \spL{2}(\Omega_1)\otimes\cdots  \otimes \spH{s}(\Omega_i)\otimes \cdots \otimes \spL{2}(\Omega_d)\,,
\end{equation*}
and the norm on $\spH{s}(\Omega)$ is {\em not} a cross norm in the sense of 
\eqref{eq:htensorspace}.

It is well-known that the numerical solution of discrete approximations to \eqref{opeq} is severly hampered by the fact
that $A$ as a mapping from $H$ to $H$ is {\em unbounded}, and {\em preconditioning} exploits that $A$ as a mapping from
$V$ to $V'$ is boundedly invertible. Much of what follows results from   the conflict:
\begin{quote}
The topologies for which $A$ has favorable mapping properties are not ``tensor-friendly'';\\
for those topologies for which $A$ has a ``tensor-friendly'' structure, it has unfavorable mapping properties.
\end{quote}
In one way or the other one has to pay for this conflict. In \cite{BD} we have chosen to work
in topologies for which $A$ becomes an isomorphism, since this seems to be the only way to warrant a rigorous error analysis.

To implement this strategy our basic assumption is that we have Riesz bases for each component Hilbert space $H_i=\spL{2}(\Omega_i)$ (see \eqref{eq:htensorspace}), which we denote by $\{ \psi^{H_i}_\nu\}_{\nu\in\nabla^{H_i}}$.
We may assume without loss of generality that all $\nabla^{H_i}$ are identical, denoted by $\nabla$.
To simplify our discussion, we shall always call $d$ the spatial dimension, which amounts to the assumption that $d_i = 1$ for $i=1,\ldots,d$; indeed, everything that follows is applicable also in the case that the actual spatial dimension $d_1 + \ldots+d_d$ of $\Omega$ is larger than the tensor order $d$, but we will make only the dependence on $d$ explicit.

In principle, regardless of the structure of $\nabla$, one can transform \eref{opeq} into the equivalent 
infinite dimensional system
\beqn
\label{bT}
\bT \bu^\circ = \bg,\quad \mbox{where}\quad \bT = \big(\langle \Psi_\nu,A\Psi_\mu\rangle\big)_{\nu\in\nabla^d},\,\,
\bg := \big(
\langle \Psi_\nu,f\rangle\big)_{\nu\in\nabla^d},
\eeqn
where $\bu^\circ = \big(\langle \Psi_\nu,u\rangle\big)_{\nu\in\nabla^d}$ is the coefficient sequence of the solution
$u$ with respect to $\Psi$. Note that for $s>0$, the operator $\bT$ is unbounded. However, when the low-dimensional
basis functions 
  $\psi^{H_i}_\nu$ are chosen to be sufficiently regular wavelets, the {\em infinite-dimensional} operator  \eref{bT} can be conveniently {\em
  preconditioned}. In this case, one can specify the structure of $\nabla$   and for our purposes 
  it suffices to know that each $\nu =(j,k)$ encodes a dyadic level $j=\abs{\nu}$
and a spatial index $k=k(\nu)$.
The crucial point is that when $V=\spH{s}(\Omega)$ is a Sobolev space,
a simple {\em rescaling} of $\Psi_\nu := \psi^{H_1}_{\nu_1} \otimes \cdots\otimes \psi^{H_d}_{\nu_d}$ 
by a sequence $\{ \omega_\nu \}$ with $\omega_\nu \sim \norm{\Psi_\nu}_{V}$ yields a Riesz basis $\{  \omega_\nu^{-1} \Psi_\nu\}$ for $V\subseteq H$ as well. 

This will now be explained in more detail in the case $s=1$, which corresponds to second-order elliptic problems, and which is the main focus of this work. Furthermore, we shall assume from now on  that  $\{ \Psi_\nu \}_{\nu\in\nabla^d}$ is actually  an \emph{orthonormal} tensor product wavelet basis  of
$\spL{2}(\Omega)$ with $\Psi_\nu \in \spH{s}(\Omega)$ for some $s>1$. It is known that, as a consequence,  the family of rescaled basis functions 
\begin{equation*}
\biggl\{  \Bigl( \sum_{i=1}^d 2^{2 \abs{\nu_i}} \Bigr)^{-\frac{1}{2}} \Psi_\nu \biggr\}_{\nu\in\nabla^d}
\end{equation*} 
forms a Riesz basis of $\spH{1}(\Omega)$ with {\em dimension-independent} condition number \cite{Dijkema:09}.
What matters here are not the specific values appearing in the above scaling weights---slightly different scaling weights 
with a comparable asymptotic behavior would serve the same purpose---but their structure as 
the Euclidean norm of a vector
\beqn
\label{can-scaling}
\om{\nu} = \om{\nu_1,\ldots,\nu_d} = \Big(\sum_{i=1}^d (\omi{i}{\nu_i})^2\Big)^{1/2}.
\eeqn
We  refer to the corresponding {\em scaling operator} 
\beqn
\label{bS}
\Sc = \big(\om{\nu}\delta_{\nu,\mu}\big)_{\nu,\mu\in\nabla^d} \,,
\eeqn
 with $\om{\nu}$ given by \eref{can-scaling}, and where $\omi{i}{\nu_i}$ are chosen such that \begin{equation}\label{omiscale}
 \omi{i}{\nu_i} \sim 2^{\abs{\nu_i}}
 \end{equation}
 with uniform constants, as the {\em canonical scaling}.
In these terms the system \eref{bT} is equivalent to the preconditioned system
\begin{equation}
\label{eqsystem}
\bA_c \bu_c = \bbf_c, \quad  \bA_c := \Sc^{-1}\bT\Sc^{-1},\,\, \bbf_c := \Sc^{-1}\bg,\,\, \bu_c = \Sc \bu^\circ, %
\end{equation}
see e.g. \cite{actanum}.
Now we have
\beqn
\label{condition}
c \|\bv\|\leq \|\bA_c \bv\| \leq C\|\bv\|,\quad \bv\in\ell_2(\nabla^d),
\eeqn
where here and below we write for simplicity $\norm{\bv} = \norm{\bv}_{\spl{2}(\nabla^d)}=\big(\sum_{\nu\in\nabla^d}\abs{v_\nu}^2\big)^{1/2}$.
The constants $c =c(A,\Psi), C=C(A,\Psi)$ thus give an estimate $C/c$ for the condition number of $\operatorname{cond}_2(\bA_c)$.

While the canonical scaling $\Sc$ with appropriately chosen $\omi{i}{\nu_i}$ can ensure a favorable conditioning, which is addressed in more detail in Section \ref{ssec:problemclass}, we shall see that the structure \eref{can-scaling} is unfavorable concerning the control of ranks. It will therefore be important to exploit some flexibility in choosing the scaling by using substitute scaling operators
$\Sr = {\rm diag}(\tilde\omega_\nu)$, which are {\em equivalent} to the canonical scaling $\Sc$ in the sense that
\beqn
\label{equivS}
\|\Sc\Sr^{-1}\|\sim 1
\eeqn
with constants independent of $d$, but for which the yet equivalent system
\beqn
\label{final}
\bA \bu = \bbf,\quad \bA = \Sr^{-1}\bT \Sr^{-1},\,\, \bbf = \Sr^{-1}\bg,
\eeqn
while still well-conditioned, offers a better angle at controlling ranks. 

Clearly, finding the coefficient sequence $\bu$ in \eref{final} (for any $\Sr$ satisfying \eref{equivS} of our choice) is equivalent
to finding the solution $u$ of \eref{opeq}, and the algorithm put forward below aims at
solving the variant \eref{final} for a suitable $\Sr$.
This in turn will be based on the fact that in the transformed version \eref{final} or \eref{eqsystem}, due to \eref{condition}, errors and residuals 
are comparable
with respect to the {\em same norm}, that is,
\begin{equation*}
\| u - v\|_V \sim \|\bu-\bv\| \sim \|\bbf -\bA \bv\| \sim \|f- Av\|_{V'},\quad v\in V,
\end{equation*}
and for a suitable damping factor $\omega$, depending on $C/c$, 
the iteration
\beqn
\label{idealiter}
\bu_{k+1} = \bu_k + \omega (\bbf - \bA \bu_k),\quad k=0,1,2,\ldots
\eeqn
converges with a fixed error reduction per step, i.e., $\|\bu_{k+1}-\bu\|\leq \rho \|\bu_{k}-\bu\|$
holds for some fixed $\rho <1$, see \cite{Cohen:02}.

Note that it would be highly desirable to keep $\rho$, that is the error reduction, independent of $d$ which requires that ${\rm cond}_2(\bA)$
be independent of $d$. We will take this up again below in Section \ref{ssec:problemclass}.

Rather than exploiting this fixed error reduction by devising perturbed iterations in such a way that the iterates essentially match the rates of best $N$-term approximations 
with respect to the given background basis $\Psi$ (see e.g. \cite{Cohen:02,Dijkema:09}), we follow the approach in \cite{BD} 
which also
uses a perturbed version of the ideal iteration \eref{idealiter} but aims at generating approximations of {\em low ranks} in 
a stable tensor format where the tensors are not taken from a given dictionary but are {\em solution dependent} and
have to be found during the solution process. To this end, following \cite{BD}, we view each entry $u_\nu=u_{\nu_1,\ldots,\nu_d}$ of the coefficient sequence $\bu$ as the
entry of a tensor of order $d$. The perturbed iteration then takes the form
 \beqn
 \label{practicaliter}
 \bu_{k+1} = {\rm C}_{\varepsilon_2(k)}\big({\rm P}_{\varepsilon_1(k)}(\bu_k + \omega
 (\bbf - \bA \bu_k))\big), \quad k=0,1,2,\ldots,
 \eeqn
 where ${\rm P}_{\varepsilon_1(k) }$, ${\rm C}_{\varepsilon_2(k) }$ are
 certain {\em reduction} operators and the $\varepsilon_i(k)$, $i=1,2$, are
 suitable tolerances which decrease for increasing $k$ so as to still guarantee the convergence of the iterates in $\ell_2$.  
 
For such an iteration to produce low-rank approximants, it is of course important that the (approximate) application of $\bA$ does not increase the
ranks of $\bu_k$ too strongly. As we will explain next, it is this point where a price has to be paid for the discretization-independent convergence and rigorous error control ensured by preconditioning. Although we consider this directly for the continuous problem, analogous effects can be observed with fixed discretizations and different types of preconditioning, see \cite{Andreev:12}.

\subsection{A Scaling Trap}\label{ssec:scaling}
As a guiding example consider  $\Omega:=(0,1)^d$, $H= \spL{2}(\Omega)$, $V= \spH{1}_0(\Omega)$ and
\begin{equation}\label{eq:example_operator}
 A \colon \spH{1}_0(\Omega) \to \spH{-1}(\Omega)\,,\quad u \mapsto -\sum_{i,j=1}^d a_{ij} \partial_i\partial_j u \,,
\end{equation}
where $(a_{ij})\in \R^{d\times d}$ is symmetric positive definite; hence, $A$ is a symmetric elliptic operator.
In order to avoid adding another layer of technicality we assume for simplicity that the coefficients $a_{ij}$ in the diffusion matrix
are constants. Hence, its conservative representation $Au=-{\rm div}(a\nabla u)$, which is used in the weak formulation below
involves the same coefficients. Also, all subsequent results carry over to sufficiently smooth {\em variable} but separable coefficients 
$a_{ij}(x)=a_i(x_i)a_j(x_j)$.

The operator has a \emph{low-rank structure}, i.e., it is a relatively short sum
of tensor product operators. This is inherited
by its representation with respect to an $\spL{2}$-orthonormal basis $\Psi$ comprised of {\em separable} functions, i.e., of
rank-one tensors. For 
 $\bT$ given by \eref{bT}, %
 one obtains
\begin{equation} 
\label{eq:unrescaled_tuckersum}
 \bT = \sum_{1\leq n_1,\ldots,n_d\leq R} c_{n_1,\ldots,n_d} \bigotimes_i \bT^{(i)}_{n_i} \,,
\end{equation}
with a certain rank parameter $R$. In fact, 
  in this case we have
\begin{align} 
\label{T1}
  \bT^{(i)}_1 := \bT_1 &= \bigl( \langle \psi_\nu, \psi_\mu \rangle \bigr)_{\mu,\nu\in\nabla} = \id\,, &  
 \bT^{(i)}_2 := \bT_2 &:= \bigl( \langle \psi'_\nu, \psi'_\mu \rangle \bigr)_{\mu,\nu\in\nabla}  \,,\\
 \label{T2}
 \bT^{(i)}_3 := \bT_3 &:= \bigl( \langle \psi'_\nu, \psi_\mu \rangle \bigr)_{\mu,\nu\in\nabla}\,, &
 \bT^{(i)}_4 := \bT_4 &:= \bigl( \langle \psi_\nu, \psi'_\mu \rangle \bigr)_{\mu,\nu\in\nabla} =-\bT_3^*
\end{align}
i.e., $R=4$, where the  coefficients $c_{n_1,\ldots,n_d}$ are given by
\begin{gather} 
 c_{2,1,\ldots,1} = a_{11}, \; c_{1, 2,1,\ldots,1} = a_{22}, \;\ldots,\; c_{1,\ldots,1,2} = a_{dd} \,,\nonumber  \\
c_{3,4,1,\ldots,1} = c_{4,3,1,\ldots,1} = a_{12} ,\;\ldots,\; 
     c_{1,\ldots,1,3,4} = c_{1,\ldots,1,4,3} = a_{d-1,d} \nonumber\\
c_{3,1,4,1,\ldots,1} = c_{4,1,3,1,\ldots,1} = a_{13} ,\;\ldots,\; 
    c_{1,\ldots,3,1,4} = c_{1,\ldots,4,1,3} = a_{d-2,d} \,, \label{coefficients}  \\[3pt]
\ldots, \nonumber \\[3pt]
\ldots, \, c_{3,1\ldots,1,4} = c_{4,1\ldots,1,3} = a_{1d}\,,\nonumber
\end{gather}
and $c_\kk{n} = 0$ for all further $\kk{n}\in \N^d$. 
We use in what follows  for multiindices in $\N^t_0$, $t\in\N$, 
the notational convention $\kk{k} = (k_1,\ldots, k_t)$,
$\kk{n} = (n_1,\ldots,n_t)$, $\rr{r} = (r_1,\ldots,r_t)$, and so forth, and for convenience define
$$  \kk{R} := (R,\ldots, R) \in \N^d \,. $$
Moreover, defining for a given $\rr{r} \in\N_0^d$
\begin{equation*}
  \KK{d}(\rr{r}) := \left\{
  \begin{array}{ll}
   \bigtimes_{i=1}^d \zinterval{1}{{r}_i} & \text{if $\min \rr{r} >
  0$,}\\
\emptyset & \text{if $\min \rr{r} =
  0$} \,,
 \end{array} \right.
\end{equation*}
we see that the minimal value of 
 $R \in \N$   such that $c_\kk{n} = 0$ if $\kk{n} \notin \KK{d}(\kk{R})$ is in the above case $R=4$ in general, or $R=2$ when the matrix of diffusion coefficients is
 diagonal.  

Hence, applying $\bT$ to a rank-one tensor  $\bv = \bv_1\otimes \cdots \otimes \bv_d$ gives rise to a sequence
$$
\bT \bv = \sum_{\kk{n}\in \KK{d}(\kk{R})} c_\kk{n} \bigotimes_i \bT^{(i)}_{n_i} \bv_i
$$
which has {\em Tucker} or {\em multilinear rank} $\kk{R}$, see below for general definitions. This fact is also  heavily used in all previously known tensor methods
for discretized operator equations. 

However, as mentioned before, $\bT$ is an unbounded operator and its preconditioned version
$\bA$ is used in the iterations \eref{idealiter} and \eref{practicaliter}. Whether employing the canonical scaling
from \eref{can-scaling} or any other equivalent one (in the sense of \eref{equivS}), the scaling weights
are {\em not separable}, reflecting the fact that neither $V$ nor its dual $V'$ are endowed with tensor product norms.
\begin{remark}
\label{rem:infinite}
While $\bT$ has low rank in the sense of \eref{eq:unrescaled_tuckersum}, the rank of $\bA$ is
infinite.
\end{remark}
Hence, each application of $\bA$ in \eref{practicaliter} yields a tensor of infinite rank, again in a sense to be made precise below.
It is therefore a pivotal issue of this paper to develop and analyze low-rank approximations to $\bA$ that remain well-conditioned.
This is why finding a suitable substitute $\Sr$ for the canonical scaling $\Sc$ is crucial.

\subsection{A Simple Example}\label{ssec:example}
One might think that the pitfall expressed by Remark \ref{rem:infinite} is a particular feature of the background wavelet basis.
The following simple example shows that this is not the case, but that the problem is rather a direct consequence of the spectral properties 
of $A$. To see this,  note that for $D^2 : C^2(0,1) \to C(0,1)$ defined by $D^2g(t)= g''(t)$, a complete $L_2$-orthonormal system of eigenfunctions
is given by $e_n(x)= c_0\sin (\pi n x)$, $n\in\N$, where $c_0 = \sqrt{\frac2\pi}$. The corresponding eigenvalues are given by $\lambda_n= (\pi n)^2$, $n\in \N$.
One easily checks that then the rank-one tensors $e_{\kk{n}}(x):= c_0^d \sin (\pi n_1x_1)\cdots \sin (\pi n_dx_d)$ form
a complete system of eigenfunctions of the Laplacian
$$
-\Delta = -\sum_{i=1}^d \id_{x_1} \otimes \cdots \otimes \id_{x_{i-1}}\otimes  D^2_{x_i} \otimes \id_{x_{i+1}} \otimes \cdots \otimes \id_{x_d}
$$
with eigenvalues $\lambda_\nu = \lambda_{\nu_1}+ \cdots + \lambda_{\nu_d}$, $\nu\in \N^d$. Representing $-\Delta$ with
respect to this basis yields 
$$
\bT := \big(\langle e_\nu, (-\Delta) e_\mu\rangle\big)_{\nu,\mu\in \N^d} = \big(\lambda_\nu \delta_{\nu,\mu}\big)_{\nu,\mu\in \N^d}\,.
$$
The ideal scaling matrix turning $\bT$ into an operator with bounded spectral condition, in this particular case into the identity, 
is in analogy to the previous considerations $\Sc := \big(\lambda^\frac12_\nu\delta_{\nu,\mu}\big)_{\nu,\mu\in\N^d}$, because then
$\Sc^{-1}\bT\Sc^{-1} = \id$.
Thus, we face the same problem: $\lambda_\nu^{-\frac12} =(\lambda_{\nu_1}+\cdots + \lambda_{\nu_d})^{-\frac12}$ as an inverse of the square root of
a sum is not separable. In fact, 
suppose that 
$$
f(x)=\bigotimes_{i=1}^d \Big(\sum_{\nu_i\in \Gamma_i} f_{i,\nu_i}c_0\sin(\pi \nu_i x_i)\Big)
$$ 
is a rank-one tensor where each tensor factor $f_i(x_i)= \sum_{\nu_i\in \Gamma_i} f_{i,\nu_i}c_0\sin(\pi \nu_i x_i)$ is a 
finite linear combination of one-dimensional eigenfunctions.   Clearly, the solution  
$u$ of $-\Delta u= f$ is given by  %
\begin{equation*}
u= \sum_{\nu\in \N^d}\lambda_\nu^{-1}\langle f, e_{\nu}\rangle e_\nu = \sum_{\nu\in \times_{i=1}^d\Gamma_i} \lambda_\nu^{-1}\,
\Big(\prod_{i=1}^d f_{i,\nu_i}\Big)\, e_\nu = \sum_{\nu\in \times_{i=1}^d\Gamma_i}  u_\nu \,( \lambda_\nu^{-1/2 } e_\nu),
\end{equation*}
where $u_\nu :=  \lambda_\nu^{-1/2}\prod_{i=1}^d f_{i,\nu_i}$.
Here we have scaled the coefficients $u_\nu$ such that approximating $u$ in $\spH{1}$ by a restriction of the above expansion to any
finite set $S\subset \times_{i=1}^d\Gamma_i$ amounts to approximating the array $(u_\nu)_{\nu}$ in $\ell_2$. Due to the multiplication by 
$\lambda_\nu^{-1/2}$ neither are the $u_\nu$ any longer separable, nor do the $\lambda_\nu^{-1/2 } e_\nu$ have rank one, and the actual rank of the order-$d$ tensor $(u_\nu)$ in general depends on
the highest frequencies occurring in the sets $\Gamma_i$. Thus, it is a priori not clear whether $u$ can be approximated well
by low-rank tensor expansions. With the present choice of eigenfunction basis, even the separable function $f\equiv 1$ would have an infinite expansion.

A central objective of the remainder of this paper is to quantitatively approximate rescaled operators of the form \eref{final}
by low-rank operators, which can then be incorporated in an adaptive iteration of the form \eref{practicaliter}.

\subsection{Problem Class and ``Excess Regularity''}\label{ssec:problemclass}
Throughout the remainder of the paper we confine the discussion to operators of the form \eref{eq:example_operator}, i.e.,
$V= \spH{1}_0(\Omega)$, $s=1$. Moreover, we require that the diffusion matrix $(a_{ij})$ be \emph{diagonally dominant} with uniformly bounded diagonal elements, that is,
\beqn
\label{aij}
\sum_{j\neq i}\abs{a_{ij}} \leq \abs{a_{ii}}\leq C,\quad i=1,\ldots, d,
\eeqn
with $C$ independent of $d$. 

We emphasize that the restriction to this problem class is made to keep the presentation accessible, but is not essential for the subsequent developments. As shown in \cite{B}, different operators, for instance Coulomb potentials, can also be treated in this framework, but since this leads to additional technicalities---particularly in the interaction with the rescaling operator $\Sc$---this is not addressed here.

In addition we make an assumption regarding some additional \emph{coordinatewise regularity}, which concerns $\bbf$ and the regularity of the $\Psi_\nu$. 
To formulate these, we need two types of additional scaling operators that act on single coordinates.

For $\omi{i}{\nu_i}$ as in \eref{can-scaling}, for $\tau\in\R$ and for $i=1,\ldots,d$, we define on the one hand the coordinatewise scaling operators 
$\Sc^{\tau}_{i}:  \R^{\nabla^d} \to \R^{\nabla^d}$ by
\beqn
\label{Si}
\Sc^{\tau}_{i} \bv := \bigl( \omi{i}{\nu_i}^\tau v_\nu \bigr)_{\nu\in \nabla^d}
\quad  \text{and}\quad \Sc_{i} := \Sc_{i}^1
\eeqn
and on the other hand, the corresponding \emph{low-dimensional} scaling operators $\Sci{i}^\tau \colon \R^\nabla \to \R^\nabla$ by
\begin{equation}\label{eq:tensor_rescaling}   
\Sci{i}^\tau \mathbf{\hat v} := \bigl( \omi{i}{\nu_i}^\tau \hat v_{\nu_i} \bigr)_{\nu_i\in\nabla} \quad  \text{and}\quad \Sci{i} := \Sci{i}^1  \,.
\end{equation}

We now assume that there exists a $t > 0$ such that for $i=1,\ldots,d$, the operators
\begin{equation}
\label{eq:coordwise_regularity}
   \Sci{i}^{-1+t} \bT_2 \Sci{i}^{-1-t},\quad \Sci{i}^{t} \bT_3 \Sci{i}^{-1-t} ,\quad 
   \Sci{i}^{-1+t} \bT_4 \Sci{i}^{-t} \,.
\end{equation}
map $\spl{2}(\nabla)$ boundedly to itself, and that
\begin{equation}
\label{eq:coordwise_reg_f}
\norm{\Sc^t \bbf}^2  =  \sum_{i=1}^d  \norm{\Sc^t_i \bbf }^2 < \infty .
\end{equation}
We shall refer in what follows to the above assumptions \eqref{eq:coordwise_regularity} and \eqref{eq:coordwise_reg_f} as {\em excess regularity assumptions of order $t>0$}.
Here $t$ can be arbitrarily small but fixed, and is only used in the complexity estimates but not required for the computation, so that these assumptions are not very impeding.

\begin{remark}
The condition \eqref{eq:coordwise_regularity} holds if the wavelets $\Psi_\nu$ are sufficiently regular to satisfy, after rescaling, a norm equivalence also for $\spH{1+t}(\Omega)$ and, by our orthonormality requirement, also for the same range of dual spaces. The condition \eqref{eq:coordwise_reg_f} then means that $f$ needs to have Sobolev regularity slightly higher than $\spH{-1}(\Omega)$.
\end{remark}

When trying to assess the computational complexity of methods based on \eqref{practicaliter} for problems of the form \eqref{eq:example_operator} with an
eye on the role of the spatial dimension $d$, one has to take into account the $d$-dependence of $\operatorname{cond}_2(\bA_c) = \norm{\bA_c}\norm{\bA_c^{-1}}$, where $\bA_c = \Sc^{-1} \bT \Sc^{-1}$. To this end, note first that
since $\{2^{-\abs{\nu}} \psi_\nu\colon \nu\in\nabla \}$ is a Riesz basis of $\spH{1}_0(0,1)$ and because of \eqref{omiscale},
for each $i$ there exist $\underline{\lambda}_1^{(i)}, \overline{\lambda}_1^{(i)} > 0$ such that
\begin{equation}\label{deriv1cond}
  \underline{\lambda}_1^{(i)} \norm{\Sc_i \bv}^2  \leq  
     \Bignorm{\sum_{\nu\in\nabla^d} v_\nu\, \partial_i \Psi_\nu }^2_{\spL{2}(\Omega)} 
      \leq \overline{\lambda}_1^{(i)} \norm{\Sc_i \bv}^2  \,.
\end{equation}
Moreover, by our assumptions, $\underline{\lambda}_1 := \min_i \underline{\lambda}_1^{(i)} $ and $\overline{\lambda}_1 := \max_i \overline{\lambda}_1^{(i)}$ are independent of $d$.

The proof of the following proposition, based on the arguments in \cite[Section 2]{Dijkema:09}, is given for the convenience 
of the reader in Appendix \ref{app:cond}.
\begin{proposition}
\label{prop:cond}
Let $\underline{\lambda}_a$   and   $\overline{\lambda}_a$ denote the smallest and largest  eigenvalue of $(a_{ij})$, respectively.
Then, one has
\beqn
\label{cond1}  
    \operatorname{cond}_2(\bA_c) \leq \frac{\overline{\lambda}_a  \overline{\lambda}_1 }{\underline{\lambda}_a  \underline{\lambda}_1} \,,
\eeqn
i.e., this condition number can depend on $d$ only via $\overline{\lambda}_a / \underline{\lambda}_a = \operatorname{cond}_2(a_{ij})$.
Moreover, when $(a_{ij})$ is diagonal the choice $\omi{i}{\nu_i} \sim \sqrt{a_{ii} } 2^{\abs{\nu_i}}$ for the scaling weights yields
\beqn
\label{cond2}
\operatorname{cond}_2(\bA_c) \leq \frac{\overline{\lambda}_1}{\underline{\lambda}_1},
\eeqn
regardless of $\operatorname{cond}_2(a_{ij})$.
\end{proposition}

Note that the $L_2$-orthonormality of $\{ \psi_\nu\}$ that we have assumed from the outset is a crucial requirement here, since otherwise the condition numbers in \eqref{cond1}, \eqref{cond2} would necessarily exhibit an exponential dependence on $d$.

Working now towards formulating a numerically implementable version of \eref{practicaliter} and analyzing its complexity
requires two further essential prerequisites: On the one hand, we need to fix the specific tensor formats to be used in such iterations.
More importantly, we need to specify the concrete form of the reduction operators in terms of tensor recompression and coarsening
and characterize their precise approximation properties.
Here we build on known results on tensor calculus from the literature 
(see  e.g.\ \cite{Lathauwer:00,Tucker:64,Falco:10,Hackbusch:09-1,Oseledets:09,Oseledets:11,Grasedyck:13,Grasedyck:10,Hackbusch:12}).
The relevant results on the analysis of the reduction operators, restated for convenience in the following section, are taken from \cite{BD}. 
On the other hand, we need to formulate a procedure for the 
approximate application of a suitably preconditioned version $\bA$ of the representation $\bT$.
This requires some essentially new ingredients, which will be developed in Section \ref{sec:apply}.

\section{Some Prerequisites}\label{sec:prereq}
For the convenience of the reader we recall first  some basic facts about tensor formats and fix related notation.
We then proceed with  the precise formulation
of  recompression and coarsening operators along with establishing their {\em near-optimality} in a sense to be made precise.
These results are taken from \cite{BD}.

\subsection{Tensor Formats}

As indicated before, we regard $\bu$ as a tensor of order $d$ on $\nabla^d =
\bigtimes_{i=1}^d \nabla$. 
We begin with considering tensor representations of the form
\begin{equation}\label{eq:tuckerd}
  \mathbf{u} =  \sum_{k_1=1}^{r_1} \cdots \sum_{k_d=1}^{r_d}
    a_{k_1,\ldots,k_d} \,\mathbf{U}^{(1)}_{k_1} \otimes \cdots \otimes
    \mathbf{U}^{(d)}_{k_d} \, \,.
\end{equation}
Here the order-$d$ tensor $\mathbf{a}= (a_{k_1,\ldots,k_d})_{1\leq k_i\leq r_i:i=1,\ldots,d}$ is referred
to as \emph{core tensor}. The matrix $\mathbf{U}^{(i)}= \big(\bU^{(i)}_{\nu_i,  k_i }\big)_{\nu_i\in \nabla^{d_i},1\leq k_i\leq r_i}$ with
column vectors $\mathbf{U}^{(i)}_k\in\spl{2}(\nabla^{d_i})$, $k = 1,\ldots,r_i$,
is called the $i$-th \emph{mode frame}, where we admit $r_i=\infty, i=1,\ldots,d$. 
When writing sometimes for convenience $(\bU^{(i)}_k)_{k\in\N}$, although the $\bU^{(i)}_k$
may be specified through \eqref{eq:tuckerd} only for $k\leq r_i$, it will always be understood to mean $\bU^{(i)}_k=0$, for $k> r_i$.
Note that in a representation of the form \eqref{eq:tuckerd},  by modifying $\mathbf{a}$
accordingly, one can always
orthogonalize the columns of $\mathbf{U}^{(i)}$ so  as to obtain
$\langle\mathbf{U}^{(i)}_k, \mathbf{U}^{(i)}_l\rangle = \delta_{kl}$, $i =
1,\ldots,d
$.
We refer to $\mathbf{U}^{(i)}$ with the latter property as
\emph{orthonormal mode frames}.

If $\mathbf{a}$ is represented directly by its entries, \eref{eq:tuckerd} corresponds to
the so-called \emph{Tucker format} \cite{Tucker:64,Tucker:66} or \emph{subspace representation}.
The {\em hierarchical Tucker} format \cite{Hackbusch:09-1}, as well
as the special case of the {\em tensor train} format \cite{Oseledets:11}, correspond
to representations in the form \eqref{eq:tuckerd} as well, but use a further structured representation for the core tensor
$\mathbf{a}$.
To this end, %
$\hdimtree{d}$ will always denote a fixed binary dimension tree of order $d$,
singletons $\{ i\} \in \hdimtree{d}$ are referred to as \emph{leaves}, $\hroot{d} := \{1,\ldots,d\}$ as \emph{root}, and elements
of $\Ical(\hdimtree{d}) := \hdimtree{d}\setminus\bigl\{\hroot{d}
,\{1\},\ldots,\{d\} \bigr\}$ as \emph{interior nodes}.
The set of leaves is denoted by $\leaf{d}$, where we
additionally set $\nonleaf{d} :=
\hdimtree{d}\setminus\mathcal{L}(\mathcal{D}_d) =
\Ical(\hdimtree{d})\cup\{ \hroot{d}\}$. 
The functions
$$
\child{i}: \cD_d\setminus \cL(\cD_d)\to \cD_d\setminus \{\hroot{d}\},\quad 
\child{i}(\alpha):= \alpha_i\,,\qquad i=1,2\,,
$$
produce the ``left'' and ``right'' children of a non-leaf node $\alpha \in \mathcal{N}(\cD_d)$.

For a family of matrices $\mathbf{B}^{(\alpha,k)}\in \spl{2}(\N\times\N)$ for $\alpha\in\nonleaf{d}$, $k\in\N$, we denote
by $\hsum{\hdimtree{d}}(\{ \mathbf{B}^{(\alpha,k)} \}) \in \spl{2}(\N^d)$ the corresponding core tensor $\mathbf{a}$ which is represented in hierarchical form by the $\mathbf{B}^{(\alpha,k)}$, or explicitly,
$$   
\mathbf{a}= \Bigl( \hsum{\hdimtree{d}}\bigl(\{
\mathbf{B}^{(\alpha,k)}\}\bigr) \Bigr)_{(k_\beta)_{\beta\in\mathcal{L}(\hdimtree{d})}}
\\
:= \sum_{(k_\gamma)_{\gamma\in 
    \Ical(\hdimtree{d})} }
       \prod_{\delta\in \nonleaf{d}}
      B^{(\delta,k_\delta)}_{(k_{\leftchild(\delta)},k_{\rightchild(\delta)})}
      \,.         
$$
Considering for each node $\alpha$ in the given (fixed) dimension tree the corresponding \emph{matricization}
$T^{(\alpha)}_\bu$, obtained by rearranging the entries of the tensor into an infinite matrix representation of a Hilbert-Schmidt operator using the indices in $\nabla^\alpha$ as row indices, the dimensions of the ranges of these
operators yield the \emph{hierarchical ranks}
$\dd_\alpha(\bu) := \dim \range T^{(\alpha)}_\bu$ for $\alpha \in \cD_d$. Except for $\alpha=\hroot{d}$, where we always have $\dd_{\hroot{d}}(\bu)=1$, these are collected in the \emph{hierarchical rank vector} $
\rank (\bu) =  {\rank_{\hdimtree{d}}(\bu)} := (\rank_\alpha (\bu) )_{\alpha\in\hdimtree{d}\setminus\{\hroot{d}\}} 
$
and give rise to the hierarchical tensor classes
\begin{equation*}
  \Hcal(\rr{r}) := \bigl\{ \mathbf{u} \in
  \spl{2}(\nabla^d) \colon  \rank_\alpha(\mathbf{u}) \leq {r}_\alpha \text{
  for all $\alpha \in \hdimtree{d}\setminus\{\hroot{d}\}$} \bigr\} \,.
\end{equation*}
In the case of singletons $\{i\}\in \hdimtree{d}$, we use the simplified notation $\rank_i (\bu) := \rank_{\{i\}}(\bu)$. We denote by $\mathcal{R} \subset (\N_0 \cup\{\infty\})^{\hdimtree{d}\setminus\{\hroot{d}\}}$ the set of hierarchical rank vectors for which $ \Hcal(\rr{r})$ is nonempty.

There is an analogous format for operators. In fact, \eref{eq:unrescaled_tuckersum} represents the second order operator
in \eref{eq:example_operator} in the Tucker format.  To apply such an operator efficiently to a tensor in hierarchical representation, we additionally need an analogous hierarchical structure for the core tensor $\mathbf{c}$ in the representation of the operator as in \eref{eq:unrescaled_tuckersum}, that is,
\begin{equation}
\label{eq:htucker_operator_core}
 \mathbf{c} = \hsum{\hdimtree{d}}\bigl(\{ \mathbf{C}^{(\alpha,\nu)}\colon \alpha\in
 \nonleaf{d},\, \nu=1,\ldots, R_\alpha\} \bigr)  
\end{equation}
for suitable $R_\alpha$.
We now give two examples  of such decompositions.
In both examples, we consider the linear dimension tree
\begin{equation}
\label{eq:lindimtree}
\hdimtree{d} = \bigl\{  \{ 1,\ldots,d\}, \{2,\ldots,d\}, \ldots, \{d-1,d\}, \{1\}, \ldots,\{d\}   \bigr\} \,.
\end{equation}
\begin{example}\label{ex:laplace}
When the diffusion matrix $(a_{i,j})_{i,j=1}^d$ in \eref{eq:example_operator}  is the identity matrix, i.e., 
the operator is the Laplacian, we obtain a hierarchical decomposition
with
$$  \mathbf{C}^{(\hroot{d},1)} = \begin{pmatrix} 0 & 1 \\ 1 & 0 \end{pmatrix}\,,\quad
  \mathbf{C}^{(\alpha,1)} = \begin{pmatrix} 1 & 0 \\ 0 & 0 \end{pmatrix}\,,\;
   \mathbf{C}^{(\alpha,2)} = \begin{pmatrix} 0 & 1 \\ 1 & 0 \end{pmatrix}\,,\; \alpha\in \Ical(\hdimtree{d})\,. 
$$
Thus, the Laplacian can be represented in hierarchical format with rank bounded by two for each node, which coincides with the value $R=2$ in \eqref{eq:unrescaled_tuckersum} for this case.
\end{example}

\begin{example}\label{ex:tridiag}
A slightly more involved example is the tridiagonal diffusion matrix with values $2$ on the main diagonal and $-1$ on the two secondary diagonals, where $R = 4$ in \eqref{eq:unrescaled_tuckersum} arising in kinetic models for dilute polymers, see \cite{BS:2007}. 
In this case, one has the following hierarchical decomposition: for the root node,
$$
   \mathbf{C}^{(\hroot{d},1)} = \Bigl( 2(\delta_{(i,j),(1,2)} + \delta_{(i,j),(2,1)}) - (\delta_{(i,j),(3,4)} + \delta_{(i,j),(4,3)} + \delta_{(i,j),(1,5)})  \Bigr)_{\substack{i=1,\ldots,4\\ j=1,\ldots,5}} \,.
$$
For each $\alpha \in \Ical(\hdimtree{d}) \setminus\{ \{d-1,d\} \}$, we have $ \mathbf{C}^{(\alpha,\nu)}\in\R^{4\times 5}$ for $\nu = 1,\ldots,5$, with values in $\{0,1\}$, where the value $1$ occurs at the following positions: entry $(1,1)$ of $\mathbf{C}^{(\alpha,1)}$, entries $(1,2)$, $(2,1)$ of $\mathbf{C}^{(\alpha,2)}$, $(3,1)$ of $\mathbf{C}^{(\alpha,3)}$, $(4,1)$ of $\mathbf{C}^{(\alpha,4)}$, and $(3,4)$, $(4,3)$, $(1,5)$ of $\mathbf{C}^{(\alpha,5)}$.
For $\alpha = \{d-1,d\}$, the matrices $\mathbf{C}^{(\{d-1,d\},\nu)}\in\R^{4\times 4}$ are defined in the same manner, but with each last column dropped\footnote{Note that for homogeneous Dirichlet boundary conditions, an additional simplification is possible, since then $\bT_3 \otimes \bT_4 = \bT_4\otimes \bT_3$.}.
\end{example}

As can be seen in the second example, the representation ranks $R_{\alpha}$ for interior nodes $\alpha \in \Ical(\hdimtree{d})$ may be larger than $R$. 
\subsection{Recompression, Contractions, and Coarsening}\label{ssec:RCC}
We proceed describing next the coarsening and recompression operators appearing in \eref{practicaliter}.
\paragraph{Near-optimal Recompression.}
Essential advantages offered by subspace based tensor formats like the hierarchical Tucker format
are that best approximations of given rank always exist, and  that {\em near-best approximations}
from the classes $\Hcal(\rr{r})$ are realized by truncation of a 
{\em hierarchical singular value decomposition} ($\Hcal$SVD), cf.\ \cite{Grasedyck:10}.
As in \cite{BD}, for a given $\bv\in \ell_2(\nabla^d)$ we denote by $\Psvd{\mathbf{v}}{\rr{r}} \mathbf{v}$ the result of truncating a $\Hcal$SVD of $\bv$ to ranks $\rr{r}$.

Moreover, we have computable error bounds $\lambda_{\rr{r}}(\bv)$ for this truncation. See \cite{Grasedyck:10} for a proof of the following result and \cite{BD} for a detailed discussion tailored to the present needs. 
\begin{remark}
\label{rem:near-best}
For any  rank vector $\rr{r}\leq \rank(\bv)$, $\rr{r}\in \Rank$, one has
\begin{equation*}
   \norm{\mathbf{v} -  \Psvd{\mathbf{v}}{\rr{r}} \mathbf{v}} 
\leq  {  \lambda_\rr{r} (\bv) }\leq
\constsvd \min_{\rank(\mathbf{w})\leq \rr{r}} \norm{\mathbf{u} - \mathbf{w}} ,   
 \quad \constsvd =\sqrt{2d-3}\,.
\end{equation*}
 \end{remark}
In order to quantify what we mean by {\em tensor sparsity}, for $r\in\N_0$ let
 $$
 {\sigma_{r}(\bv) = \sigma_{r,\Hcal}(\bv):= \inf\,\bigl\{\norm{ \bv- \bw} \,:\; \bw \in
\Hcal(\rr{r}) \text{ with $\rr{r}\in\Rank$, $\abs{\rr{r}}_\infty \leq r$} \}\,. }
$$
This allows us to consider corresponding {\em approximation classes}. 
To this end, giving 
  a positive, strictly
increasing {\em growth sequence} $\ga = \bigl(\ga(n)\bigr)_{n\in \N_0}$ with $\ga(0)=1$
and $\ga(n)\to\infty$, as $n\to\infty$,
we define
$$
\Acal(\ga)= \AH{\ga}:=  { \bigl\{\bv\in {\e2} : \sup_{r\in\N_0} %
\ga({r})\,
\sigma_{r,\Hcal}(\bv)=:\abs{\bv}_{\AH{\ga}} { <\infty }\bigr\} }
$$ 
and
$\norm{\bv}_{\AH{\ga}}:= \norm{\bv} + \abs{\bv}_{\AH{\ga}}$.
We call the growth sequence $\ga$ \emph{admissible} if
$$
\garatio:= \sup_{n\in\N} \ga(n)/\ga(n-1)<\infty \,,
$$
which corresponds to a restriction to at most exponential growth.

Rather than seeking (near-)best approximations for a given rank vector, we ask for approximations meeting a given target accuracy with (near-)minimal maximum ranks.
\begin{remark}
\label{rem:howtoread}
In this regard we have the following 
way of reading $\bv\in \AH{\ga}$ in mind: a given target accuracy $\varepsilon$ can
be realized at the expense of
ranks of the size $\gamma^{-1}(\abs{\bv}_{\AH{\ga}}/\varepsilon)$ so that a rank bound
of the form $\gamma^{-1}(C\abs{\bv}_{\AH{\ga}}/\varepsilon)$, where $C$ is a  constant, marks
a near-optimal performance. 
\end{remark}

Evaluating the bounds $ \lambda_\rr{r} (\bv)$ allows one to determine near-minimal ranks
$$
\rsvd(\bu,\eta) \in \argmin\bigl\{|\rr{r}|_\infty: \rr{r}\in \Rank,\;  {  \lambda_{\rsvd(\mathbf{u}, \eta)}(\bu) } \leq \eta\bigr\},
$$
that ensure the validity of  a given accuracy tolerance $\eta>0$. Given $\bv\in \ell_2(\nabla^d)$, this, in turn, gives rise to
a computable near-minimal rank approximation
\begin{equation*}
\hatPsvd{\eta} \bv:= \Psvd{\bv}{\rsvd(\bv,\eta)}\bv\,,
\end{equation*}
from $\Hcal(\rr{r})$. In fact,
we have by definition
\begin{equation*}
\norm{\bv- \hatPsvd{\eta}\bv}\leq  \lambda_{\rsvd(\bv,\eta)} (\bv) \leq \eta,\qquad
\abs{\rank(\hatPsvd{\eta}\bv)}_\infty = \abs{\rsvd(\bv,\eta)}_\infty.
\end{equation*}

\paragraph{Coarsening and Contractions.}
In addition to such a near-optimal tensor recompression operator we need in addition a mechanism to
approximate the columns in a given mode frame by finitely supported sequences, again in a way that
preserves a given accuracy tolerance. To this end, we define for any $\hat d \in \N$ and $\Lambda \subset \nabla^{\hat d}$
the restriction of a given $\bv\in \ell_2(\nabla^{\hat d})$ to the index set $\Lambda$ by
\begin{equation*}
  \Restr{\Lambda} \mathbf{v} := \mathbf{v} \odot \chi_\Lambda\,,\quad \mathbf{v}\in
  \spl{2}(\nabla^{\hat d}) \,,
\end{equation*} 
i.e., all entries with index $\nu\not\in \Lambda$ are replaced by zero. 
For each $N\in\N_0$, the errors of best $N$-{\em term approximation} are then given by
\begin{equation*}
  \sigma_N(\mathbf{v}) := \inf_{\substack{\Lambda\subset\nabla^{\hat d}\\
  \#\Lambda\leq N}} \norm{\mathbf{v} - \Restr{\Lambda} \mathbf{v}} \,.
\end{equation*}
The compressibility of $\bv$ can again be described through   {\em approximation classes}.
For $s >0$, we denote by $\As(\nabla^{\hat d})$ the set of
$\mathbf{v}\in\spl{2}(\nabla^{\hat d})$ such that
\begin{equation*}
  \norm{\mathbf{v}}_{\As(\nabla^{\hat d})} := \sup_{N\in\N_0} (N+1)^s
  \sigma_N(\mathbf{v}) < \infty \,.
\end{equation*}
Endowed with this (quasi-)norm, $\As(\nabla^{\hat d})$ becomes a (quasi-)Banach space.
When no confusion can arise, we shall suppress the index set dependence 
and write $\As = \As(\nabla^{\hat d})$.

The following concept, which allows us to relate a hidden low-dimensional sparsity 
of $\bv\in \ell_2(\nabla^d)$ to the {\em joint sparsity} of associated mode frames, was introduced first in \cite{B}, see also \cite{BD}.
To this end, for any vector $\kk{x} = (x_i)_{i=1,\ldots,d}$ 
and for $i\in\zinterval{1}{d}$, we employ the notation
\begin{equation}\label{eq:delentry}
  \kk{\check x}_i := (x_1, \ldots, x_{i-1},
 x_{i+1},\ldots,x_d)  \,, \quad
  \kk{\check x}_i|_y := (x_1, \ldots, x_{i-1}, y,
  x_{i+1},\ldots,x_d)
\end{equation}
to refer to the corresponding vector with entry $i$ deleted or
entry $i$ replaced by $y$, respectively.
 In a slight abuse of terminology we define for $\bu\in \ell_2(\nabla^d)$ and
for $i\in\zinterval{1}{d}$, using the notation \eqref{eq:delentry},
\begin{equation}
\label{contractions}
 \pi^{(i)}(\mathbf{u}) = \bigl( \pi^{(i)}_{\nu_i} (\mathbf{u})
 \bigr)_{\nu_i\in\nabla} 
  :=\biggl(  \Bigl(\sum_{\check{\nu}_i} \abs{u_{\nu}}^2 \Bigr)^{\frac{1}{2}}  \biggr)_{\nu_i \in\nabla} \in \ell_2(\nabla) \,,
\end{equation}
briefly referred to in what follows as $i$th {\em contraction}.  

For later purposes we record some basic facts from \cite{BD}. 
Let $\bu,\bv \in\spl{2}(\nabla^d)$, $\nu \in\nabla$ and $i\in\{1,\ldots,d\}$. Then we have $\norm{\bu} = \norm{\pi^{(i)}(\bu)}$ as well as
\begin{equation}
\label{triangleeq}
  \pi^{(i)}_\nu ( \bu + \bv ) \leq
      \pi^{(i)}_\nu (\bu) + \pi^{(i)}_\nu(\bv),
\end{equation}
and for each $\eta >0$,
\begin{equation}
\label{Ppieta}
  \pi^{(i)}_\nu(\hatPsvd{\eta} \bu) \leq \pi^{(i)}_\nu(\mathbf{u}).
\end{equation}
The contractions can easily be computed using the hierarchical singular value decomposition: let in addition $\mathbf{U}^{(i)}$ be mode frames of an $\Hcal$SVD of $\bu$, and let $(\sigma^{(i)}_k)$ be the corresponding sequences of singular values of the matricizations $T^{(\{i\})}_\bu$, then
\begin{equation*}
 \pi^{(i)}_\nu (\mathbf{u}) = \Bigl( \sum_{k}
 \bigabs{\mathbf{U}^{(i)}_{\nu, k}}^2 \bigabs{\sigma^{(i)}_{k}}^2
 \Bigr)^{\frac{1}{2}}  \,.
\end{equation*}

To quantify the actual number of nonzero entries on components of tensor representations, the notation
$$
 \supp_i(\bu) := \supp\big(\pi^{(i)}(\bu)\big) 
$$
will be useful. 

As a first important application of the sequences \eqref{contractions}, we identify next {\em near-best} $N$-term approximations to an order-$d$ tensor
without considering all entries, but using instead only its contractions.
To this end, consider a non-increasing rearrangement
\beqn
\label{dim-sorting}
\pi^{(i_1)}_{ \nu^{i_1,1}}(\bu)\geq \pi^{(i_2)}_{ \nu^{i_2,2}}(\bu)\geq \cdots \geq \pi^{(i_j)}_{ \nu^{i_j,j}}(\bu)\geq \cdots,\quad \nu^{i_j,j}\in \nabla^{}, 
\eeqn
of the entire set of contractions for all tensor modes,
$
 \bigl\{\pi^{(i)}_\nu(\bu) : \nu\in \nabla,\, i=1,\ldots,d \bigr\}.
$
 Next, retaining only the $N$ largest from the latter total ordering  {\eqref{dim-sorting}} and redistributing them to the {respective}
dimension bins
$\Lambda^{(i)}(\bu;N):= \bigl\{\nu^{i_j,j}: i_j= i,\, j=1,\ldots, N\bigr\}$, %
$i=1,\ldots, d$,
the product set
$$
\Lambda(\mathbf{u};N) := \bigtimes_{i=1}^d
\Lambda^{(i)}(\mathbf{u};N)
$$
can be obtained at a cost that is roughly $d$ times the analogous low-dimensional cost.
By construction, one has
$$
\sum_{i=1}^d \# \Lambda^{(i)}(\mathbf{u};N) \leq N
$$
and
\begin{equation}\label{dim-sorting-optimality}
\sum_{i=1}^d\sum_{\nu\in \nabla^{}\setminus \Lambda^{(i)}(\bu;N)} |\pi^{(i)}_\nu(\bu)|^2
= \min_{\hat \Lambda}\Big\{\sum_{i=1}^d\sum_{\nu\in \nabla^{}\setminus \hat \Lambda^{(i)} } |\pi^{(i)}_\nu(\bu)|^2\Big\},
\end{equation}
where $\hat\Lambda $ ranges over all product sets
$ \bigtimes_{i=1}^d \hat\Lambda^{(i)}$ with $\sum_{i=1}^d \#\hat
\Lambda^{(i)} \leq N$.

\begin{proposition}[cf.\ \cite{B,BD}]
\label{prp:tensor_coarsening_est}
For any  $\mathbf{u} \in \spl{2}(\nabla^d)$ one has
\begin{equation}
\label{eq:tensor_coarsening_errest}  
{ \norm{\mathbf{u} - \Restr{\Lambda(\mathbf{u};N)}\bu} \leq \Bigl( \sum_{i=1}^d\sum_{\nu\in \nabla^{}\setminus \Lambda^{(i)}(\bu;N)}
\bigabs{\pi^{(i)}_\nu(\mathbf{u})}  \Bigr)^{\frac12} =: \mu_N(\bu) \,,  }
\end{equation}
and for any $\hat\Lambda =
\bigtimes_{i=1}^d \hat\Lambda^{(i)}$ with $\Lambda^{(i)} \subset\nabla^{}$
satisfying $\sum_{i=1}^d \#\hat
\Lambda^{(i)} \leq N$, one has
\begin{equation*}
 \norm{\mathbf{u} - \Restr{\Lambda(\mathbf{u};N)}\bu} \leq \mu_N(\bu) \leq \sqrt{d} \norm{\mathbf{u} -
 \Restr{\hat\Lambda}\mathbf{u}} \,.
\end{equation*}
\end{proposition}

Again we switch from near-best approximations for a given budget (here $N$) to approximations realizing a given {\em target accuracy} with near-minimal cost.
To this end, we define
 {$N(\bv,\eta) := \min\bigl\{ N\colon \mu_N(\bv) \leq \eta
  \bigr\}$, where $\mu_N$ is defined   in \eqref{eq:tensor_coarsening_errest},}
as well as the {\em thresholding} procedure
\begin{equation*}
\hatCctr{\eta} (\mathbf{v}) :=  \Restr{\Lambda(\mathbf{u};N(\bv;\eta))}.
\mathbf{v}\,,
\end{equation*}
  As a consequence of \eqref{eq:tensor_coarsening_errest},
we have
$$
  \norm{\mathbf{v} - \Cctr_{\mathbf{v},N} \mathbf{v}} 
  {\leq  \mu_N(\bv) } \leq \constcrs  \min_{\sum_i \# \supp_i(\mathbf{w})\leq N} \norm{\mathbf{u} - \mathbf{w}},
  \quad \constcrs =\sqrt{d}.
$$

In \cite{BD}, we have obtained the following result concerning a combined reduction technique, both with respect to ranks as well as
sparsity of the mode frames, with near-optimal performance.

\begin{theorem}
\label{lmm:combined_coarsening}
Let $\mathbf{u}, \mathbf{v} \in \spl{2}(\nabla^d)$ with
$\mathbf{u}\in\AH{\ga}$, $\pi^{(i)}(\mathbf{u}) \in \cA^s$ for
$i=1,\ldots,d$, and $\norm{\mathbf{u}-\mathbf{v}} \leq \eta$. Let 
$\constsvd = \sqrt{2d-3}$ and $\constcrs = \sqrt{d}$. Then, for any fixed $\alpha >0$,
\begin{equation*}
\mathbf{w}_{\eta} := \hatCctr{\constcrs
  (\constsvd+1)(1+\alpha)\eta} \bigl(\hatPsvd{\constsvd
(1+\alpha)\eta} (\mathbf{v}) \bigr) \,,
\end{equation*}
satisfies
\begin{equation}
\label{eq:combinedcoarsen_errest}
  \norm{\bu - \bw_\eta} \leq C(\alpha,\constsvd,\constcrs)\, \eta \,,
\end{equation}
where $ C(\alpha,\constsvd,\constcrs) :=  \bigl( 1 + \constsvd(1+\alpha) + \constcrs (\constsvd+1)(1+\alpha)\bigr)$, as well as
\begin{equation}\label{eq:combinedcoarsen_rankest}
 \abs{\rank(\bw_\eta)}_\infty \leq \ga^{-1}\bigl(\garatio
 \norm{\bu}_{\AH{\ga}}/(\alpha \eta)\bigr)\,,\qquad \norm{\bw_\eta}_{\AH{\ga}}
 \leq C_1 \norm{\bu}_{\AH{\ga}} ,
\end{equation}
with $C_1 = (\alpha^{-1}(1+\constsvd(1+\alpha)) + 1)$ and
\begin{equation}\label{eq:combinedcoarsen_suppest}
\begin{aligned}
 \sum_{i=1}^d \#\supp_i (\mathbf{w}_\eta) &\leq {2 \eta^{-\frac{1}{s}} d\, \alpha^{-\frac{1}{s}} } \Bigl(
 \sum_{i=1}^d \norm{ \pi^{(i)}(\bu)}_{\As} \Bigr)^{\frac{1}{s}}  \,, \\
   \sum_{i=1}^d \norm{\pi^{(i)}(\bw_\eta)}_{\As} &\leq C_2 \sum_{i=1}^d
   \norm{\pi^{(i)}(\bu)}_{\As} ,
 \end{aligned}
\end{equation}
with $C_2 = 2^s(1+ {3^s}) + 2^{{4s}} { \alpha^{-1} \bigl( 1+  \constsvd(1+\alpha) + \constcrs  (\constsvd + 1)(1+\alpha) \bigr)  }  d^{\max\{1,s\}}$.
\end{theorem}

\begin{remark}\label{rem:CRcomplexity}
Both $\hatPsvd{\eta}$ and $\hatCctr{\eta}$ require a hierarchical singular value decomposition of their inputs.
For a compactly supported $\bv$ given in hierarchical format, the number of operations required for obtaining such a decomposition is bounded, up to a fixed multiplicative constant, by $d\abs{\rank(\bv)}_\infty^4 + \abs{\rank(\bv)}_\infty^2 \sum_{i=1}^d \# \supp_i\bv$.
\end{remark}

\section{Adaptive Application of Rescaled Low-Rank Operators}\label{sec:apply}
The remaining crucial issue for a numerical realization of the iteration \eref{practicaliter} is 
the {\em adaptive application} of  a suitably rescaled version  $\bA$  of a given   operator $\bT$ of finite hierarchical rank. 
Throughout the remainder of the paper we concentrate on $\bT$ given by \eref{eq:unrescaled_tuckersum} 
with low-dimensional components given by
 \eref{T1} and \eref{T2}.
Specifically,
we wish to construct for a given $\bv\in \ell_2(\nabla^d)$ with  finite hierarchical ranks and
 any target tolerance $\eta >0$ an approximation
$\bw_\eta \in \ell_2(\nabla^d)$, satisfying $\|\bw_\eta - \bA \bv\|\leq \eta$, where $\bw_\eta$  has as low hierarchical ranks and as small lower-dimensional supports $\supp_i\bw_\eta$ as possible.

We have already pointed out that scaling operators of the form $\Sc$ with weights from \eref{can-scaling}
cause the preconditioned operator to have infinite rank and obstruct the understanding of low-rank approximations.
The first major issue is therefore to identify equivalent scalings (in the sense of \eref{equivS}) that better support finding such approximations in a quantifiable sense.

The second issue is {\em representation sparsity} of the generated mode frames which will be addressed by 
directly exploiting known results for low-dimensional
wavelet methods. In this context we continue employing at times  the canonical scaling $\Sc$,
since it allows us to use corresponding low-dimensional results on matrix  compression in the most convenient way.

\subsection{Near-Separable Scaling Operators}\label{ssec:nearsep}
The central objective of this section is to identify a scaling operator $\tilde\bS$ which is  {\em equivalent}
to the canonical scaling $\bS$ in the   sense of \eref{equivS}, but can be approximated by separable operators in an
efficient and  quantifiable way.
The main tool is the following result, whose proof  is deferred to Section \ref{sec:proofs}.

\begin{theorem}\label{thm:expsum_relerr}
Let $\alpha(x) := \ln^2(1+e^x)$, $w(x) :=  2 \pi^{-1/2}(1 + e^{-x})^{-1}$. For an arbitrary but fixed $\delta \in (0,1)$
choose some
\begin{equation*}
  h \in \biggl(0, \frac{\pi^2}{5(\abs{\ln(\delta/2)} + 4)}\biggr]  \,,  
\end{equation*}
and set
\beqn
\label{n+}
n^+ = n^+(\delta) := \ceil{ h^{-1} \max\{4 \pi^{-\frac12}, \sqrt{\abs{\ln (\delta/2)}}\} }.
\eeqn
Then, defining
\begin{equation}
\label{eq:expapprox_def}   
   \varphi_{h,n}(t) :=  \sum_{k=-n}^{n^+} h\, w(kh)\, e^{-\alpha(kh)\, t}  \,, \quad \varphi_{h,\infty}(t) := \lim_{n\to\infty} \varphi_{h,n}(t)\,,
\end{equation}
one has 
\begin{equation}\label{eq:expapprox_sinc_est_0}
   \biggabs{ \frac{1}{\sqrt{t}} -   \varphi_{h,\infty}(t) } 
  \leq \frac{\delta}{\sqrt{t}} \quad \text{for all $t\in[1,\infty)$.} 
\end{equation}
For any $\eta > 0$ and $T>1$, provided that
$n \geq \ceil{h^{-1}(\ln 2\pi^{-\frac12} +\abs{\ln (\min\{\delta/2,\eta\})} + \textstyle\frac12\displaystyle \ln T)}$, one has in addition
\begin{equation}\label{eq:expapprox_sinc_est}
\bigabs{ t^{-\frac12} -   \varphi_{h, n }(t) } \leq \frac{\delta}{\sqrt{t}}\,\; \text{ and }\,
\quad 
 \bigabs{ \varphi_{h, \infty}(t) -   \varphi_{h, n }(t) } 
  \leq \frac{\eta}{\sqrt{t}} \quad \text{for all $t\in[1,T]$.} 
\end{equation}
 \end{theorem}
\vspace*{2mm}
\newcommand{\omin}{\omega_\mathrm{min}}
\newcommand{\hatomin}{\hat{\omega}_\mathrm{min}}

To define the modified scaling operator and its approximations the values
 $\delta\in(0,1)$,  $h$, $n^+=n^+ (\delta)$ will be kept fixed according to Theorem \ref{thm:expsum_relerr}.  Furthermore, let
\begin{equation*}
\hatomin :=  \min_{\nu\in\nabla}\min_{i} \omi{i}{\nu}\,,\quad
\omin := \min_{\nu\in\nabla^d} \omega_\nu \geq  \sqrt{d}\, \hatomin \,.
\end{equation*}
For any $n\in\N$, we define now
$$
\Sa{n} \bv 
    = \big( \tilde\omega_{n, \nu } \; v_\nu \big)_{\nu\in\nabla^d} \,, \quad \mbox{where}
    \quad   \tilde\omega_{n, \nu } := \omin \bigl[ \varphi_{h,n}\bigl( (\om{\nu} / \omin)^2 \bigr) \bigr]^{-1},
$$
where the $\om{\nu}$ are defined by \eref{bS}.

\begin{remark}
\label{rem:consequence}
As a consequence of this definition the operator $\Sa{n}^{-1}$ can be represented as a sum of $1 + n^+ (\delta) + n$ {\em separable} terms. 
In the limit $n\to \infty$, we obtain the {\em reference scaling}
\beqn
\label{tilde-limit}   
\Sr\bv := \big( \tilde\omega_{\nu } \; v_\nu \big)_\nu \quad \text{where}\quad 
\tilde\omega_\nu := \lim_{n\to\infty} \tilde\omega_{n, \nu } = \omin \bigl[ \varphi_{h,\infty}\bigl( (\om{\nu} / \omin )^2 \bigr) \bigr]^{-1}   \,. 
\eeqn
\end{remark}

We can now rephrase the statement \eqref{eq:expapprox_sinc_est} in terms of the approximations $\Sa{n}$.
Since the role of $t$ is played by $(\omega_\nu/ \omin)^2$, it will be important to identify the set of indices in $\nabla^d$ for which 
\eqref{eq:expapprox_sinc_est} applies, namely 
\beqn
\label{LambdaT}
\Lambda_T := \bigl\{\nu\in\nabla^d : (\om{\nu})^2 \leq  (\omin)^2  T \bigr\}.
\eeqn
Moreover, the larger $T$ and hence the scale of indices covered by $\Lambda_T$, the more summands are needed to replace
the reference scaling by a finite expansion with a desired relative precision. More precisely, let
 \beqn
\label{Nsc}
\begin{aligned}
  M(\eta;T) &:= \ceil{h^{-1}(\ln 2\pi^{-\frac12} + \abs{\ln (\min\{\delta/2, \eta\})} + \textstyle\frac12\displaystyle \ln T)} \,,\\[2mm]
   M_0(T) &:= M(\delta/2,T)\,.
 \end{aligned}
\eeqn
Then, whenever $n\geq M(\eta;T)$,  one has for any $\eta>0$ and $T>1$  
 both $\abs{\om{\nu}(\om{\nu}^{-1} - \tilde\omega_{n,\nu}^{-1})} \leq \delta$ and $\abs{\om{\nu}(\tilde\omega_{\nu}^{-1} - \tilde\omega_{n,\nu}^{-1})} \leq \eta$ for $\nu\in\Lambda_T$.  In other words,
\beqn
\label{tilde-diff}
  \norm{\bS (\bS^{-1} - \tbS^{-1}_n) \Restr{\Lambda_T} } \leq \delta
 \quad\text{and}\quad \norm{\bS (\tbS^{-1} - \tbS^{-1}_n) \Restr{\Lambda_T} } \leq \eta \,.
\eeqn
Note furthermore that as an immediate consequence of \eqref{eq:expapprox_sinc_est_0},
$$
 1-\delta \leq \tilde\omega_{\nu}^{-1} \omega_\nu \leq 1+ \delta,\quad \nu\in \nabla^d.
$$
Since by definition,
\beqn
\label{smaller}
\tilde\omega_{n,\nu}^{-1}\leq \tilde\omega_\nu^{-1},\quad n\in \N,\,\,\nu\in \nabla^d,
\eeqn
we also obtain $\tilde\omega_{n,\nu}^{-1}\omega_\nu \leq 1+ \delta$, $n\in\N$, $\nu\in \nabla^d$.
The lower estimate $1-\delta \leq \tilde\omega_{n,\nu}^{-1}\omega_\nu$, however, holds only under additional restrictions: by \eqref{eq:expapprox_sinc_est},
\beqn
\label{low-up}
1-\delta \leq \tilde\omega_{n,\nu}^{-1}\omega_\nu \leq 1+ \delta,\quad \mbox{whenever}\quad \nu\in \Lambda_T,\,\, n\geq M_0(T),
\eeqn
For later record we summarize these observations as follows.
\begin{remark}
\label{rem:StildeS}
For the diagonal operators $\Sc, \Sr, \Sa{n}$, we have
\beqn
\label{SStilde}
\|\bS\tbS_n^{-1}\|,\, \|\bS\tbS^{-1}\|\leq 1+ \delta,\quad n\in\N,\quad \qquad \|\tbS\bS^{-1}\| \leq (1-\delta)^{-1} ,
\eeqn
and in particular, the spectral condition of $\tilde\bS\bS^{-1}$ is bounded by $(1+\delta)/(1-\delta)$. Moreover, for any $T>1$ and $n\geq M_0(T)$,
$$
(1-\delta)\|\bS^{-1}\bv\| \leq \|\tilde\bS_n^{-1}\bv\| \leq (1+\delta)\|\bS^{-1}\bv\| \quad \mbox{when } \,\, {\rm supp}\,(\bv)\subseteq \Lambda_T.
$$
\end{remark}

Low-rank approximations based on sinc quadrature have been constructed previously e.g.\ in \cite{Hackbusch:06-1}. 
Theorem \ref{thm:expsum_relerr}, however, has two new features that are particularly useful for our purposes here.
First, our choice of parameters yields a \emph{relative} error estimate, which leads to a substantially better dependence on 
the range parameter $T$ than with standard constructions. Second, adjustments of the finite rank scalings $\tilde\bS_n^{-1}$
can be done by simply {\em adding terms} to the expansion.

In fact, keeping $\delta \in (0,1)$ and a corresponding $h$  fixed, for any given finitely supported $\bv$  and any target accuracy $\eta >0$,
we can determine the number $n$ of terms in the series $\varphi_{h,\infty}$ so that the finite rank scaling $\tilde\bS_n^{-1}$
replaces, for this $\bv$, the reference scaling $\tilde\bS^{-1}$ with accuracy $\eta$ in the sense of \eref{tilde-diff}.
To determine this $n$ we adjust $T$ so that $\supp \bv \subseteq \Lambda_T$, which via \eref{Nsc} yields a lower bound for $n$.
We shall see in Section \ref{sec:proofs} that under certain minimal Sobolev regularity assumptions, this requires $\ln T \sim \max_{\nu\in\supp\bv} \max_i \abs{\nu_i}$. Roughly speaking, this means that for solution accuracy $\varepsilon$, we need $\ln T \sim \abs{\ln\varepsilon}$ and hence a number of terms proportional to $\abs{\ln \varepsilon}$.
Using known exponential sum approximations as in \cite{Hackbusch:06-1} or \cite{Braess:09}  would instead  lead to a number of terms growing like $\abs{\ln \varepsilon}^2$. 

\begin{remark}
A related problem with preconditioning for a fixed discretization in the context of tensor representations is addressed in \cite{Andreev:12}, 
where a BPX-type preconditioner is approximated in the hierarchical tensor format.
There, approximations of a rescaling sequence similar to $\omega_\nu^{-1}$ are constructed numerically in a preparation step, based on direct evaluation and subsequent approximation based on Remark \ref{rem:near-best}, or alternatively based on heuristic black-box approximation. Unfortunately, these approaches do not seem to offer any direct control of relative errors and resulting condition numbers, and are therefore not
suitable for our purposes. The numerical results given in \cite[Table 3.1]{Andreev:12}, however, are consistent with maximum ranks scaling linearly in the maximum discretization level, analogously to our construction.
\end{remark}

\subsection{Low-Rank Adaptive Operator Compression}\label{sssec:apply}
We wish to solve the variant \eref{final} with $\Sr$ given by \eref{tilde-limit}.
What keeps us from applying the results from \cite{BD} directly is the lack of a concrete low-rank approximation
of $\bA$. The objective of this section is to devise such a low-rank approximation based on the operators $\Sa{n}$
introduced above.
 
Aside from controlling ranks we have to exploit the near-sparsity of the preconditioned versions $\bA$ to eventually ensure representation sparsity
of the mode frames.
For low-rank operators this has been done in \cite{BD}. Again the non-separability of the scaling operators
requires additional new concepts.

Nevertheless, a central idea is to exploit the fact that appropriately rescaled versions of the low-dimensional components  $\bT^{(i)}_{n_i}$ of $\bT$
are {\em compressible} in the following sense.  
\begin{definition}\label{def:scompressibility}
Let $\Lambda$ be a countable index set and let $s^* > 0$. An operator
$\mathbf{B}\colon \spl{2}(\Lambda)\to \spl{2}(\Lambda)$  is called
\emph{$s^*$-compressible} if for any $0 < s < s^*$, there exist 
summable positive sequences $(\alpha_j)_{j\geq 0}$, $(\beta_j)_{j\geq 0}$ 
and for each $j\geq 0$, there exists $\mathbf{B}_j$ with at most $\alpha_j 2^j$
nonzero entries per row and column, such that $\norm{\mathbf{B} - \mathbf{B}_j}
\leq \beta_j 2^{-s j} $. 
For a given $s^*$-compressible operator $\mathbf{B}$, we denote the
corresponding sequences by $\alpha(\mathbf{B})$, $\beta(\mathbf{B})$.
Furthermore, we say that  the $\mathbf{B}_j$ have \emph{level decay} if there exists $\gamma >0$ such that $\abs{\abs{\nu} - \abs{\mu}} > \gamma j$ implies $B_{j,\nu\mu} = 0$.
\end{definition}

Note that one can always scale down one of the two sequences $\alpha(\mathbf{B})$, $\beta(\mathbf{B})$
at the expense of the other one. It will be convenient to always assume in what follows that 
\beqn
\label{betascale}
\norm{\beta(\mathbf{B})}_{\ell_1}\leq \norm{\bB}.
\eeqn
Standard wavelet representations of many operators relevant in applications are known to be $s^*$-compressible for some $s^* > 0$, 
see \cite{Cohen:01,Cohen:02,Stevenson:02}. The level decay property is satisfied for each of these examples. For our model problem, we shall rely in particular on the construction and analysis for spline wavelets given in \cite{Stevenson:02} where $s^*$ is shown to exceed the order of the
trial functions.

Let us briefly recall from \cite{Cohen:01} how $s^*$-compressibility is used in the low-dimensional regime.
Suppose that $J\in \N$ and that $\{\Lambda_j\}_{j=0}^{J+1}$ is
{\em any} partition of the index set $\Lambda$. Then, one has for any $\bv\in \ell_2(\Lambda)$
\begin{align}
\label{motivate}
\bB\bv &= \sum_{j=0}^{J+1}\bB \Restr{\Lambda_j}\bv = \sum_{j=0}^{J} \bB_{J-j}\Restr{\Lambda_j}\bv + 
\sum_{j=0}^{J}(\bB- \bB_{J-j})\Restr{\Lambda_j}\bv + \bB \Restr{\Lambda_{J+1}}\bv \nonumber\\
&=: \tilde\bB_J \bv + \mathbf{E}_J \bv.
\end{align}
Since for any $s<s^*$ one has $\norm{ \mathbf{E}_J \bv }  \leq \sum_{j=0}^{J}\beta_j(\bB)2^{-sj}\norm{ \Restr{\Lambda_j}\bv} + \norm{\bB
\Restr{\Lambda_{J+1}}\bv }$
one obtains
$$
\| \mathbf{E}_J \bv\|\leq 2^{s+1} 2^{-sJ} \norm{\bv}_{\As}\|\beta(\bB)\|_{\ell_1} + 2^{-Js} \|\bB\|\,\norm{\bv}_{\As}
\leq (2^{s+1}+1) 2^{-Js} \|\bB\|\,\norm{\bv}_{\As},
$$
provided that $\Lambda_j = {\rm supp}\,(\bv_{2^j}-\bv_{2^{j-1}})$, $j\leq J$, where $\bv_k$ is a best $k$-term approximation to $\bv$,
and $\Lambda_{J+1}:= \Lambda\setminus \Lambda_J$.

Rather then applying this principle directly to $\bA$, as in \cite{Dijkema:09}, we apply it to each component $\bT^{(i)}_{n_i}$
in \eref{eq:unrescaled_tuckersum} and 
consider first approximations $\tilde\bT_J$ to $\bT$, given by \eqref{eq:unrescaled_tuckersum}, in the form
$$ 
  \tilde{\bT} = \tilde\bT_J := \sum_{\kk{n} \in \KK{{d}}(\kk{R})} c_{\kk{n}} \bigotimes_i \tilde{\bT}^{(i)}_{n_i} \,,
$$
where the $\tilde{\bT}^{(i)}_{n_i}$ will be specified next via the concept of compressibility for
the specific cases $\bT^{(i)}_{n_i} = \bT_{n_i}$ from \eref{T1}, \eref{T2}. 
Recall, however, that compressibility
of such low dimensional operators is only known for properly {\em scaled} counterparts. 

In fact,
for sufficiently regular $\psi_\nu$, and with the low-dimensional scaling matrices $ \Sci{i}$  defined in \eref{eq:tensor_rescaling}, the operators
\begin{equation}\label{eq:1dscmatdef} 
 {\mathbf{A}}^{(i)}_2 := \hatbS_i^{-1} \bT_2 \hatbS_i^{-1}\,,\quad 
  {\mathbf{A}}^{(i)}_3 := \bT_3 \hatbS_i^{-1} \,,\quad 
  {\mathbf{A}}^{(i)}_4 := \hatbS_i^{-1} \bT_4 \,,     
  \end{equation}
are bounded on $\spl{2}(\nabla)$ and $s^*$-compressible for some $s^* >0$.
Note that $\norm{ {\mathbf{A}}^{(i)}_4} = \norm{ {\mathbf{A}}^{(i)}_3}$.
This means that for any fixed $s<s^*$ 
\begin{equation}\label{eq:diff_scompr}
\begin{aligned}
\norm{\hatbS_i^{-1} (\bT_{2} -  \bT_{2,j}) \hatbS_i^{-1}} &\leq \beta_j( {\mathbf{A}}^{(i)}_2) \,2^{-sj} \,,  \\
 \norm{ (\bT_{3} -  \bT_{3,j}) \hatbS_i^{-1}} &\leq \beta_j( {\mathbf{A}}^{(i)}_3)\, 2^{-sj} \\
 \norm{\hatbS_i^{-1} (\bT_{4} -  \bT_{4,j})} &\leq \beta_j( {\mathbf{A}}^{(i)}_4)\, 2^{-sj} \,, \end{aligned}
 \end{equation}
where $\mathbf{T}_{n,j}:= \hatbS_i \mathbf{A}^{(i)}_{n,j} \hatbS_i$ and $\mathbf{A}^{(i)}_{n,j}$ 
is the $j$th compression of $\mathbf{A}^{(i)}_{n}$, $n=2,3,4$, according to Definition \ref{def:scompressibility}. 
Therefore, we construct  for any given $J\in \N$ the $\tilde\bT^{(i)}_{n_i}=  \tilde\bT^{(i)}_{n_i,J}$
by the principle \eref{motivate}. In fact, for a given partition $\Lambda^{(i)}_{n_i,[p]}$, $p=0,\ldots,J+1$, we set
\beqn
\label{tildeTni}
\tilde{\bT}^{(i,J)}_{n_i} = \tilde{\bT}^{(i)}_{n_i} := \sum_{p=0}^{J+1} {\bT}^{(i)}_{n_i,[p]} \Restr{\Lambda^{(i)}_{n_i,[p]} }
\eeqn
where as in \eref{motivate}
$$
\mathbf{T}^{(i)}_{n_i,[p]} :=   \mathbf{T}_{n_i,J-p}, \quad p=0,\ldots,J, \,\quad  \mathbf{T}^{(i)}_{n_i,[J+1]} := 0.%
$$
 In fact, as above, the sets $\Lambda^{(i)}_{n_i,[p]}$ 
 provide the vehicle for adaptivity and will depend on a given input sequence
 $\bv\in \ell_2(\nabla^d)$ as follows. 
 For each $i\in \{1,\ldots,d\}$ and for $j\in\N$, we choose  $\bar\Lambda^{(i)}_{j}$ as the
support of the best $2^j$-term approximation $(\pi^{(i)}(\mathbf{v}))_{2^j}$ of $\pi^{(i)}(\mathbf{v})$ so that, in particular,
$\bar\Lambda^{(i)}_p \subset \bar\Lambda^{(i)}_{p+1}$.
If $\mathbf{T}^{(i)}_{n_i} = \id$, we simply set $\mathbf{\tilde T}^{(i)}_{n_i}
= \id$.  If $\mathbf{T}^{(i)}_{n_i} \neq \id$,  we  set $\bar\Lambda^{(i)}_{-1}:=\emptyset$ and
\begin{equation}
\label{eq:supps}
\Lambda^{(i)}_{n_i,[p]}(\bv) = \Lambda^{(i)}_{n_i,[p]} := \Lambda^{(i)}_{[p]} := \left\{
\begin{array}{ll}
\bar\Lambda^{(i)}_{p} \setminus  \bar\Lambda^{(i)}_{p-1}, & p=0,\ldots,J  ,\\
\nabla^{d_i} \setminus \bar\Lambda^{(i)}_J, & p=J+1,\\
\emptyset, & p>J+1,
\end{array}\right.
\end{equation}
As an immediate consequence one has
$$
\norm{\Restr{\Lambda^{(i)}_{n_i,[p]}}\pi^{(i)}(\bv)} = \norm{ (\pi^{(i)}(\mathbf{v}))_{2^p}- (\pi^{(i)}(\mathbf{v}))_{2^{p-1}} } \leq (1+2^s)2^{-ps} \norm{\pi^{(i)}(\bv)}_{\As},
$$
i.e., with increasing $p$ the successively coarser approximations ${\bT}^{(i)}_{n_i,[p]}$ are applied to finitely supported
vectors of successively smaller norms.

Of course we will not apply $\tbT_J$ but the rescaled version
$
\tilde\bA_J := \tbS^{-1}\tbT_J\tbS^{-1}
$
which, however, still has unbounded rank and hence requires a further low-rank approximation $\tbS_n$ of $\tbS$. Here $n$ depends
on  the support of the input vector $\bv$ in such a way that, in an appropriate sense, $\tbS_n$ and $\tbS$
are equivalent on $\supp \bv$. To make this precise,
given any finitely supported $\bv\in \ell_2(\nabla^d)$ and any $J\in \N$, let 
\beqn
\label{Tv}
T(J;\bv):= {\rm argmin}\,\{T' >0 :  \supp \bv \cup \supp \tbT_J \bv \subseteq \Lambda_{T'}\}.%
\eeqn
Moreover, define %
\begin{multline}
\label{eJv}
e_J(\bv) :=  \sum_{i=1}^d C^{(i)}_\bA  \Bigl[  
     \sum_{p=0}^J \Bigl(\sum_{n=2}^R \beta_{J-p}(\bA^{(i)}_n)  \Bigr) 2^{-s(J-p)}   \norm{\Restr{\Lambda^{(i)}_{[p]}} \pi^{(i)}(\bv)} 
\\
  +  \sum_{n=2}^R \norm{{\bA}^{(i)}_n} %
\,  \norm{\Restr{\Lambda^{(i)}_{[J+1]}} \pi^{(i)}(\bv)} \Bigr] ,
\end{multline}
where 
\beqn
\label{eq:maxsequences-0}  
    C^{(i)}_\bA := \max \Bigl\{  \abs{a_{ii}}  ,2
      \sum_{j\neq i}  \norm{{\bA}^{(j)}_3} 
     \abs{a_{ij}} , 
   2  \sum_{j\neq i}  \norm{{\bA}^{(j)}_4} 
          \abs{a_{ij}}  \Bigr\} 
           \leq \max\bigl\{ 1,2 \max_{\substack{j\neq i\\ n=3,4}} \norm{\bA^{(j)}_n} \bigr\} \abs{a_{ii}}. 
\eeqn
In the last inequality we have used that $(a_{ij})$ is diagonally dominant.
In view of \eref{aij}, $ C^{(i)}_\bA $ is thus in particular independent of $d$. Note that for a given finitely supported $\bv$, the a posteriori quantity
$e_J(\bv)$ can be evaluated. It clearly decreases when $J$ increases. This decay is faster, the faster the errors of $2^p$-term approximation of the contractions $\pi^{(i)}(\bv)$ decay.

With these prerequisites at hand, for any given tolerance  $\eta >0$, which we will always assume to satisfy
$\eta \leq 2 \norm{\bA}\norm{\bv}$ -- which is natural, since otherwise $\bA\bv$ can be approximated by zero -- %
we set
\begin{equation}
\label{etaJ}
J(\eta) :=  {\rm argmin}\,\{J\in\N: (1+\delta)^2e_J(\bv)\leq \eta/4\},\qquad
 c(\bv) \,\eta  :=    
\frac{\eta(1-\delta)}{4 \|\bA\|\|\bv\|} ,
\end{equation}
and
\beqn
\label{nsize}
m(\eta;\bv) := M(c(\bv)\eta\,; \,T(J(\eta);\bv)) ,
\eeqn
where $M$ is defined in \eqref{Nsc},
to define the procedure 
$\apply(\bv;\eta): \bv \to \bw_\eta$ by
\beqn
\label{weta-def}
\bw_\eta := \tbS^{-1}_{m(\eta;\bv)}\tbT_{J(\eta)}\tbS^{-1}_{m(\eta;\bv)}\bv.
\eeqn

\begin{remark}
The smallest $T''$ for which $\supp \tbT_J \bv \subseteq \Lambda_{T''}$ is usually larger than the smallest $T'$ 
for which $\supp \bv \subseteq \Lambda_{T'}$. Thus, the number $n'$ of terms needed in a viable scaling of the input vector $\bv$
in \eqref{weta-def} need not be equal to $m(\eta;\bv)$ but can typically be chosen as a smaller integer.
This should be exploited in a numerical realization, but for ease of exposition we work with the above ``symmetric'' version.
\end{remark}

\begin{proposition}
\label{prop:weta}
The finitely supported $\bw_\eta$ defined by \eref{weta-def} satisfies
$\norm{\bA\bv -\bw_\eta} \leq \eta$.
\end{proposition}

We defer the proof of this fact and a further analysis of the procedure  $\apply(\bv;\eta)$, in particular regarding the sparsity of the corresponding mode frames
and the resulting ranks, to Section \ref{sec:proofs}. 

In order to control the ranks of the numerical approximations we shall make use of the excess regularity discussed in
Section \ref{ssec:problemclass}. We shall exploit this through the following strengthened notion of compressibility.

\begin{definition}\label{def:sobolev_compr}
We say that $\bB \colon \spl{2}(\nabla) \to \spl{2}(\nabla)$ is $s^*$-compressible with \emph{Sobolev stability of order $t>0$}, if there exists $C_t>0$ such that  $\norm{\Sci{i}^{t}(\bB - \bB_j)\Sci{i}^{-t}}\leq  C_t \beta_j(\bB)$ for $i=1,\ldots,d$.
\end{definition}

\section{An Adaptive Algorithm and its Complexity}\label{sec:adalg}
\subsection{Formulation of the Algorithm}
As already mentioned in Section \ref{ssec:example}, for the exact right hand side $\bbf$ both the $\supp_i \bbf$ and $\abs{\rank(\bbf)}_\infty$ may be unbounded. In a practical realization of \eqref{practicaliter}, we therefore need to work with approximations, that is, with a procedure $\rhs $ which generates for a fixed given $\mathbf{f}$
and any positive tolerance $\eta >0$ an approximation $\rhs(\eta)$ to $\mathbf{f}$  in the hierarchical Tucker format that satisfies
\beqn
\label{rhs}
\| \mathbf{f}-\rhs(\eta)\|\leq \eta.
\eeqn
In our complexity results, we focus on the costs for constructing a solution $\bu$ for given $\bbf$; we thus assume sufficient knowledge of the data for the explicit construction of a routine $\rhs$ of suitable complexity, which we make more precise in Section \ref{ssec:main} and Appendix \ref{app:rhs}.

Furthermore, we denote by $\coarsen(\cdot;\eta)$ and $\recompress(\cdot;\eta)$   numerical realizations of $\hatCctr{\eta}$ and $\hatPsvd{\eta}$, respectively,
see also \cite{BD}. 
These routines, together with the scheme $\apply$ defined above, are the core ingredients of a numerical realization of the iteration \eref{practicaliter}.

 The following adaptive scheme---Algorithm \ref{alg:tensor_opeq_solve}---has been proposed in essence in \cite{BD}, see also  \cite{B} for a predecessor.
The main difference in the present work lies in the formulation of the routine $\apply$ which, due to the scaling problem discussed in Section \ref{ssec:scaling}, poses severe additional difficulties regarding the complexity and,
in particular, concerning tensor rank bounds for the iterates.

\begin{algorithm}[!ht]
\caption{$\quad \mathbf{u}_\varepsilon = \solve(\mathbf{A},
\mathbf{f}; \varepsilon)$} \begin{algorithmic}[1]
\Require $\Bigg\{$\begin{minipage}{12cm}$\omega >0$ and $\rho\in(0,1)$ such that
$\norm{\id - \omega\mathbf{A}} \leq \rho$,\\
$c_\bA \geq \norm{\bA^{-1}}$, $\varepsilon_0 \geq c_\bA \norm{\mathbf{f}}$, \\
$ \kappa_1, \kappa_2, \kappa_3 \in (0,1)$ with $\kappa_1 +
\kappa_2 + \kappa_3 \leq 1$, and $\beta_1 \geq 0$, $\beta_2 > 0$.\end{minipage}
\Ensure $\mathbf{u}_\varepsilon$ satisfying $\norm{\mathbf{u}_\varepsilon -
\mathbf{u}}\leq \varepsilon$.
\State $\mathbf{u}_0 := 0$
\State $k:= 0$, $I := \min\{ j \colon \rho^j (1 + (\omega + \beta_1 + \beta_2) j) \leq
\textstyle \frac12 \displaystyle\kappa_1\}$\label{alg:jchoice} 
\While{$2^{-k} \varepsilon_0 > \varepsilon$}
\State $\mathbf{w}_{k,0}:=\mathbf{u}_k$, $j \gets 0$
\Repeat
\State $\eta_{k,j} := \rho^{j+1} 2^{-k} \varepsilon_0$
\State $\mathbf{r}_{k,j} := \apply( \mathbf{w}_{k,j} ; \frac{1}{2}\eta_{k,j})
- \rhs(\frac{1}{2}\eta_{k,j})$ 
\State $\mathbf{w}_{k,j+1} := \coarsen\bigl(\recompress(\mathbf{w}_{k,j} - \omega \mathbf{r}_{k,j} ;
\beta_1 \eta_{k,j}); \beta_2 \eta_{k,j} \bigr)$ \label{alg:tensor_solve_innerrecomp}
\State $j\gets j+1$.
\Until{($j \geq I \quad \vee \quad c_\bA \rho
\norm{\mathbf{r}_{k,j-1}} + (c_\bA \rho  + \omega + \beta_1 + \beta_2 )
\eta_{k,j-1} \leq \kappa_1 2^{-(k+1)} \varepsilon_0$)} \label{alg:cddtwo_looptermination_line}
\State $\mathbf{u}_{k+1} := \coarsen\bigl(\recompress(\mathbf{w}_{k,j};
\kappa_2 2^{-(k+1)} \varepsilon_0) ; \kappa_3 2^{-(k+1)}
\varepsilon_0\bigr)$\label{alg:cddtwo_coarsen_line} 
\State $k \gets k+1$
\EndWhile
\State $\mathbf{u}_\varepsilon := \mathbf{u}_k$ 
\end{algorithmic}
\label{alg:tensor_opeq_solve}
\end{algorithm}

The following fact follows exactly as in \cite{BD}. It holds for any fixed choice of the parameters $\kappa_i$ for i=1,2,3 and $\beta_1, \beta_2$
subject to the stated constraints. 
These parameters will later be further specified for a quantitative  complexity analysis.

\begin{proposition}
\label{prp:tensor_iteration_opeq_convergence}
Let the damping factor $\omega >0$ in Algorithm \ref{alg:tensor_opeq_solve} satisfy
$\norm{\id - \omega\mathbf{A}} \leq \rho < 1$.
Then the intermediate steps $\mathbf{u}_{k}$ of Algorithm
\ref{alg:tensor_opeq_solve} satisfy $\norm{\mathbf{u}_k - \mathbf{u}} \leq
2^{-k}\varepsilon_0$, and in particular,
the output $\mathbf{u}_\varepsilon$ of Algorithm \ref{alg:tensor_opeq_solve}
satisfies $\norm{\mathbf{u}_\varepsilon - \mathbf{u}} \leq \varepsilon$.
\end{proposition}

\begin{remark}
The scheme produces an approximation $\bu_\varepsilon \approx \Sr \bu^\circ$, with $u^\circ_\nu = \langle \Psi_\nu,u\rangle$ as in \eqref{bT}. Recovering $\bu^\circ$, the coefficients with respect to the original tensor product orthonormal basis $\{\Psi_\nu\}$ of $\spL{2}(\Omega)$, therefore requires an additional approximate application of $\Sr^{-1}$ based on Theorem \ref{thm:expsum_relerr}.
\end{remark}

\subsection{The Main Result}\label{ssec:main}

We shall now formulate the main result of this paper, which roughly states the following:
if the data $f$ satisfy certain conditions on tensor structure and representation sparsity, and if the exact solution
satisfies similar conditions, then the computed approximation $\bu_\varepsilon$ also exhibits  
a near-optimal low-rank tensor structure and representation sparsity.  Most importantly, the algorithm does not make use of any 
a priori information on such approximability properties. Instead these features---referred to as {\em benchmark assumptions}---of the problem and the exact solution, though {\em not known} explicitly,
 will be shown to  be automatically inherited by the numerical approximation.
 
We formulate next the assumptions under which the main result holds. We begin with conditions on the data $\bA, \bbf$
which are natural for low-rank approximate solutions with sparse factors can be expected to exist.

\begin{assumptions}
\label{ass:A-f}
Concerning the scaled matrix representation $\bA$ given by \eref{final} and the right hand side $\bbf$ we require the following properties for some fixed $s^*, t > 0$:
\begin{enumerate}[{\rm(i)}]
 \item The lower-dimensional component operators $\bA^{(i)}_{n_i}$ as defined in \eqref{eq:1dscmatdef} are $s^*$-compressible with the level decay property and with Sobolev stability of order $t$. 
 \item The number of operations required for evaluating each entry in the approximations $\bT_{n,j}$ as in \eqref{eq:diff_scompr} is uniformly bounded.
\item $\bA$ has a bounded condition, i.e., $\norm{\bA}, \norm{\bA^{-1}} < \infty$.
\item We have an estimate $c_\bA = \norm{\bA^{-1}}$, and the initial error estimate $\varepsilon_0$ overestimates the true value of $\norm{\bA^{-1}}\norm{\bbf}$ only up to some absolute multiplicative constant, i.e., $\varepsilon_0  \lesssim  \norm{\bA^{-1}}\norm{\bbf}$.
\item
The contractions of $\bbf$ are compressible, i.e., $\pi^{(i)}(\bbf)\in \As$, $i=1,\ldots,d$, for any $s$ with $0 < s
 <s^*$.
\item
The problem \eqref{opeq} has excess regularity $t$ as in \eqref{eq:coordwise_regularity}, \eqref{eq:coordwise_reg_f}.
\end{enumerate}
\end{assumptions}

We state next the assumptions concerning the procedure $\rhs$ for approximating the right hand side $\bbf$ that will be used
in the subsequent complexity analysis. We refer to the appendix for scenarios where these assumptions can be realized.

\begin{assumptions}
\label{ass:rhs}
The procedure $\rhs$ is assumed to have  the following properties:
\begin{enumerate}[{\rm(i)}]
\setcounter{enumi}{6}
 \item \label{ass:rhsapprox}There exists an approximation $\bbf_\eta := \rhs(\eta)$ such that \eref{rhs} and
  \begin{gather*}
  \norm{\pi^{(i)}(\bbf_\eta)}_\As \leq C^\text{{\rm sparse}} \norm{\pi^{(i)}(\bbf)}_{\As}    ,\\
   \sum_i \#\supp_i(\bbf_\eta) \leq C^{\text{{\rm supp}}} \,d \,\eta^{-\frac1s}\, \Bigl(\sum_i \norm{\pi^{(i)}(\bbf)}_{\As}\Bigr)^{\frac1s}, \\  \abs{\rank(\bbf_\eta)}_\infty \leq   C_\bbf^{\text{{\rm rank}}}\, \abs{\ln \eta}^{b_\bbf} \,,
  \end{gather*}
hold,  where $C^\text{{\rm sparse}},C^{\text{{\rm supp}}},C_\bbf^{\text{{\rm rank}}} > 0$,  $b_\bbf \geq 1$ are  independent of $\eta$, and $C^\text{{\rm sparse}}$, $C^{\text{{\rm supp}}}$ are independent of $\bbf$.
 \item \label{ass:rhsops}The number of operations required for evaluating $\rhs(\eta)$ is bounded, with a constant $C^\text{{\rm ops}}_\bbf(d)$, by
  $\ops(\bbf_\eta) \leq C^\text{{\rm ops}}_\bbf(d) \bigl[ \abs{\ln\eta}^{3b_\bbf} 
    + \abs{\ln \eta}^{b_\bbf}  \eta^{-\frac1s} \bigr]  $.
 \item $\rhs$ preserves the excess regularity of the problem, that is, there exists $C^\text{{\rm reg}}_\bbf  >0$ independent of $\eta$ such that
 \begin{equation}
  \label{eq:rhsreg}
    \norm{\Sc_{i}^{t} \bbf_\eta} \leq C^\text{{\rm reg}}_\bbf \norm{\Sc_{i}^{t} \bbf} \,.
 \end{equation}
\end{enumerate}
\end{assumptions}

\begin{remark}
Recalling that $\bbf = \Sc^{-1} \bg$, we can obtain \eqref{ass:rhsapprox} and \eqref{ass:rhsops} from Proposition \ref{prop:rhs} in Appendix \ref{app:rhs}, where $b_\bbf = b_\bg + 1$; in particular, if $b_\bg$ is independent of $d$, so is $b_\bbf$.
\end{remark}

Under the above conditions on the data and their processing we are primarily interested to see now whether the adaptive algorithm produces
in a quantifiable way low-rank sparse approximate solutions if the exact solution permits such approximations.
We state now our precise {\em benchmark assumptions} on the solution $\bu$.

\begin{assumptions}\label{ass:approximability}
Concerning the approximability of the solution $\bu$, we   assume:
\begin{enumerate}[{\rm(i)}]
\setcounter{enumi}{9}
\item \label{ass:uapprox}$\bu \in \AH{\ga_\bu}$ with $\ga_\bu(n) = e^{d_\bu n^{1/b_\bu}}$
for some $d_\bu>0$, $b_\bu \geq 1$.
\item $\pi^{(i)}(\bu) \in \As$ for $i=1,\ldots,d$, for any $s$ with $0 < s
 <s^*$.
\end{enumerate}
\end{assumptions}

The rationale of Assumptions \ref{ass:approximability}\eqref{ass:uapprox} is to assess the performance of the highly nonlinear scheme in situations
where the solution does admit low-rank approximations, quantified here by a poly-logarithmic growth of ranks given by $\gamma_\bu^{-1}$,
see Remark \ref{rem:howtoread}.

In order to analyze the dimension-dependence of the complexity of our algorithm, we would ideally need  a {\em reference family} of problems 
exhibiting the same level of difficulty for each $d$. 
Although this is not quite possible, there are problem elements that can be compared for different values of $d$, such as for instance
the structure of the Laplacian. It is therefore important to state next exactly how the relevant quantifies relate to the spatial dimension $d$. 

\begin{assumptions}\label{ass:dim}
In our comparison of problems for different values of $d$, we assume:
\begin{enumerate}[{\rm(i)}]
\setcounter{enumi}{11}
 \item The following are independent of $d$: the constants $d_\bu$, $b_\bu$, $C^\text{{\rm sparse}}$, $C^{\text{{\rm supp}}}$, $C_\bbf^{\text{{\rm rank}}}$; the excess regularity index $t$, and $C^\text{{\rm reg}}_\bbf$ in \eqref{eq:rhsreg}.
  \item The following quantities remain bounded independently of $d$: $\norm{\bA}$ and $\norm{\bA^{-1}}$, see Proposition \ref{prop:cond}; the maximum hierarchical representation rank $\max_\alpha R_\alpha$ of $\bT$;
  the quantities $\norm{\pi^{(i)}(\bu) }_\As$ in the benchmark assumptions, $\norm{\pi^{(i)}(\bbf)}_{\As}$ in Assumptions \ref{ass:approximability}\eqref{ass:rhsapprox}, and the values $\norm{\Sc_{i}^{t} \bbf}$, each for $i=1,\ldots,d$.
\item  In addition, we assume that  $C^\text{{\rm ops}}_\bbf(d)$ as in Assumptions \ref{ass:approximability}\eqref{ass:rhsops} grows at most polynomially as $d\to \infty$.
 \end{enumerate}
 \end{assumptions}

\begin{remark}
 \label{rem:damping}
 As a consequence of the $d$-independent bound on $\norm{\bA}\norm{\bA^{-1}}$, the reduction rate $\rho$ is independent of $d$ and hence the damping parameter $\omega$ can be chosen independently of $d$.
 \end{remark}

We have already seen in Proposition \ref{prp:tensor_iteration_opeq_convergence}  that Algorithm \ref{alg:tensor_opeq_solve}
terminates without any additional assumptions on $\bu$ and in that sense converges. The following main result of this work concerns the complexity of the scheme
when $\bu$ satisfies the benchmark assumptions.

\begin{theorem}
\label{thm:complexity}
Suppose that  Assumptions \ref{ass:A-f}, \ref{ass:rhs} hold and that Assumptions \ref{ass:approximability} are valid for the solution $\bu$  of $\mathbf{A}\bu = \mathbf{f}$.
Let $\alpha > 0$ and let $\constsvd, \constcrs$ be as in Theorem
\ref{lmm:combined_coarsening}.
Let the constants $\kappa_1,\kappa_2,\kappa_3$ in
Algorithm \ref{alg:tensor_opeq_solve} be chosen as
\begin{gather*}
  \kappa_1 = \bigl(1 + (1+\alpha)(\constsvd + \constcrs +
  \constsvd\constcrs)\bigr)^{-1}\,, \\
  \kappa_2 = (1+\alpha)\constsvd \kappa_1\,,\qquad 
  \kappa_3 = \constcrs(\constsvd + 1)(1+\alpha)\kappa_1 \,,
\end{gather*}
and let $\beta_1 \geq 0$, $\beta_2 >0$ be arbitrary but fixed.
Then the approximate solution $\bu_\varepsilon$ produced by Algorithm \ref{alg:tensor_opeq_solve} for $\varepsilon < \varepsilon_0$
satisfies
\begin{gather}
 \label{eq:complexity_rank} 
 \abs{\rank(\bu_\varepsilon)}_\infty 
   \leq \, \bigl( d_\mathbf{\bu}^{-1}
   \ln\bigl[ 2(\alpha \kappa_1)^{-1} \rho_{\ga_\bu}
   \,\norm{\bu}_{\AH{\ga_\mathbf{\bu}}}\,\varepsilon^{-1}\bigr] \bigr)^{ b_\mathbf{\bu}}
    \lesssim (\abs{\ln \varepsilon} + \ln d)^{b_\bu} \,,
   \\
 \label{eq:complexity_supp} \sum_{i=1}^d \#\supp_i(\bu_\varepsilon) \lesssim
     d^{1 + s^{-1}} \, \Bigl(\sum_{i=1}^d \norm{ \pi^{(i)}(\bu)}_{\As} \Bigr)^{\frac{1}{s}}
           \varepsilon^{-\frac{1}{s}} \,,
\end{gather}
as well as
\begin{gather}
  \label{eq:complexity_ranknorm} 
  \norm{\bu_\varepsilon}_{\AH{\ga_\mathbf{\bu}}}
  \lesssim \sqrt{d}\,
      \norm{\bu}_{\AH{\ga_\mathbf{\bu}}}    \,,   \\
  \label{eq:complexity_sparsitynorm} \sum_{i=1}^d \norm{
  \pi^{(i)}(\bu_\varepsilon)}_{\As} \lesssim d^{1 + \max\{1,s\}} 
      \sum_{i=1}^d \norm{ \pi^{(i)}(\bu)}_{\As}  \,.
\end{gather}
The multiplicative constant in \eqref{eq:complexity_ranknorm} depends only on $\alpha$,
those in \eqref{eq:complexity_supp} and \eqref{eq:complexity_sparsitynorm} depend only on
$\alpha$ and $s$.

If in addition, Assumptions \ref{ass:dim} hold, then for the number of required operations $\ops(\bu_\varepsilon)$, we have the estimate
\begin{equation} 
\label{eq:complexity_totalops}
\ops(\bu_\varepsilon) \leq Cd^a \,d^{c s^{-1} \ln d} d^{24 c \ln \ln d}
   \abs{\ln \varepsilon}^{c s^{-1} \ln d + 2 \max\{b_\bu, b_\mathbf{f}\}}\,
       \varepsilon^{-\frac{1}{s}} \,,
\end{equation}
where $C,a$ are constants independent of $\varepsilon$ and $d$, and $c$ is the smallest $d$-independent value such that $I \leq c \ln d$ for $I$ as in line \ref{alg:jchoice} of Algorithm \ref{alg:tensor_opeq_solve}. In particular,  $c$   does not depend on $\varepsilon$ and $s$.
\end{theorem}

Note that the operation count in \eref{eq:complexity_totalops} is essentially of the form
$$
\ops(\bu_\varepsilon)\lesssim d^{C_1\ln d}|\ln \varepsilon|^{C_2 \ln d + 2 \max\{b_\bu, b_\mathbf{f}\}}\,
       \varepsilon^{-\frac{1}{s}},
$$
where $C_1,C_2$ are constants independent of $d$ and $\varepsilon$.
\section{Complexity Analysis and Proof of Theorem \ref{thm:complexity}}\label{sec:proofs}

\subsection{Analysis of Scaling Operators}\label{ssec:proofs-scaling}

\begin{theorem}
\label{thm:expsum_relerr_2}
Let $\delta_0 \in (0,1)$ and
\begin{equation}\label{eq:expapprox_h_2} 
  h \in \biggl(0, \frac{\pi^2}{5(\abs{\ln\delta_0} + 4)}\biggr]  \,.
\end{equation}
Then with $\alpha$, $w$  defined as in Theorem \ref{thm:expsum_relerr}, and $\varphi_{h,n}$ and $\varphi_{h,\infty}$ as in \eqref{eq:expapprox_def} with $n^+  \geq  \ceil{ h^{-1} \max\{4\pi^{-\frac12}, \sqrt{\abs{\ln \delta_0}}\} }$, we have
\begin{equation*}%
   \biggabs{ \frac{1}{\sqrt{t}} -   \varphi_{h,\infty}(t) } 
  \leq \frac{\delta_0}{\sqrt{t}} \quad \text{for all $t\in[1,\infty)$.} 
\end{equation*}
Moreover, for any $\varepsilon > 0$ and  
for all $n\geq \ceil{h^{-1}(\ln 2\pi^{-\frac12} + \abs{\ln\varepsilon})}$, one has
\begin{equation*}
   \bigabs{ \varphi_{h,\infty}(t) -   \varphi_{h,n}(t) }  
  \leq \varepsilon 
  \quad \text{for all $t\in[1,\infty)$.} 
\end{equation*}
\end{theorem}

An immediate consequence of Theorem \ref{thm:expsum_relerr_2} can be formulated as follows.

\begin{corollary}\label{cor:expsum_relerr}
Under the assumptions of Theorem \ref{thm:expsum_relerr}, let in addition $\delta_1 > 0$ such that
$\delta:=\delta_0 + \delta_1 < 1$, and let $T>1$. Then for $\varphi_{h,\infty}$ and $\varphi_{h,n}$ with $n^+$ as in Theorem \ref{thm:expsum_relerr_2} and
\begin{equation*}
  n \geq \ceil{h^{-1}(\ln 2\pi^{-\frac12} + \abs{\ln \delta_1} + \textstyle\frac12\displaystyle \ln T)}
\end{equation*}
we have
\begin{equation*} %
   \bigabs{ \varphi_{h,\infty}(t) -   \varphi_{h,n}(t) } \leq \frac{\delta_1}{\sqrt{t}}
  \,,\quad
   \bigabs{ t^{-\frac12} -   \varphi_{h,n}(t) }  \,
    \leq \,\frac{\delta}{\sqrt{t}}
   \quad \text{for all $t\in[1,T]$.} 
\end{equation*}
\end{corollary}

Choosing $\delta_0 = \delta_1= \delta/2$ in Corollary \ref{cor:expsum_relerr} provides the proof of Theorem \ref{thm:expsum_relerr}.
For the proof of Theorem \ref{thm:expsum_relerr_2}, we need the following definition and approximation estimate from \cite{Stenger:93}.
\newcommand{\zd}{\zeta}

\begin{definition}
\label{def:sincquad_hardynorm}
For $\zd>0$, let $\mathcal{D}_\zd = \{  z \in \C \colon \, \abs{\Im z} < \zd
\}$ and for $0 < \varepsilon < 1$,
\[  
\mathcal{D}_\zd(\varepsilon) = \{  z \in \C \colon \, \abs{\Re z} <
\varepsilon^{-1} ,\, \abs{\Im z} < \zd(1-\varepsilon) \}\,.   
\] 
For $v$ analytic in $\mathcal{D}_\zd$ let
$  
N_1(v,\mathcal{D}_\zd) =
\lim_{\varepsilon \to 0} \int_{\partial \mathcal{D}_\zd(\varepsilon)} \abs{v(z)}\,\abs{dz} \,.   
$
\end{definition}

\begin{theorem}[{cf.\ \cite{Stenger:93}, Theorem 3.2.1}]
\label{thm:sincquad_stenger}
Let $g$ be analytic in $\mathcal{D}_\zd$ with $N_1(g,\mathcal{D}_\zd) <\infty$, then
\begin{equation*}
  \biggl| \int_\R g(x)\,dx  - h \sum_{k \in \Z} g(kh)  \biggr|  
  \leq  \frac{e^{-\pi \zd/h}}{2 \sinh(\pi \zd/h)} N_1(g,\mathcal{D}_\zd)  \,.
\end{equation*}
\end{theorem}

\begin{proof}[Proof of Theorem \ref{thm:expsum_relerr_2}]
Our starting point is the representation (cf.\ \cite{Hackbusch:06-1})
$$  
\frac{1}{\sqrt{t}} = \frac{2}{\sqrt{\pi}} \int_\R \frac{e^{-t \ln^2(1+e^x)}}{1+e^{-x}} \,dx \,. 
$$
The integrand is analytic, in particular, in the strip $\{ x+iy \colon x\in\R, \abs{y}\leq \pi/10\}$, and in order to apply Theorem \ref{thm:sincquad_stenger}, we need to estimate the quantity
$$
  N_1(g, \mathcal{D}_\zd) = \int_\R \abs{g(x+i\zd)}  \,dx
    + \int_\R \abs{g(x-i\zd)} \,dx
    \,,
$$
where $g(z):= \frac{2}{\sqrt{\pi}} \frac{e^{-t \ln^2(1+e^z)}}{1+e^{-z}}$.
Note first that $\abs{1 + e^{x\pm i\zd}}^2 \geq 1 + e^{2x} \geq \frac12 (1 + e^x)^2$ for $x\in\R$.
Let
$$ 
r_\zd(x) := \Re \ln^2(1+e^{x\pm i\zd}) = \frac14 \ln^2(1 + 2 e^{x} \cos \zd + e^{2x}) 
      - \biggl( \arctan \frac{\sin \zd}{ \cos \zd + e^{-x}}  \biggr)^2  \,. 
$$

For $\abs{\zd} \leq \frac{\pi}{10}$, we now prove that
$r_\zd(x)\geq \frac{1}{4} x^2$ for $x\geq 0$ and $r_\zd(x) \geq \frac{1}{8} e^{2x}$ for $x\leq 0$.
We first consider $x\leq 0$. Using that $\ln(1+y)\geq \frac{1}{2}y$ for any $y\in[0,2]$, we obtain
$$  
\frac14 \ln^2(1 + 2 e^{x} \cos \zd + e^{2x})  \geq \frac14 (e^x \cos \zd)^2 \,,
$$
and furthermore
$$
  \biggl( \arctan \frac{\sin \zd}{ \cos \zd + e^{-x}} \biggr)^2 \leq \biggl(  \frac{\sin \zd}{ \cos \zd + e^{-x}} \biggr)^2  \leq \zd^2 e^{2x} \,, \quad x\in\R\,.   
$$
Hence $r_\zd(x) \geq \frac14 (e^x \cos \zd)^2 - \zd^2 e^{2x}$, and the estimate 
$ \cos \zd \geq (\textstyle\frac{1}{2}\displaystyle + 4\zd^2)^{\frac12} $,
which holds for $\abs{\zd} \leq \frac\pi{10}$,
yields $\frac14 (e^x \cos \zd)^2 - \zd^2 e^{2x} \geq \frac18 e^{2x}$ for $x\leq 0$, as claimed. 

We now consider $x >0$, where we shall repeatedly use
$$   
\biggabs{\arctan \frac{\sin \zd}{ \cos \zd + e^{-x}}} \leq \abs{\zd} \,, \quad x\in\R\,.   
$$
To see that $r_\zd(x)\geq \frac{1}{4} x^2$ for $x\in (0,1)$, we observe first that $\frac14 \ln^2(1 + 2\cos \zd + 1) -\zd^2 \geq \frac14$ holds for $\zd = \pi/10$, and hence also for $\abs{\zd} \leq \pi/10$ by monotonicity. Consequently, for $x\in(0,1)$, one has
$$ 
   r_\zd(x) \geq \frac14 \ln^2(1 + 2e^x \cos \zd + e^{2x}) - \zd^2  
     > \frac14 \ln^2 ( 1 + 2e^0 \cos \zd + e^{0}) - \zd^2 \geq \frac14 > \frac14 x^2 \,.
$$
In the remaining case $x\geq 1$, we use the estimate $\ln(1+e^{2x})\geq 2x$ to obtain
$$  
\frac14 \ln^2(1 + 2 e^{x} \cos \zd + e^{2x}) \geq \frac14 \ln^2(1 + e^{2x}) \geq
  \frac{1}{4} (2x)^2 = x^2\,, 
$$
and thus $r_\zd(x) \geq x^2 - \zd^2$. Consequently, $r_\zd(x)\geq \frac14 x^2$ follows, since in the latter case $\zd^2 < \frac{3}{4} \leq \frac34 x^2$.

In summary, for $\abs{\zd}\leq \pi/10$, we obtain
$$ 
 \int_{\R^+} \biggabs{ \frac{e^{-t \ln^2(1+e^{x\pm i\zd})}}{1+e^{-(x\pm i\zd)}} } \,dx
   \leq 2  \int_{\R^+} \frac{ e^{-t \,r_\zd(x)} }{1 + e^{-x}} \,dx
    \leq 2 \int_{\R^+}  e^{-\frac{t}{4}  x^2} \,dx
   = 2\sqrt{\pi} \, t^{-\frac12}  $$
as well as
\begin{equation*}
    \int_{\R^-} \biggabs{ \frac{e^{-t \ln^2(1+e^{x\pm i\zd})}}{1+e^{-(x\pm i\zd)}} } \,dx 
    \leq 2 \int_{\R^+} \frac{e^{-\frac{t}{8} e^{-2x}}}{1+e^{x}}\,dx \\
  = 2\int_0^1 \frac{e^{-\frac{t}{8} \xi^2}}{(1+\xi^{-1})\xi} \,d\xi \leq 2 t^{-\frac12}    \,,
\end{equation*}
where we have used the substitution $x=-\ln \xi$. 

Theorem \ref{thm:sincquad_stenger} now yields
\begin{align}
\biggabs{ \frac{1}{\sqrt{t}} - 
   \sum_{k\in\Z} h\, \omega(kh) e^{-\alpha(kh)\, t} } 
   &\leq 8(1 + \pi^{-\frac12}) t^{-\frac12} \frac{e^{-\pi \zd/h}}{2 \sinh(\pi \zd/h)} \nonumber\\
   &\leq 16(1 + \pi^{-\frac12}) \,t^{-\frac12} e^{-\pi^2 /(5h)} \,, \label{stenger_final_est}
\end{align}
where we have used $\zd = \pi/10$ and that $h\leq \pi^2/(5 \ln 2)$ by our assumption on $h$, which in turn implies $e^{-\pi \zd/h}/(2 \sinh(\pi \zd/h)) \leq 2 e^{-2\pi \zd / h}$. Again by the choice of $h$ as in \eqref{eq:expapprox_h_2}, the right hand side in \eqref{stenger_final_est} is bounded by $\frac12 t^{-\frac12} \delta_0$.

The estimates for $n^+$ and  $n$ follow from the decay of the integrand on $\R$: on the one hand, we have 
$$  
\sum_{k> n^+} h\, \omega(kh)  e^{-\alpha(kh)\, t} \leq 2\pi^{-\frac12} h \int_{n^+}^\infty  e^{-t(xh)^2}\, dx \leq t^{-\frac12} \,  2\pi^{-\frac12} \int_{n^+ h\sqrt{t}}^\infty e^{-x^2}\,dx  \,, 
$$
and furthermore
$$    
2\pi^{-\frac12} \int_{n^+ h\sqrt{t}}^\infty e^{-x^2}\,dx
  \leq   2\pi^{-\frac12}  \int_{n^+ h}^\infty  \frac{2x e^{-x^2}}{n^+ h} \,dx
    \leq   2\pi^{-\frac12} \frac{e^{-(n^+ h )^2}}{n^+ h}  \,.
$$
The expression on the right hand side is bounded by $\frac12 \delta_0$ for $n^+ \geq \max\{ 4\pi^{-\frac12} h^{-1},  h^{-1} \sqrt{\abs{\ln\delta_0}}  \}$, which leads to the condition on $n^+$ stated in the assertion.
On the other hand, 
$$
  \sum_{k < -n} h\, \omega(kh)  e^{-\alpha(kh)\, t} 
    \leq   2\pi^{-\frac12} \int_{n h}^\infty e^{-x} \,dx  
      \leq 2\pi^{-\frac12} e^{-n h}\,,
$$
and the expression on the right hand side is bounded by $\varepsilon$ for all $t\in[1,T]$ for 
$n \geq h^{-1}(\ln 2\pi^{-\frac12} + \abs{\ln\varepsilon})$.
\end{proof}

We record next some consequences of Theorem \ref{thm:expsum_relerr} and the related definitions from Section \ref{ssec:nearsep} that will be required later.
First  we quantify the equivalence between the two systems \eqref{eqsystem} and \eqref{final}.
\begin{remark}
\label{rem:equivsystems}
For any $\bB\in \R^{\nabla^d\times\nabla^d}$ and $\bv\in \ell_2(\nabla^d)$, 
\beqn
\label{tildeswitch}
(1-\delta)\|\bS^{-1}\bB\bS^{-1} (\bS \tbS^{-1} \bv)\| \leq \|\tbS^{-1}\bB\tbS^{-1}\bv\|\leq (1+\delta)\|\bS^{-1}\bB\bS^{-1} (\bS \tbS^{-1} \bv)\|.
\eeqn
\end{remark}
\begin{proof}
We infer from Remark \ref{rem:StildeS} that
\begin{equation*}
  \norm{ \tbS^{-1}\bB \tbS^{-1} \bv} =
     \norm{ (\tbS^{-1} \bS) \bS^{-1} \bB \bS^{-1} (\bS \tbS^{-1} \bv)} 
    \leq (1 + \delta) \norm{\bS^{-1} \bB \bS^{-1} (\bS \tbS^{-1} \bv)} \,.
\end{equation*}
The lower bound follows  from Remark \ref{rem:StildeS} in an analogous fashion.
\end{proof}

The significance of \eref{tildeswitch} becomes clear when taking $\bB = \bT -\tilde\bT$ where $\tilde\bT$
is an approximation for $\bT$. Here $\tilde\bT$ stands for a ``compressed'' version of $\bT$.
Recall that matrix compression is usually done for the energy scaled version $\bA$, not for the $L_2$ representation
$\bT$. However, since the process of discarding matrix entries and scaling 
commutes and since, in view of \eref{tensorscale}, we can make use of existing results
for the lower-dimensional  canonical scaling, we can compare the corresponding variants.

\begin{lemma}
\label{lem:finite-n}
Let $\bv \in \spl{2}(\nabla^d)$ and $T>0$ such that 
\beqn
\label{supp-cond}
\supp \bv \subseteq \Lambda_T, \quad \supp (\bS^{-1} \tbT \bS^{-1}\bv) \subseteq \Lambda_T,
\eeqn
and define $\tilde{\bf D} := \tbS^{-1}(\tbT - \bT)\tbS^{-1}$.
Then whenever $n\geq M(\eta;T)$,
 one has
\begin{eqnarray}
\label{alt-bound}
\norm{( \tbS^{-1}\bT\tbS^{-1} -  \tbS^{-1}_n\tbT\tbS^{-1}_n)\bv } & \leq &  
\|\tilde{\bf D}\bv\|+ \|\tilde{\bf D}(\id - \tbS\tbS^{-1}_n)\bv\| \nonumber\\
&& + \frac{\eta}{1-\delta}\|\tilde{\bf D}(\tbS\tbS_n^{-1}\bv)\|
+\frac{2 \eta}{1-\delta}\|\bA\|\|\bv\|.\qquad
\end{eqnarray}
\end{lemma}
\begin{proof}
Note  that
\begin{equation}
\label{eq:scalingapprox_basic}
 \norm{(\tbS^{-1} \bT \tbS^{-1} - \tbS^{-1}_n \tbT \tbS^{-1}_n)\bv} \\
 \leq \norm{ \tbS^{-1} (\bT - \tbT) \tbS^{-1} \bv} + 
  \norm{\tbS^{-1} \tbT \tbS^{-1} \bv - \tbS^{-1}_n \tbT \tbS^{-1}_n \bv } \,.
\end{equation}
The second term corresponds to the deviation of the finite-rank operator $\tbS^{-1}_n$ from the reference $\tbS^{-1}$.
Here we obtain
\begin{align}
\label{secondpart}
  \norm{(\tbS^{-1} \tbT \tbS^{-1} - \tbS^{-1}_n \tbT \tbS^{-1}_n)\bv} &\leq
    \norm{\tbS^{-1} \tbT (\tbS^{-1} - \tbS^{-1}_n) \bv} 
     + \norm{(\tbS^{-1} - \tbS^{-1}_n) \tbT \tbS^{-1}_n \bv}.%
\end{align}
To bound the second summand on the right hand side of \eqref{eq:scalingapprox_basic},
 we estimate the first summand on the right hand side of \eref{secondpart} by
\begin{eqnarray*}
 \norm{\tbS^{-1} \tbT (\tbS^{-1} - \tbS^{-1}_n) \bv} & = &\norm{\tbS^{-1} \tbT \tbS^{-1}(\id - \tbS\tbS^{-1}_n)\bv}\\
 &\leq & \norm{\bA (\id - \tbS\tbS^{-1}_n)\bv} +\norm{\tbS^{-1}( \tbT  - \bT)\tbS^{-1}(\id - \tbS\tbS^{-1}_n)\bv}.
 \end{eqnarray*}
Now note that, whenever $\supp \bv \subseteq \Lambda_T$, $n\geq M(\eta;T)$, we infer from Remark \ref{rem:StildeS} that
\beqn
\label{I-S}
\big|\big(\id - \tbS\tbS^{-1}_n\big)_\nu\big| = \bigabs{\tilde\omega_\nu (\tilde\omega_\nu^{-1} - \tilde\omega_{n,\nu}^{-1})} \leq (1-\delta)^{-1}
\bigabs{\omega_\nu (\tilde\omega_\nu^{-1} - \tilde\omega_{n,\nu}^{-1})} \leq (1-\delta)^{-1} \eta.
\eeqn
Hence we obtain
$$
 \norm{\tbS^{-1} \tbT (\tbS^{-1} - \tbS^{-1}_n) \bv}\leq \frac{\eta}{1-\delta} \|\bA\|\,\|\bv\| + \|\tbS^{-1}(\bT-\tbT)\tbS^{-1}(\id - \tbS\tbS^{-1}_n)\bv\|.
$$
As for the second summand on the right hand side of \eref{secondpart}, we  
argue as above, now using the second relation in \eqref{supp-cond}, to conclude that 
\begin{eqnarray*}
\norm{(\tbS^{-1} - \tbS^{-1}_n) \tbT \tbS^{-1}_n \bv} 
  &=& \norm{\Restr{\Lambda_T} (\id- \tbS\tbS^{-1}_n)(\tbS^{-1}\tbT\tbS^{-1})(\tbS\tbS^{-1}_n)\bv}\\  
 &\leq&  \frac{\eta}{1-\delta} \norm{(\tbS^{-1}\tbT\tbS^{-1})(\tbS\tbS^{-1}_n)\bv}\\
&\leq & \frac{\eta}{1-\delta}\big(\norm{\tbS^{-1}(\bT-\tbT)\tbS^{-1}(\tbS \tbS^{-1}_n\bv)} + \norm{\bA}\,\norm{\bv}\big) , 
\end{eqnarray*}
where we have also used \eqref{SStilde} and \eref{smaller}. Combining both estimates confirms the assertion \eref{alt-bound}.
\end{proof}

 {As will be seen later the estimates \eref{alt-bound} can benefit from the fact that the compressed version $\tbT$ of $\bT$
depends on the given $\bv$ so that the quantities $ \|\tbS^{-1}(\bT-\tbT)\tbS^{-1}\bv\|$ are small and controlled
by a posteriori bounds.}

We conclude this section interrelating the compressibility of  the contractions
of solutions to the systems \eqref{eqsystem} and \eref{final} which differ only by the rescaling.
\begin{remark}
\label{rem:allthesame}
As before let $\sigma_N(\hat\bv)$ denote the error of best $N$-term approximation of $\hat\bv\in \ell_2(\nabla)$ and
let $\tilde\bv := \bS\tbS^{-1}\bv$ for any given $\bv\in \ell_2(\nabla^d)$. Then one has
\beqn
\label{sigmaNtilde}
\sigma_N(\pi^{(i)}(\tilde\bv)) \leq (1+\delta)\sigma_N(\pi^{(i)}(\bv)),
\eeqn
and 
\beqn
\label{sigmaNtilde-2}
\sigma_N(\pi^{(i)}(\bv)) \leq (1-\delta)^{-1}\sigma_N(\pi^{(i)}(\tilde\bv)).
\eeqn
Hence we have in particular
 \beqn
\label{Asame}
\|\pi^{(i)}(\tilde\bv)\|_{\mathcal{A}^s}\leq (1+\delta) \|\pi^{(i)}(\bv)\|_{\mathcal{A}^s}   %
\leq \frac{1+\delta}{1-\delta} \|\pi^{(i)}(\tilde\bv)\|_{\mathcal{A}^s},\quad    \bv \in \ell_2(\nabla^d),\, i=1,\ldots,d.
\eeqn
Moreover, for $\tilde\bv := \bS\tbS_n^{-1}\bv$, \eref{sigmaNtilde} holds again for all $\bv\in\ell_2(\nabla^d)$, while \eref{sigmaNtilde-2}
holds in this case only for $\supp \bv\subseteq \Lambda_T$ when $n\geq M_0(T)$.
 \end{remark}
\subsection{Analysis of the procedure {\rm\textsc{apply}}}\label{ssec:apply-analysis}
\newcommand{\bD}{\mathbf{D}}
\newcommand{\bC}{\mathbf{C}}

The following main result of this section    collects  the relevant
properties of  the procedure $\apply$.

\begin{theorem}
\label{thm:apply}
Given $\eta>0$,  and any  finitely supported $\bv\in \ell_2(\nabla^d)$,  let $\bw_\eta$ be defined by \eref{weta-def}.
Then the following statements hold:
\begin{enumerate}[{\rm(i)}]

\item We have the estimates 
\begin{align} 
\label{eq:approx-eta}
\norm{\mathbf{A}\mathbf{v} - \mathbf{\bw_\eta}} & \leq \eta\,,   \\
   \#  \supp_i (\bw_\eta)  &\leq 
  \|\hat\alpha\|_{\ell_1}  \eta^{-\frac{1}{s}}
  \Big(2^4(2^{s}+2) R^{1+s} \sum_{i=1}^d C^{(i)}_\bA \max_{n>1} \norm{\bA^{(i)}_n}\,   \norm{\pi^{(i)}(\bv)}_{\As}
   \Big)^{\frac1s},
   \label{eq:tensor_apply_support} 
\end{align}
where $\hat\alpha :=(\hat\alpha_k)_{k\in\N}$ and $\hat\alpha_k := \max_{i\in\{1,\ldots,d\}} \max_{n>1} \alpha_k({\bA}^{(i)}_n)$.
\item
The outputs of $\apply$ are sparsity-stable in the sense that for $ i\in\{1,\ldots,d\}$,
\beqn
\label{sparsity-stable}
\norm{\pi^{(i)}( \bw_\eta )}_\As \leq \Bigl( \check{C}^{(i)}_\bA  + \frac{2^{3s+2}}{2^s - 1} \norm{\hat{\alpha}}_{\spl{1}}^s 
\max_{n > 1} \norm{\bA^{(i)}_n}\, 
C^{(i)}_\bA \Bigr) R^s (1+\delta)^2 \,  \norm{\pi^{(i)}(\bv)}_\As \,,
\eeqn
 where $C^{(i)}_\bA$ is defined in \eqref{eq:maxsequences-0} and 
\beqn
\label{checkAi-bound}
\check{C}^{(i)}_\bA  :=  12\, (d-1) \max_{j\neq i} \abs{a_{jj}}\,  \bigl(\max_{i,n_i} \norm{\bA^{(i)}_{n_i}}\bigr)^2
 \,.
\eeqn
\item
For the hierarchical ranks of  $\bw_\eta$, we have the bounds
\begin{equation}
 {\rank_\alpha (\bw_\eta  )} \leq \bigl(\hat m(\eta;\bv)\bigr)^2 R_\alpha \rank_\alpha(\mathbf{v}),\quad \alpha \in \cD_d \,,
   \label{eq:tensor_apply_ranks}
\end{equation}
with $R_\alpha$ as in \eqref{eq:htucker_operator_core}, where \begin{equation}
\label{scrankstotal}
\hat m(\eta;\bv) := 1 +n^+(\delta) + m(\eta;\bv), \end{equation}
with $n^+(\delta)$ given by \eref{n+} in Section \ref{ssec:nearsep}, and $m(\eta;\bv)$ defined in \eqref{nsize}.
\item
The number ${\ops}(\bw_\eta)$ of floating point operations required to compute $\bw_\eta$ in the 
hierarchical Tucker format for a given $\bv$ with ranks $\rank_\alpha(\bv)= r_\alpha$, $\alpha \in \cD_d\setminus \{0_d\}$, and $r_{\hroot{d}}=1$, scales like
\begin{multline}
\label{eq:flops}
{\ops}(\bw_\eta)\lesssim  \sum_{\alpha\in \Ncal(\hdimtree{d})}  \bigl(\hat m(\eta;\bv)\bigr)^6 R_\alpha r_\alpha \prod_{q=1}^2
R_{c_q(\alpha)}r_{c_q(\alpha)} \\
 +  \eta^{-1/s} \sum_{i=1}^d \|\hat\alpha\|_{\ell_1}  \bigl(\hat m(\eta;\bv)\bigr)^2 R r_i \Big(\sum_{j=1}^d C^{(j)}_\bA R\|\pi^{(j)}(\bv)\|_{\As}\Big)^{1/s},
\end{multline}
where the constant is independent of $\eta, \bv$, and $d$.
\item
Assume in addition that the approximations $\bT_{n,j}$ have the level decay property (see Definition \ref{def:scompressibility}). Denoting by $L(\bv)$ the largest coordinatewise level appearing in $\bv$, the scaling ranks $\hat m(\eta;\bv)$ as defined in \eqref{scrankstotal} can be bounded by
\beqn
\label{metav}
\hat m(\eta;\bv) \le C(\delta, s,\bA)\,\Big[1 + L(\bv)  +  \abs{\ln \eta } +   \ln \Big(\sum_{i=1}^d
 \norm{\pi^{(i)}(\bv)}_\As\Big) \Big].
\eeqn
\end{enumerate}
\end{theorem}

The proof of Theorem \ref{thm:apply} is based on several auxiliary results. We begin with some 
useful facts concerning scaling of tensor product operators. 

For later reference, we recall the simple fact that for a rank-one operator $\bB = \bB^{(1)} \otimes \bB^{(2)}\otimes \cdots \otimes \bB^{(d)}$, one has 
\begin{multline}
\label{Btensor}
\bB = \big( \bB^{(1)}\otimes \cdots \otimes \bB^{(i-1)}\otimes \id_i\otimes \bB^{(i+1)}\otimes\cdots\otimes \bB^{(d)} \big) \\
\big(\id_1\otimes \cdots\otimes \id_{i-1}\otimes \bB^{(i)}\otimes \id_{i+1}\otimes\cdots\otimes \id_d\big)  
\end{multline}
with the canonical interpretation when $i=1,d$.

\begin{lemma}
\label{lem:tensor-scaling}
For $\bB,  \mathbf{C}\in \R^{\nabla\times\nabla}$ one has
\beqn %
\label{tensorscale}
\begin{aligned}
\bignorm{ \Sc^{-1}\big[\bB \otimes \id_2 \otimes \cdots\otimes \id_d \big]\Sc^{-1} } &\leq \bignorm{\Sci{1}^{-1}\bB \Sci{1}^{-1} }  \\[2mm]
\bignorm{ \Sc^{-1}\big[\bB \otimes \mathbf{C}\otimes \id_3 \otimes \cdots\otimes \id_d \big]\Sc^{-1} } &\leq 
\min \bigl\{ \bignorm{\bB \Sci{1}^{-1}} \bignorm{ \Sci{2}^{-1} \mathbf{C}} ,\,
     \bignorm{\Sci{1}^{-1} \bB } \bignorm{  \mathbf{C} \Sci{2}^{-1}} \bigr\},
\end{aligned}
\eeqn
and permuting the variables, analogous relations hold for $\bB$ at the $i$-th and $\mathbf{C}$ at the $j$-th position, with $\Sci{1}$ and $\Sci{2}$ replaced by $\Sci{i}$ and $\Sci{j}$, respectively.
\end{lemma}

\begin{proof}
From the observation that $\norm{\Sc^{-1}(\Sci{1} \otimes\id_2\otimes\cdots\otimes\id_d)}\leq 1$, the first relation in \eqref{tensorscale} is clear. The second inequality follows by an analogous argument.
\end{proof}

We proceed now analyzing the adaptive application of rescaled versions of the operator $\bT$ first for the canonical scaling $\Sc$,
because this allows us most conveniently to tap results on matrix compression in the univariate case.
To this end, we define   the approximation
\beqn
\label{eq:sc_approx_operator}
  \mathbf{\tilde A}_{c,J} := \bS^{-1}\tilde\bT_J \bS^{-1}, \quad \tbT_J = \sum_{\kk{n}\in\KK{d}(\kk{R})} c_\kk{n} 
       \bigotimes_{i=1}^d  \mathbf{\tilde T}^{(i)}_{n_i}  .
\eeqn
In order to simplify notation in the following error estimates, in analogy to \eqref{eq:1dscmatdef}, we introduce the abbreviations
\begin{equation} \label{eq:tilde1dscmatdef}
\tilde{\bA}^{(i)}_{2} :=  \hatbS_i^{-1}\mathbf{\tilde T}^{(i)}_{n_i}\hatbS_i^{-1} \,,
\quad \tilde{\bA}^{(i)}_3 :=  \mathbf{\tilde T}^{(i)}_{3}\hatbS_i^{-1}  \,,
\quad \tilde{\bA}^{(i)}_4 :=  \hatbS_i^{-1}\mathbf{\tilde T}^{(i)}_{4}\,,\quad  i=1,\ldots, d, 
\end{equation}
for the compressed versions of the properly scaled lower dimensional components of $\bT$.
Note that by Definition \ref{def:scompressibility} and \eref{betascale}, we have the uniform bounds
\beqn
\label{eq:approxfactor_unifest}
  \norm{\tilde{\bA}^{(i)}_{n_i}} \leq 2 \norm{{\bA}^{(i)}_{n_i}} , \quad n_i\leq R, \,\, i=1,\ldots, d.
\eeqn

The next result, although still formulated for the canonical scaling $\Sc$, will serve as a first step towards an adaptive application of 
$\bA$ defined by \eref{final}. Although similar in spirit to a comparable result in \cite{BD}, the presence of the scaling operator $\bS$ 
requires a slightly different treatment.

\begin{lemma}
\label{thm:opapprox_basic}
Let  $\bA_c =\bS^{-1}\bT\bS^{-1}$ be defined by \eqref{eqsystem} and assume that \eqref{eq:diff_scompr} holds for $s<s^*$. Moreover, let $\bv \in \spl{2}(\nabla^d)$ with $\pi^{(i)}(\bv)\in\As$, $i=1,\ldots,d$. Then for each $J\in\N$ and $\tilde\bA_{c,J}$, defined by \eref{eq:sc_approx_operator} with the $\bv$-dependent partitions 
\eref{eq:supps}, one has the 
a posteriori bound
\beqn
\label{aposteriori}
   \norm{\bA_c \bv - \mathbf{\tilde A}_{c,J} \bv} \leq 
   e_J(\bv),
\eeqn
where $e_J(\bv)$ is defined by \eqref{eJv}, as well as the
a priori estimate
\beqn
\label{eq:full_opapprox_est}
   \norm{\bA_c \bv - \mathbf{\tilde A}_{c,J} \bv} \leq 
        2^{-sJ} (2^{s}+2) \sum_{i=1}^d C^{(i)}_\bA \Bigl( \sum_{n=2}^R \norm{\bA^{(i)}_n} \Bigr)   \norm{\pi^{(i)}(\bv)}_{\As},
\eeqn
where the constants $C^{(i)}_\bA $ have already been defined in \eqref{eq:maxsequences-0}.
Moreover,  one has  the support estimate
\beqn
\label{eq:full_opapprox_support}
  \#\supp_i\mathbf{\tilde A}_{c,J} \bv \leq R\norm{\hat\alpha}_{\spl{1}} 2^J \,,\quad i=1,\ldots, d.
\eeqn
\end{lemma}

\begin{proof}
Noting that
$$
\bigotimes_{i=1}^d \bB^{(i)} - \bigotimes_{i=1}^d \bC^{(i)} = \sum_{j=1}^d \bigotimes_{i=1}^{j-1}\bB^{(i)} \otimes(\bC^{(j)}
-\bB^{(j)})\bigotimes_{i=j+1}^d \bC^{(i)},
$$
again with the canonical interpretation for $j=1,d$,
 we can write
\begin{align*}
 &\bS^{-1}(\bT - \tbT)\bS^{-1} \\
\qquad &= \sum_{\kk{n}\in \KK{d}(\kk{R})} c_{\kk{n}}\bS^{-1}\Big(\bigotimes_{i=1}^d \bT^{(i)}_{n_i} - \bigotimes_{i=1}^d \tilde\bT^{(i)}_{n_i}\Big)\bS^{-1}\\
&=  \sum_{n_1\leq R, \,p\in \N} 
\bS^{-1} \,(\bT^{(1)}_{n_1} - \tilde{\bT}^{(1)}_{n_1,[p]} ) \Restr{\Lambda^{(1)}_{n_1,[p]}} 
    \otimes  \Bigl( \sum_{\kk{\check{n}}_1} c_{\kk{n}} \bT^{(2)}_{n_2} \otimes \cdots\otimes \bT^{(d)}_{n_d} \Bigr)\, \bS^{-1}\\[3mm]
& \quad + \cdots + 
 \sum_{n_d\leq R,\,p \in\N}\bS^{-1}\,  \Bigl( \sum_{\kk{\check{n}}_d} c_{\kk{n}} \tilde{\bT}^{(1)}_{n_1} \otimes \cdots\otimes \tilde{\bT}^{(d-1)}_{n_{d-1}} \Bigr) \otimes (\bT^{(d)}_{n_d} - \tilde{\bT}^{(d)}_{n_d,[p]} ) \Restr{\Lambda^{(d)}_{n_d,[p]}}\, \bS^{-1}.\qquad
\end{align*}
Using the triangle inequality and recalling that $\|\bv\|_{\ell_2(\nabla^d)} = \|\pi^{(i)}(\bv)\|_{\ell_2(\nabla)}$, we obtain
\begin{equation}
\label{eq:basic_approx_est}
 \norm{\bS^{-1} (\bT  - \tilde{\bT})\bS^{-1} \bv}  \leq
  \sum_{n_1,p} \varepsilon^{(1)}_{n_1,p}
      \norm{\Restr{\Lambda^{(1)}_{n_1,[p]}} \pi^{(1)}(\bu) } +\ldots 
    + \sum_{n_d,p} \varepsilon^{(d)}_{n_d,p} \norm{\Restr{\Lambda^{(d)}_{n_d,[p]}} \pi^{(d)}(\bv) }
\end{equation}
where
\begin{align*}
 \varepsilon^{(1)}_{n_1,p} &:= \Bignorm{\bS^{-1} \,(\bT^{(1)}_{n_1} - \tilde{\bT}^{(1)}_{n_1,[p]} ) \Restr{\Lambda^{(1)}_{n_1,[p]}} 
    \otimes  \Bigl( \sum_{\kk{\check{n}}_1} c_{\kk{n}} \bT^{(2)}_{n_2} \otimes \cdots\otimes \bT^{(d)}_{n_d} \Bigr)\, \bS^{-1}}  \\
    & \vdots \\
  \varepsilon^{(d)}_{n_d,p} &:= 
    \Bignorm{\bS^{-1}\,  \Bigl( \sum_{\kk{\check{n}}_d} c_{\kk{n}} \tilde{\bT}^{(1)}_{n_1} \otimes \cdots\otimes \tilde{\bT}^{(d-1)}_{n_{d-1}} \Bigr) \otimes (\bT^{(d)}_{n_d} - \tilde{\bT}^{(d)}_{n_d,[p]} ) \Restr{\Lambda^{(d)}_{n_d,[p]}}\, \bS^{-1}}  \,.
\end{align*}
To derive specific  bounds for the quantities  $\varepsilon^{(i)}_{n_i,p}$ we exploit  the structure of $\bT$ and how the global scaling operator
$\bS$ relates to the low-dimensional factors $\bT^{(i)}_{n_i}$. In fact, note that by \eqref{coefficients}, in each summand at most two tensor factors 
are different (up to scaling) from the identity so that we are in the situation of Lemma \ref{lem:tensor-scaling}.
Specifically, for $t=0$ we infer from \eqref{tensorscale} that
\begin{gather*} 
\bignorm{\bS^{-1}\bigl[ (\bT_2 - \mathbf{\tilde T}^{(1)}_2) \otimes \id_2 \otimes \cdots\otimes\id_d  \bigr] \bS^{-1} } \leq \bignorm{ \hatbS^{-1}_1(\bT_2 - \mathbf{\tilde T}^{(1)}_2)\hatbS^{-1}_1 }  \,, \\
\bignorm{\bS^{-1}\bigl[ (\bT_3 - \mathbf{\tilde T}^{(1)}_3) \otimes \bT_4\otimes\id_3 \otimes \cdots\otimes\id_d  \bigr] \bS^{-1} } \leq \bignorm{ (\bT_3 - \mathbf{\tilde T}^{(1)}_3)\hatbS^{-1}_1} \bignorm{\hatbS^{-1}_2 \bT_4 }  \,.
\end{gather*}
Using these estimates with suitable permutations of coordinates, we obtain
\beqn %
\label{epsest}
\begin{array}{c}
 \varepsilon^{(i)}_{1,p} = 0\,,\quad 
  \varepsilon^{(i)}_{2,p} \leq \abs{a_{ii}} \bignorm{\hatbS_i^{-1} (\bT_2 - \bT_{2,J-p}) \hatbS_i^{-1}}  \,,  \\[2mm]
  \varepsilon^{(i)}_{3,p} \leq \sum_{j\neq i} \abs{a_{ij}}   
      \max\bigl\{ \bignorm{\hatbS_j^{-1} \bT_4 } , 
           \bignorm{\hatbS_j^{-1} \mathbf{\tilde T}^{(j)}_{4} } \bigr\}  \, \bignorm{ (\bT_3 - \bT_{3,J-p}) \hatbS^{-1}_i} \,,
\end{array} %
\eeqn
as well as an analogous estimate for $\varepsilon^{(i)}_{4,p}$.
Now, recall that by \eqref{eq:approxfactor_unifest} $ \bignorm{\hatbS_j^{-1} \mathbf{\tilde T}^{(j)}_{4} }\le 2
 \bignorm{\hatbS_j^{-1} \bT_4 }$ which is used in the definition of \eref{eq:maxsequences-0}. We then
combine \eref{epsest} and \eref{eq:supps} with \eqref{eq:diff_scompr} to infer from\eqref{eq:basic_approx_est} 
that   %
\begin{multline*}     
  \quad\norm{\bA_c\bv - \mathbf{\tilde A}_{c,J}\bv} 
\leq 
    \sum_{i=1}^d C^{(i)}_\bA  \Bigl[  
     \sum_{p=0}^J \Bigl(\sum_{n=2}^R \beta_{J-p}(\bA^{(i)}_n)  \Bigr) 2^{-s(J-p)}   \norm{\Restr{\Lambda^{(i)}_{[p]}} \pi^{(i)}(\bv)}\nonumber \\
      +  \sum_{n=2}^R \norm{{\bA}^{(i)}_n} 
         \norm{\Restr{\Lambda^{(i)}_{[J+1]}} \pi^{(i)}(\bv)} \Bigr] = e_J(\bv) ,  \quad
 \end{multline*}
 which, in view of\eref{eJv}, confirms
  the bound \eqref{aposteriori}.
Moreover, on account of the choice of the sets $\Lambda^{(i)}_{[p]}$ and since $\pi^{(i)}(\bv)\in \cA^s$, we have $\norm{\Restr{\Lambda^{(i)}_{[p]}} \pi^{(i)}(\bv)} \leq
(1+2^{-s})2^{-s(p-1)} \norm{\pi^{(i)}(\bv)}_{\As}$, which gives 
\begin{align}  
\label{AJest}     
  \norm{\bA_c\bv - \mathbf{\tilde A}_{c,J}\bv}   &\leq 
    \sum_{i=1}^d C^{(i)}_\bA  \left\{  
     \sum_{p=0}^J  \Bigl(\sum_{n=2}^R \beta_{J-p}(\bA^{(i)}_n)  \Bigr) 2^{-s(J-p)} (1+2^{-s})2^{-s(p-1)} \norm{\pi^{(i)}(\bv)}_{\As}\right.\nonumber \\
    &\left.\quad  +  \sum_{n=1}^R  2^{-s J} \norm{{\bA}^{(i)}_n}  %
         \norm{\pi^{(i)}(\bv)}_{\As}  \right\}  \nonumber \\
      &\leq 2^{-sJ} (2^{s}+2) \sum_{i=1}^d C^{(i)}_\bA \Bigl( \sum_{n=2}^R \norm{\bA^{(i)}_n} \Bigr)   \norm{\pi^{(i)}(\bv)}_{\As} \,.
\end{align}
Finally, as in \cite{BD}, the estimate
$$   \#\supp_i \mathbf{\tilde A}_{c,J} \mathbf{v} 
    \leq  \sum_{n=1}^{R} (\hat\alpha^{(i)}_{J} 2^{J} 2^0
     + \hat\alpha^{(i)}_{J-1} 2^{J-1} 2^1 
     + \ldots + \hat\alpha^{(i)}_0 2^0 2^J)   $$
yields  \eqref{eq:full_opapprox_support}.
\end{proof}

The next  step is to infer compressibility of $\bA$ from compressibility of  $\bA_c$. 

\begin{proof}[Proof of Proposition \ref{prop:weta}]
Let for a given $\bv\in\ell_2(\nabla^d)$ the compressed operator $\tbT=\tbT_J(\bv)$ be defined by \eref{eq:sc_approx_operator}.
Using \eref{tildeswitch}
with $\bB=\bT - \tbT_J$, we obtain 
\begin{eqnarray}
\label{compare}
\norm{\tbS^{-1}(\bT - \tbT_J)\tbS^{-1}\bv} &\leq & (1+\delta)\norm{\bS^{-1}(\bT - \tbT)\bS^{-1}(\bS\tbS^{-1}\bv)} \nonumber\\
&= &(1+\delta)\norm{(\bA_c  - \tilde\bA_{c,J})
(\bS\tbS^{-1}\bv)}.
\end{eqnarray}
Since for $\tilde\bv := \bS\tbS^{-1}\bv$, by Remark \ref{rem:allthesame}, one has $\norm{\Restr{\Lambda^{(i)}_{[p]}} \pi^{(i)}(\tilde\bv)} \leq (1+\delta)\norm{\Restr{\Lambda^{(i)}_{[p]}}
\pi^{(i)}(\bv)}$, we conclude that for
\beqn
\label{tildecompress}
\tilde\bA_J := \tbS^{-1}\tbT_J\tbS^{-1}
\eeqn
one has
\beqn
\label{compare2}
\norm{\bA\bv - \tilde\bA_J\bv}\leq \tilde e_J(\bv) := (1+\delta)^2 e_J(\bv),
\eeqn
where $e_J(\bv)$ is the bound from \eref{aposteriori}, defined in \eref{eJv}. Thus, the same a-posteriori bounds as in the case of canonical scalings can be used to 
make $\norm{\bA\bv - \tilde\bA_J\bv}$ as small as necessary by increasing $J$.

As $\tilde\bA_J$ still has infinite rank, the next step is to replace $\tbS$ by $\tbS_n$, where $n$ depends on the support of $\bv$.
Specifically, given the target accuracy $\eta >0$, we fix 
\beqn
\label{defs}
J=J(\eta), \quad T=T(J(\eta);\bv), \quad n =m(\eta;\bv)=M(\zeta; T),\quad \zeta :=c(\bv)\eta,
\eeqn
defined in \eqref{Tv}, \eqref{etaJ}, \eqref{nsize}.
Invoking Lemma \ref{lem:finite-n}, \eref{alt-bound} with $\eta$ replaced by $\zeta$ yields
\begin{eqnarray*}
\norm{\bA \bv - \tbS^{-1}_n\tbT_J\tbS^{-1}_n\bv}&\leq & \norm{(\bA- \tilde\bA_{J}) \bv}
+ \norm{(\bA- \tilde\bA_{J})(\id - \tbS\tbS_n^{-1}) \bv}\nonumber\\
&& + \frac{\zeta}{1-\delta}  \|(\bA - \tilde\bA_J)(\tbS\tbS^{-1}_n\bv)\| +\frac{2\zeta}{1-\delta} \|\bA\|\,\|\bv\| \nonumber \\ %
&\leq & (1+\delta)^2\Big( e_J(\bv) + e_J\big((\id - \tbS\tbS_n^{-1}) \bv\big) +\frac{\zeta}{1-\delta}e_J(\tbS\tbS^{-1}_n\bv)
\Big) \nonumber\\
&& + \frac{2\zeta}{1-\delta} \|\bA\|\,\|\bv\|,
\end{eqnarray*}
where we have used \eref{compare2} in the last step.
Since for $T$ as in \eqref{defs}, $(\tbS\tbS^{-1}_n)_\nu\leq 1$, $\nu\in \Lambda_{T}$ and recalling \eqref{I-S},
 we conclude that for $n=m(\eta;\bv), J=J(\eta), \zeta =c(\bv)\eta$, defined by \eqref{etaJ}, \eqref{nsize},
\beqn
\label{compare4}
\norm{\bA \bv - \tbS^{-1}_n\tbT_J\tbS^{-1}_n\bv}\leq  (1+\delta)^2  e_J(\bv)\Big(1+ \frac{2\zeta}{1-\delta}\Big)
+  \frac{2\zeta}{1-\delta} \|\bA\|\,\|\bv\|.
\eeqn
Note that whenever $\zeta \leq (1-\delta)/2$, which holds by the condition $\eta\leq 2\norm{\bA}\norm{\bv}$ required prior to \eqref{etaJ}, 
we infer from the definition of $J=J(\eta)$ that
the first summand on the right hand side of \eqref{compare4} is bounded by $\eta/2$.
By definition of $\zeta$ in \eqref{defs} and \eqref{etaJ}, the second summand is also bounded by $\eta/2$,
which completes the proof of Proposition \ref{prop:weta}.
\end{proof}

For the following proof, we introduce additional auxiliary notation, complementing $\Sci{i}$ defined in \eqref{Si}: we denote by $\check{\bS}_i\colon \spl{2}(\nabla^{d-1})\to \spl{2}(\nabla^{d-1})$ the rescaling operator with the $i$-th coordinate omitted, that is, 
\begin{equation}\label{eq:checkSi}
\bigl(\check{\bS}_i \bv \bigr)_\nu  = \Bigl( \sum_{j < i} (\omi{i}{\nu_i})^2 + \sum_{j > i} (\omi{i}{\nu_{i-1}})^2\Bigr)^{\frac12} \,v_\nu  \qquad \text{for $\bv\in \R^{\nabla^{d-1}}$, $\nu\in \nabla^{d-1}$.}
\end{equation}

\begin{proof}[Proof of Theorem \ref{thm:apply}]
The first claim  \eqref{eq:approx-eta} of Theorem \ref{thm:apply} has already been established above.
To verify \eref{eq:tensor_apply_support} we make use of \eref{eq:full_opapprox_est} and the fact that
the support of $\bw_\eta$ is independent of the particular scaling and hence is given by
  $\tilde\bA_{J(\eta)} \bv$. Clearly, in view of \eref{AJest} and  \eref{etaJ}, one has
  $J(\eta)\leq \bar J(\eta)$ with
\begin{equation*}
\bar J(\eta):= {\rm argmin}\,\biggl\{J\in \N: 
2^{-sJ} (2^{s}+2) R \sum_{i=1}^d C^{(i)}_\bA \max_{n>1} \norm{\bA^{(i)}_n}\,   \norm{\pi^{(i)}(\bv)}_{\As} \leq \frac\eta{4(1+\delta)^2} \biggr\} \,,
\end{equation*}
which yields
\beqn
\label{barJeta}
 J(\eta) \leq \left\lceil \log_2\Big( \eta^{-1/s}\Big(4(1+\delta)^2  (2^{s}+2) R \sum_{i=1}^d C^{(i)}_\bA \max_{n>1} \norm{\bA^{(i)}_n}\,   \norm{\pi^{(i)}(\bv)}_{\As}
 \Big)^{\frac1s}\Big)
\right\rceil .
\eeqn
Inserting this into  \eref{eq:full_opapprox_support} yields \eref{eq:tensor_apply_support}.

To prove \eqref{sparsity-stable} we reduce the problem to the setting considered in \cite{BD}
by appropriate estimates for the rescaling operators $\bS^{-1}$. 
It suffices  to discuss the case $i = 1$. 
Note first that for $n$, $J$ as in \eqref{compare4}, as a consequence of \eqref{low-up}, for $\tilde\bv := \Sc \Sa{n}^{-1} \bv$ and $\nu_1 \in \supp_1 \bw_\eta$ we have
\begin{equation}\label{eq:piscalingchange}  \pi^{(1)}_{\nu_1} (\bw_\eta) = \pi^{(1)}_{\nu_1}(\Sa{n}^{-1} \tilde \bT_J \Sa{n}^{-1} \bv) \leq (1+\delta) \pi^{(1)}_{\nu_1}(\Sc^{-1} \tilde \bT_J \Sc^{-1} \tilde\bv)\,. 
\end{equation}
We exploit again the specific structure of $\bT$ given by
\eqref{coefficients}, which in particular means that $\bT^{(i)}_1=\id$ and in each summand  at most two factors are different from the identity.
Recalling the notation \eqref{eq:delentry} and $\check{\bS}^{-1}_1$ as introduced in \eqref{eq:checkSi}, we obtain in view of \eref{Btensor} and \eref{triangleeq}, 
\begin{multline}
\label{eq:opapprox_As_1}
  \pi^{(1)}_{\nu_1}(\Sc^{-1} \tilde \bT_J \Sc^{-1}\tilde\bv)  \leq  \pi^{(1)}_{\nu_1} \Bigl( \id \otimes \check{\bS}^{-1}_1 \sum_{\kk{n}\in{\KK{{d}}}(1,R,\ldots,R)} c_{\kk{n}} \bigotimes_{i=2}^d  \mathbf{\tilde T}^{(i)}_{n_i} \bS^{-1} \tilde\bv \Bigr) \\
   +   \sum_{n=2}^R  \pi^{(1)}_{\nu_1} \Bigl(\bS^{-1}\mathbf{\tilde T}^{(1)}_{n} \otimes \sum_{\kk{n}\in{\KK{{d}}}(1,R,\ldots,R)} c_{\kk{\check n}_1|_n} \bigotimes_{i=2}^d  \mathbf{\tilde T}^{(i)}_{n_i} \bS^{-1} \tilde\bv  \Bigr) =: D_{1,\nu_1} + \sum_{n=2}^R D_{n,\nu_1} \,.
\end{multline}
To bound $D_{1,\nu_1}$ we estimate  
\begin{align*}
 \biggnorm{\check{\bS}^{-1}_1 \sum_{\kk{n}\in{\KK{{d}}}(1,R,\ldots,R)} c_{\kk{n}} \bigotimes_{i=2}^d  \mathbf{\tilde T}^{(i)}_{n_i} \check{\bS}^{-1}_1}
&\leq    \sum_{\kk{n}\in{\KK{d}}(1,R,\ldots,R)} \abs{c_{\kk{n}}}
\bignorm{ \check{\bS}^{-1}_1 \bigotimes_{i=2}^d  \mathbf{\tilde T}^{(i)}_{n_i} \check{\bS}^{-1}_1} \,, \nonumber
\intertext{and recall from \eqref{coefficients} that at most two factors in the tensor products on the right hand side differ from the identity.
Invoking again Lemma \ref{lem:tensor-scaling}, and bearing
  \eqref{eq:approxfactor_unifest} in mind, we conclude that}
  \biggnorm{\check{\bS}^{-1}_1 \sum_{\kk{n}\in{\KK{{d}}}(1,R,\ldots,R)} c_{\kk{n}} \bigotimes_{i=2}^d  \mathbf{\tilde T}^{(i)}_{n_i} \check{\bS}^{-1}_1} &\leq  
4 \max_{i,n_i}\norm{\bA^{(i)}_{n_i}}^2
  \sum_{\kk{n}\in{\KK{{d}}}(1,R,\ldots,R)} \abs{c_{\kk{n}}}  \leq
 \check{C}^{(1)}_\bA \,,
 \end{align*}
where in the last step we have used that $(a_{ij})$ is diagonally dominant. Hence, by \eqref{Btensor} and since the entries of the diagonal operators $(\id\otimes\check{\bS}_1 ) \bS^{-1}$ are bounded by one, one has
\begin{equation*}  
 D_{1,\nu_1} \leq  \check{C}^{(1)}_\bA \pi^{(1)}_{\nu_1} \bigl((\id \otimes\check{\bS}_1 )\bS^{-1}\tilde\bv\bigr)  
 \leq    \check{C}^{(1)}_\bA  \pi^{(1)}_{\nu_1} (\tilde\bv) \,.
\end{equation*}
Regarding $D_{n,\nu_1}$ for $n>1$, we have the estimates
\begin{equation*}
 D_{n,\nu_1} \leq \begin{cases}
  \displaystyle a_{11} \,\pi^{(1)}_{\nu_1} \bigl([ \tilde \bA^{(1)}_{2}  \otimes \id\otimes\cdots\otimes\id] (\Sc_1 \Sc^{-1} \tilde\bv)\bigr),  & n =2\, ,\\[6pt]
  \displaystyle \sum_{j>1} \abs{a_{1j}} \,\pi^{(1)}_{\nu_1} \bigl( [ \tilde\bA^{(1)}_{3}\otimes\id\otimes\cdots\otimes\id \otimes \tilde\bA^{(j)}_{4} \otimes\id\otimes\cdots\otimes\id ] (\Sc_1 \Sc^{-1} \tilde\bv) \bigr) , & n =3\,, \\[9pt]
  \displaystyle \sum_{j>1} \abs{a_{1j}}\, \pi^{(1)}_{\nu_1} \bigl( [\tilde\bA^{(1)}_4 
  \otimes\id\otimes\cdots\otimes\id \otimes  \tilde\bA^{(j)}_{3} \otimes\id\otimes\cdots\otimes\id] (\Sc_j \Sc^{-1} \tilde\bv)  \bigr) , & n =4 \,.
 \end{cases}
\end{equation*}
By \eqref{Btensor}, we obtain for $j=2,\ldots,d$,
\begin{multline*}
\pi^{(1)}_{\nu_1} \bigl( [ \tilde\bA^{(1)}_{3}\otimes\id\otimes\cdots\otimes\id \otimes \tilde\bA^{(j)}_{4} \otimes\id\otimes\cdots\otimes\id ] (\Sc_1 \Sc^{-1} \tilde\bv) \bigr)  \\
  \leq \norm{\tilde\bA^{(j)}_{4} }  \,
  \pi^{(1)}_{\nu_1} \bigl( [ \tilde\bA^{(1)}_{3}\otimes\id\otimes\cdots\otimes\id ] (\Sc_1 \Sc^{-1} \tilde \bv) \bigr) 
\end{multline*}
as well as
\begin{multline*}
  \pi^{(1)}_{\nu_1} \bigl( [\tilde\bA^{(1)}_4 
   \otimes\id\otimes\cdots\otimes\id \otimes  \tilde\bA^{(j)}_{3} \otimes\id\otimes\cdots\otimes\id] (\Sc_j \Sc^{-1} \tilde\bv)  \bigr) \\
 \leq  \norm{ \tilde\bA^{(j)}_{3}} \, \pi^{(1)}_{\nu_1} \bigl( [\tilde\bA^{(1)}_4 
      \otimes\id\otimes\cdots\otimes\id] (\id \otimes \check{\bS}_1) \Sc^{-1} \tilde \bv  \bigr).
\end{multline*}

Note next that the entries of the diagonal operators $\Sc_1 \bS^{-1}$ are bounded by one as well, and therefore $\pi^{(1)}_\nu(\Sc_1 \bS^{-1}\tilde\bv) \leq \pi^{(1)}_\nu(\tilde\bv)$ for all $\nu\in\nabla$.
We can thus follow the lines of the proof of \cite[Theorem 8]{BD} to infer that, in particular,
\begin{equation}
\label{eq:opapprox_compAs}
\norm{\pi^{(1)}\big((\tilde{\bA}^{(1)}_n \otimes \id\otimes \cdots\otimes \id) (\Sc_1 \bS^{-1}\tilde\bv)\big)}_\As
   \leq  \frac{2^{3s+2}}{2^s - 1} \norm{\hat{\alpha}}_{\spl{1}}^s \bigl( 2 %
   \norm{\bA^{(1)}_n}  \bigr)
   \norm{\pi^{(1)}( \tilde\bv)}_\As  
\end{equation}
for $n=2,3,4$, where we have made use of \eqref{eq:approxfactor_unifest}. Moreover, \eqref{eq:opapprox_compAs} holds with $\Sc_1 \bS^{-1}\tilde\bv$ replaced by $(\id \otimes\check{\bS}_1 )\bS^{-1}\tilde\bv$ as well. For $n=2$, $2 \norm{\bA^{(1)}_2}$ appears as a single factor
in the bound for $D_{2,\nu_1}$, while for $n=3,4$ two such factors arise.  Recalling the definition of $C^{(i)}_\bA$ in \eqref{eq:maxsequences-0}, we cover all cases by the bound
\begin{equation} \label{eq:Dnest}
  \norm{D_{n,\cdot}}_\As \leq  C^{(i)}_\bA \frac{2^{3s+2}}{2^s - 1} \norm{\hat{\alpha}}_{\spl{1}}^s \max_{m> 1}
  \norm{\bA^{(1)}_m}  
   \norm{\pi^{(1)}( \tilde\bv)}_\As    \,,\quad n=2,3,4 \,.
\end{equation}
 We now use \eref{eq:opapprox_As_1} to combine these estimates, obtaining
$$
   \norm{\pi^{(1)}_{\nu_1} (\Sc^{-1} \tilde \bT_J \Sc^{-1} \bv)}_\As \leq 
      R^s\sum_{n=1}^R \norm{D_{n,\cdot}}_\As .
$$
Finally, $\norm{\pi^{(1)}( \tilde\bv)} \leq (1+\delta) \norm{\pi^{(1)}( \bv)}$ by Remark \ref{rem:allthesame}, which we use in \eqref{eq:Dnest}, and with \eqref{eq:piscalingchange} we arrive at the desired bound for $i=1$.
Analogous bounds for $i=2,\ldots,d$, are obtained in the same way, confirming \eqref{sparsity-stable}.

The rank bound \eqref{eq:tensor_apply_ranks} follows from \cite[Theorem 8, (99)]{BD}, taking into account
that the ranks $\rank_\alpha(\bv)$ of $\tbS^{-1}_{m(\eta;\bv)}\bv$ can be bounded by $m(\eta;\bv)\rank_\alpha$, $\alpha
\in \cD_d$, and that the application of $\tbS^{-1}_{m(\eta;\bv)}$ to $\tbT_{J(\eta)}\tbS^{-1}_{m(\eta;\bv)}\bv$ causes another
multiplication by  $m(\eta;\bv)$. Likewise, the estimate \eqref{eq:flops} of the computational complexity follows from
the previous observation combined with \eqref{eq:tensor_apply_support} and \cite[Remark 12]{BD}.

To prove \eqref{metav}, we need to estimate $L(\tbT_{J(\eta)}\tilde\bv)$, where $\tilde\bv := \tbS^{-1}_{m(\eta;\bv)} \bv$. Note that $L(\tilde\bv) = L(\bv)$, and by the level decay property of the approximations of lower-dimensional component operators, we thus obtain $L(\tbT_{J(\eta)}\tilde\bv) \leq L(\bv) + C_1(\bA,s) J(\eta)$. From \eqref{barJeta} we know that 
$$
J(\eta)\leq \frac 1s\Big( \abs{\log_2\eta} + \ln \big(C_2(\bA,s)\sum_{i=1}^d \norm{\pi^{(i)}(\bv)}_\As\big)\Big).
$$
Moreover, we have 
$$\om{\nu} \leq \sqrt{d} \max_{i=1,\ldots,d} \omi{i}{\nu_i} \leq c \sqrt{d} \max_{i=1,\ldots,d} 2^{\abs{\nu_i}}\,, $$ 
where $c = \max_{\nu\in\nabla^d} \max_{i} 2^{-\abs{\nu_i}}\,\omi{i}{\nu_i}$. Hence, for an index $\nu\in \nabla^d$ to belong
to $\Lambda_{T }$ as in \eqref{LambdaT}, since $\omega_\text{min} \geq \sqrt{d}\, \hatomin$, it is sufficient that 
$   c^2 \max_{i} 2^{2\abs{\nu_i}} \leq   T  (\hatomin)^2 $.
Consequently, $\Lambda_T$ contains $\tilde\bT_{J(\eta)}\tilde\bv$ if 
$$L(\bv) + C_1(\bA,s) J(\eta) \leq  \frac12 \log_2 T + \log_2 c^{-1} \hatomin .$$
The assertion \eqref{metav} now follows from \eqref{nsize}, which in turn uses \eqref{Nsc}.
In the latter, it thus remains to estimate $\abs{\ln (\min \{ \delta/2,c(\bv)\eta  \})}$, where $c(\bv)\eta = \frac12(1-\delta) \min\{1, \eta/(2\norm{\bA}\norm{\bv})\}$. Hence $\abs{\ln c(\bv)\eta} \leq C_3(\bA, \delta) +  \abs{\ln \eta} + \max\{ 0, \ln \norm{\bv} \} $, where $\norm{\bv} = \norm{\pi^{(i)}(\bv)} \leq \norm{\pi^{(i)}(\bv)}_\As$ for $i=1,\ldots,d$, providing \eqref{metav}. 
\end{proof}

\subsection{Control of Rank Growth}

The ranks arising in the procedure for applying operators introduced in the previous section depend on the range of values that the approximate scaling sequence needs to cover. In the case of wavelet bases, this is directly related to the maximum currently active wavelet level.

The following lemma gives a bound for the maximum possible active level  that can occur in the output of $\coarsen(\bv;\varepsilon)$. It depends both on some additional higher regularity (expressed by a bound on the quantities $\norm{\Sci{i}^t \pi^{(i)}(\bv))}$) and on the sizes of the lower-dimensional supports $\supp_i(\bv)$. This bound will subsequently be used in conjunction with Theorem \ref{thm:apply}(v).

\begin{lemma}
\label{lmm:lvlbound}
For given $\bv\in\spl{2}(\nabla^d)$, we consider $\mathbf{p} := (\pi^{(i)}_{\nu}(\bv))_{(i,\nu)}$ as a vector on $\Ical := \{1,\ldots,d\}\otimes \nabla$. 
Assume that
$$ 
   \#\supp \mathbf{p} = \sum_{i=1}^d \#  \supp_i(\bv) < \infty
$$ 
and that for some $t>0$ one has $\norm{\Sci{i}^t \pi^{(i)}(\bv))} <\infty$ for all $i=1,\ldots,d$. Let $\varepsilon > 0$ and let $\mathbf{p}_\varepsilon$ be the vector of minimal support in $\Ical$ such that $\norm{\mathbf{p} - \mathbf{p}_\varepsilon}_{\spl{2}(\Ical)}\leq\varepsilon$. 
Let $C_\omega^{(i)} := \sup_{\mu\in\nabla} \omi{i}{\mu}^{-t}\, 2^{t\abs{\mu}} $.
Then for all $(i,\nu)\in \supp \mathbf{p}_\varepsilon$ one has
$$   
\abs{\nu} \leq t^{-1} \log_2 \Bigl[ \varepsilon^{-1} \,C^{(i)}_\omega \,\norm{\Sci{i}^t \pi^{(i)}(\bv))} \sqrt{\#\supp \mathbf{p}}\Bigr] .
$$
\end{lemma}

\begin{proof}
Let $C_i:= C_\omega^{(i)} \norm{\hatbS^t_i \pi^{(i)}(\bv))}$ and $N:= \#\supp \mathbf{p}$.
Suppose that $(i,\mu) \in\supp \mathbf{p}_\varepsilon$ and $\abs{\mu} > t^{-1} (\log_2 C_i\sqrt{N} - \log_2 \varepsilon)$.
It follows that
$$
\abs{\pi^{(i)}_\mu(\bv)} \leq \norm{\hatbS^t_i \pi^{(i)}(\bv))} \,\omi{i}{\mu}^{-t} \leq C_i 2^{-t \abs{\mu}} < C_i (C_i\sqrt{N})^{-1} \varepsilon = \frac{\varepsilon}{\sqrt{N}}\,.$$
Let $\hat\Lambda := \supp \mathbf{p} \setminus \supp\mathbf{p}_\varepsilon$. Then necessarily, $\abs{\pi^{(j)}_\nu(\bv)}\leq \abs{\pi^{(i)}_\mu(\bv)}$
holds  for all $(j,\nu) \in\hat\Lambda$ and thus
$$   \sum_{(j,\nu)\in\hat\Lambda\cup \{ (i, \mu)\}} \abs{\pi^{(j)}_\nu(\bv)}^2  <   N \frac{\varepsilon^2}{N} \leq \varepsilon^2\,,    
$$
  contradicting the definition of $\mathbf{p}_\varepsilon$.
\end{proof}

We shall apply the above lemma to the result of line \ref{alg:tensor_solve_innerrecomp} in Algorithm \ref{alg:tensor_opeq_solve}. 
There the value of $\varepsilon$ in the lemma corresponds to $\eta_{k,j} = \rho^{j+1}2^{-k} \delta$ and $\mathbf{p}_\varepsilon$ in the lemma is the result of $\coarsen$ in the algorithm.
We note that, as a consequence of \eqref{dim-sorting} and \eqref{dim-sorting-optimality}, this routine indeed yields $\mathbf{p}_\varepsilon$ with precisely the properties required in Lemma \ref{lmm:lvlbound}. In order to obtain the desired bounds for the maximum active wavelet levels in our iterates $\bw_{k,j}$, we   still need suitable bounds for $\norm{\Sci{i}^t \pi^{(i)}(\bw_{k,j})}$.

\subsection{Control of Higher Regularity}

\begin{lemma}\label{lmm:crs_stab}
For any $t>0$ and $\eta >0$, we have
\begin{equation}\label{eq:crs_stab_i}
   \norm{\hatbS^t_i \pi^{(i)}( \hatCctr{\eta} \bv) } \leq \norm{ \hatbS^t_i \pi^{(i)}(\bv)}\,,\quad
     \norm{\hatbS^t_i \pi^{(i)}( \hatPsvd{\eta} \bv) } \leq \norm{ \hatbS^t_i \pi^{(i)}(\bv)}\,,\quad i = 1,\ldots, d \,,
\end{equation}
for any $\bv \in \e2$.
\end{lemma}

\begin{proof}
The first inequality in \eqref{eq:crs_stab_i} is clear, the second is an immediate consequence of the componentwise estimate \eqref{Ppieta} for $\pi^{(i)}(\bv)$.
\end{proof}

We now consider the evolution of $\norm{\Sci{i}^t \pi^{(i)}( \bw_{k,j})}$, with $\bw_{k,j}$ defined in Algorithm \ref{alg:tensor_opeq_solve}. 
Note that by our excess regularity assumptions on $\bA$ and $\bbf$, we know that $\max_i\norm{\Sc_{i}^t  \bbf} < \infty$ as well as 
\begin{equation}\label{excessxi}
    \xi := \max_{\substack{i=1,\ldots,d\\ n=2,3,4}} \norm{\Sci{i}^t \bA^{(i)}_{n} \Sci{i}^{-t}}  <  \infty \,.
\end{equation}

\begin{proposition}
Under the assumptions of Theorem \ref{thm:complexity}, the iterates $\bw_{k,j}$ of Algorithm \ref{alg:tensor_opeq_solve} satisfy
\begin{equation}
\label{w_reg_bound}
   \norm{\Sci{i}^{t} \pi^{(i)}(\bw_{k,j})} \leq \frac{\gamma^{kI + j+1} - 1}{\gamma - 1} \bar C_\bbf \,,
\end{equation}
where
$$
\gamma := 1 + \omega (1+\delta)^2 \Bigl[ \check C^{(i)}_\bA + C^{(i)}_\bA R\, \bigl(\xi + C_t \norm{\bA^{(i)}_n} \bigr)    \Bigr]  \,,\quad
\bar C_\bbf := \omega \,C^{\rm reg}_\bbf  \max_{i} \norm{ \Sc_i^t \bbf } \,.
$$
\end{proposition}

Note that under Assumptions \ref{ass:dim}, $\bar C_\bbf$ as well as the quantities arising in the definition of $\gamma$ are independent of $d$, except for $\check{C}^{(i)}_\bA$, which by \eqref{checkAi-bound} grows at most linearly in $d$.

\begin{proof}
Note that for each outer loop index $k$, its inner loop over $j$ can be summarized as
$$ 
 \bw_{k,j+1} = \hatCctr{\beta_2\eta_{k,j}} \hatPsvd{\beta_1\eta_{k,j}} \bigl[ (\id - \omega \tilde\bA_{k,j}) \bw_{k,j} + \omega \bbf_{k,j} \bigr] .
$$
Here, abbreviating $\eta:= \frac12 \eta_{k,j}$,  we recall that $\tilde\bA_{k,j} := \Sa{m(\eta;\bw_{k,j})}^{-1} \tilde{\bT}_{J(\eta)}  \Sa{m(\eta;\bw_{k,j})}^{-1}$, as in \eqref{weta-def}, and $\bbf_{k,j} := \rhs(\eta)$. Moreover,
  by step 1 in Algorithm \ref{alg:tensor_opeq_solve}, we have  $\norm{\Sci{i}^{t}\pi^{(i)}(\bu_0)} =0$, for each $i=1,\ldots,d$,
  and  \eqref{eq:rhsreg} implies $\norm{\Sci{i}^{t} \pi^{(i)}(\bbf_{k,j})} \leq C^{\rm reg}_\bbf \norm{\Sci{i}^{t} \pi^{(i)}(\bbf)}$. We shall repeatedly use that $\norm{\Sci{i}^{t} \pi^{(i)}(\bv)} = \norm{\Sc_i^{t} \bv}$ for any $\bv$.

Using Lemma \ref{lmm:crs_stab}, we obtain
\begin{align}
  \norm{\Sci{i}^{t} \pi^{(i)}(\bw_{k,j+1}) } 
      &= \norm{\Sci{i}^{t} \pi^{(i)} (\hatCctr{\beta_2\eta_{k,j}} \hatPsvd{\beta_1\eta_{k,j}}
           [ (\id - \omega \tilde\bA_{k,j}) \bw_{k,j} + \omega \bbf_{k,j} ] ) }\notag \\
    &\leq \norm{\Sci{i}^t \pi^{(i)}( \bw_{k,j})} +  \omega \norm{\Sci{i}^{t} \pi^{(i)}  ( \tilde\bA_{k,j} \bw_{k,j})} + \omega \norm{\Sci{i}^{t} \pi^{(i)}  ( \bbf_{k,j} ) } . \label{eq:highersobolev_estimate}
\end{align}
We define now $\tilde\bw_{k,j}:= \Sc\Sa{n}^{-1}\bw_{k,j}$ and argue, for  $\tilde\bA^{(i)}_n$ as in \eqref{eq:tilde1dscmatdef},     
in complete analogy to
  the estimates following \eqref{eq:opapprox_As_1} to conclude that
  \begin{multline*}    
  \pi^{(i)}_{\nu_i}( \tilde\bA_{k,j} \bw_{k,j})) \leq 
    \check C^{(i)}_\bA (1+\delta)^2 \pi^{(i)}_{\nu_i}(\bw_{k,j})  \\
      + C^{(i)}_\bA (1+\delta) \sum_{n=2}^R 
        \pi^{(i)}_{\nu_i} \bigl(\Sc_i^t [\id\otimes\cdots  \id\otimes \tilde\bA^{(i)}_n \otimes \id \cdots \otimes\id ]( \mathbf{D}_i  \tilde\bw_{k,j}) \bigr)\,,
\end{multline*}
where $\mathbf{D}_i = \Sc_i \Sc^{-1} $ for $n=2,3$ and $\mathbf{D}_1 = (\id\otimes\check\Sc_1) \Sc^{-1},\ldots, \mathbf{D}_d = (\check\Sc_d\otimes\id)\Sc^{-1} $ for $n=4$.
We now add and substract $\bA^{(i)}_n$ from \eqref{eq:1dscmatdef} in the last summands, apply $\Sci{i}^t$, sum over $\nu_i$, and use \eqref{excessxi} as well as Remark \eqref{rem:StildeS} to obtain
\begin{multline*}
  \norm{\Sci{i}^{t} \pi^{(i)}  ( \tilde\bA_{k,j} \bw_{k,j})}  
     \leq     (1+\delta)^2  \bigl( \check C^{(i)}_\bA \norm{\Sci{i}^t \pi^{(i)}(\bw_{k,j})} +   C^{(i)}_\bA  R \xi  \norm{\Sci{i}^{t} \pi^{(i)}  (  \bw_{k,j})}   \bigr)  \\
   +  C^{(i)}_\bA (1 + \delta)  \sum_{n=2}^R 
    \norm{ \Sc_i^t [\id\otimes\cdots  \id\otimes( \bA^{(i)}_n - \tilde \bA^{(i)}_n) \otimes \id \cdots \otimes\id ]( \mathbf{D}_i  \tilde\bw_{k,j})} \,.
\end{multline*}
By Definition \ref{def:sobolev_compr} and \eqref{tildeTni},
$$
 \norm{ \Sc_i^t [\id\otimes\cdots\id\otimes ( \bA^{(i)}_n - \tilde \bA^{(i)}_n) \otimes \id \cdots \otimes\id ] ( \mathbf{D}_i  \tilde\bw_{k,j})}
    \leq C_t  \norm{\beta(\bA^{(i)}_{n_i})}_{\ell_1}  \norm{\Sci{i}^{t} \pi^{(i)}  (  \tilde\bw_{k,j})} \,.
$$
Using in addition  \eqref{betascale}, we thus have
\begin{equation*}
  \norm{\Sci{i}^{t} \pi^{(i)} (\bw_{k,j+1}) } \leq \gamma \norm{\Sci{i}^{t} \pi^{(i)} (\bw_{k,j})} + \bar C_\bbf  \,. 
\end{equation*}
Using $j\leq I$ (see step \ref{alg:jchoice} in Algorithm \ref{alg:tensor_opeq_solve}), we arrive at \eqref{w_reg_bound}.
\end{proof}

\subsection{Proof of the Main Result}

In the following, we make an effort to track the dependence of arising constants on various parameters, in particular on $d$; this is necessarily more technical than what would be needed to present just the essence of the result, which lies mainly in the interplay of Theorem \ref{lmm:combined_coarsening}, Theorem \ref{thm:apply}, and Lemma \ref{lmm:lvlbound}. 

\begin{proof}[Proof of Theorem \ref{thm:complexity}]
Let $\varepsilon_k := 2^{-k} \varepsilon_0$.
Note that \eqref{eq:complexity_rank} and \eqref{eq:complexity_ranknorm} follow from \eqref{eq:combinedcoarsen_rankest} in Theorem \ref{lmm:combined_coarsening}, whereas  \eqref{eq:combinedcoarsen_suppest} yields \eqref{eq:complexity_supp} and \eqref{eq:complexity_sparsitynorm}.  
As a consequence of \eqref{eq:combinedcoarsen_suppest}, we also have
$$
  \sum_i \norm{\pi^{(i)}(\bw_{k,0})}_{\As} \leq C_1d^{1+\max\{1,s\}}  \sum_i \norm{\pi^{(i)}(\bu)}_{\As} \,,
$$
where $C_1$ is a constant independent of $d$. In what follows, newly introduced constants are always independent of $d$ unless stated otherwise.
By Theorem \ref{thm:apply}(ii), we have
\begin{equation}
\label{compl_applystab}
   \bignorm{\pi^{(i)}\bigl(\apply(\bw_{k,j}; \textstyle\frac12\displaystyle \eta_{k,j}\bigr))}_\As \leq C_2 d \bignorm{\pi^{(i)}(\bw_{k,j})}_\As \,.
\end{equation}
In this regard, note that $R$ and $\norm{\hat\alpha}_{\ell_1}$ are, by construction,   independent of $d$ and  that the same holds,
by \eqref{eq:maxsequences-0} combined with \eqref{aij}, for $C^{(i)}_\bA$.   Recall that $\check C^{(i)}_\bA$ grows at most linearly in $d$ by \eqref{checkAi-bound}.
Consequently, 
$$
   \norm{\pi^{(i)}(\bw_{k,j+1})}_\As \leq C_3 d \norm{\pi^{(i)}(\bw_{k,j})}_\As 
      + C_4 \norm{\pi^{(i)}(\bbf)}_\As,
$$
where we may assume without loss of generality that $C_3 d > 1$. Hence for all $k$ and $j$, we have
\begin{multline}
\label{eq:pigrowth}
   \sum_{i=1}^d \norm{\pi^{(i)}(\bw_{k,j})}_\As \leq (C_3d)^j C_1d^{1+\max\{1,s\}} \sum_{i=1}^d \norm{\pi^{(i)}(\bu)}_{\As} \\
       + C_4 \bigl( C_3d - 1 \bigr)^{-1} \bigl( (C_3d)^j - 1 \bigr)  \sum_{i=1}^d \norm{\pi^{(i)}(\bbf)}_{\As} \,.
\end{multline}

As a further consequence of \eqref{eq:combinedcoarsen_suppest} in Theorem \ref{lmm:combined_coarsening}, using $\kappa_1^{-1} \lesssim d$, we also know that
$$
  \sum_{i=1}^d \#\supp_i (\bw_{k,0})  \leq C_5 \,d^{1 + s^{-1}}\, (2^{-k}\varepsilon_0)^{-\frac1s}\Bigl(\sum_{i=1}^d \norm{\pi^{(i)}(\bu)}_{\As}\Bigr)^{\frac1s} \,.
$$
In view of steps 7 and 8 in Algorithm \ref{alg:tensor_opeq_solve},
we infer now from Theorem \ref{thm:apply}(i) and Assumptions \ref{ass:rhs}(\ref{ass:rhsapprox}) that
\begin{multline*}
   \sum_{i=1}^d \#\supp_{i }  (\bw_{k,j+1})  \leq \sum_{i=1}^d \#\supp_i (\bw_{k,j})   
   +  C_6 d \, \eta_{k,j}^{-\frac1s} \Bigl(\sum_{i=1}^d \norm{\pi^{(i)}(\bw_{k,j})}_{\As}\Bigr)^{\frac1s} \\
     + C^{\text{{\rm supp}}} \, d \, \eta_{k,j}^{-\frac1s} \Bigl(\sum_{i=1}^d \norm{\pi^{(i)}(\bbf)}_{\As}\Bigr)^{\frac1s} \,,
\end{multline*}
where, on account of Assumptions \ref{ass:dim}, $C^{\text{{\rm supp}}}$ is independent of $d$.  
The last summand in this bound results from \eqref{eq:tensor_apply_support}, using the same observations as in \eqref{compl_applystab}.  Thus, for any $k$ and $j$, we have
\begin{multline*}
   \sum_{i=1}^d \#\supp_i (\bw_{k,j})  \leq C_5 d^{1+s^{-1}} \, (2^{-k}\varepsilon_0)^{-\frac1s}\Bigl(\sum_{i=1}^d \norm{\pi^{(i)}(\bu)}_{\As}\Bigr)^{\frac1s} \\
     + \sum_{n = 0}^{j-1} ( \rho^{n+1} 2^{-k} \varepsilon_0)^{-\frac1s}  \Bigl[ C_6 d \Bigl(  \sum_{i=1}^d \norm{\pi^{(i)}(\bw_{k,n})}_{\As}\Bigr)^{\frac1s} + C^{\text{{\rm supp}}} d\,  \Bigl(\sum_{i=1}^d \norm{\pi^{(i)}(\bbf)}_{\As}\Bigr)^{\frac1s} \Bigr] \,.
\end{multline*}
Combining this   with \eqref{eq:pigrowth}, defining 
\begin{equation*}
   C_{\bu,\bbf} :=  \max\Bigl\{ \Bigl( \sum_{i=1}^d \norm{\pi^{(i)}(\bu)}_\As \Bigr)^\frac{1}{s} , 
            \Bigl( \sum_{i=1}^d \norm{\pi^{(i)}(\mathbf{f})}_\As \Bigr)^\frac{1}{s} 
                      \Bigr\} \,,
\end{equation*}
and recalling that $\eta_{k,j} = \rho^{j+1}2^{-k} \varepsilon_0$, we arrive at
\begin{equation}
\label{eq:wkjsuppbound}
  \sum_{i=1}^d  \#\supp_i (\bw_{k,j})  \leq   C_7 d^{p_1} %
  \, \bigl(C_3d\bigr)^{\frac{j}s}   \,C_{\bu,\bbf} \eta_{k,j}^{-\frac1s}\,,
\end{equation}
where $p_1:= \max\{2 + s^{-1}, 1 + 2s^{-1}\}$.

We are now in a position to invoke Lemma \ref{lmm:lvlbound}. Here the requirement that $\beta_2>0$ in Algorithm \ref{alg:tensor_opeq_solve} enters. Combining \eqref{eq:wkjsuppbound} with \eqref{w_reg_bound} for $i=1,\ldots,d$, and for each $\nu \in \supp_i(\bw_{k,j}) \subset \nabla$, we conclude that
\begin{equation}
\label{eq:lvlbd1}
  \abs{\nu}  \leq L_{k,j} := t^{-1} \log_2 \Biggl[ C_8d^{p_1}\, \eta_{k,j}^{-1}\, \gamma^{k I + j} \bar C_\bbf \sqrt{ \eta_{k,j}^{-\frac1s}\,  \bigl(C_3d\bigr)^{\frac{j}s} C_{\bu,\bbf} } \Biggr] \,.
\end{equation}
We rewrite this for convenience as 
\beqn
\label{convenience}
L_{k,j} =  t^{-1} \log_2 \Bigl[ C_9(d)\, C_3^{\frac{j}{2s}} d^{p_1+\frac{j}{2s}}\eta_{k,j}^{-1-\frac{j}{2s}}\gamma^{kI+j} \Bigr],
\eeqn
where $C_9(d) := C_8 \bar C_\bbf C_{\bu,\bbf}^{1/2}$, which may depend on $d$ via $C_{\bu,\bbf}$; note that
$C_{\bu,\bbf} \leq d^{\frac1s} \hat C_{\bu,\bbf}$ with 
$ \hat C_{\bu,\bbf} := \max_i\{ \norm{\pi^{(i)}(\bu)}_\As^{1/s}, \norm{\pi^{(i)}(\bbf)}_\As^{1/s}  \}$ which, by Assumptions \ref{ass:dim}, is independent of $d$.

In order to estimate the right hand side in \eref{convenience}, we need a suitable estimate for $\log_2 \gamma^{Ik}$, which contains the outer iteration index $k$.
We will relate this quantity to the current tolerance $\eta_{k,j}$. To this end, note that 
$$
\log_2 \gamma^k =  \bigl(\abs{\log_2 \eta_{k,j}} + j \abs{\log_2 \rho} + \abs{\log_2 \rho\varepsilon_0}\bigr) \log_2\gamma. 
$$
Hence the bound in \eqref{eq:lvlbd1} can be rewritten in the form
\begin{align*}
t L_{k,j} &\leq \log_2 C_9(d) + \frac{j}{2s} \log_2 C_3 + \Big(p_1+ \frac{j}{2s}\Big) \log_2 d + \Big(1+ \frac{j}{2s}\Big)\abs{\log_2 \eta_{j,k}} \\
&\qquad  + I \log_2 \gamma^k + j\log_2 \gamma\quad \\
&= \log_2 C_9(d) + \frac{j}{2s} \log_2 (C_3d) + p_1\log_2 d + \Big(1+ \frac{j}{2s} + I\log_2 \gamma\Big) \abs{\log_2 \eta_{j,k}} \\
&\qquad + (\log_2 \gamma) \big(j+jI \abs{\log_2 \rho} + I \abs{\log_2(\rho\varepsilon_0)} \big). %
\end{align*}
To proceed, recall that  by Assumptions \ref{ass:approximability} and \ref{ass:dim},
  $t$  and $\varepsilon_0$ are independent of $d$. Moreover,
by  Remark \ref{rem:damping}, $\ln\rho$ is bounded from above and below independently of $d$, see Remark \ref{dimdependence} for a further discussion of this point. Finally, we know that there exist constants $c,C$ such that
$$
j\leq I \leq c  \ln d,\quad \gamma \le C d \,.
$$
Hence, there exists a constant $C_{10}$ such that
\begin{equation}
\label{eq:lvlbd2}
  L_{k,j} \leq C_{10} \bigl((\ln d)^2  \abs{ \ln\eta_{k,j}} + (\ln d)^3 + \ln C_9(d)
   \bigr) \,.
\end{equation}
Here and in the following, for simplicity we consider without loss of generality  the case that $\ln d > 1$.

In the notation of Theorem \ref{thm:apply}(v), we have $L(\bw_{k,j}) \leq L_{k,j}$.
Furthermore, note that $\ln C_9(d) \leq s^{-1} \max\{1,\ln (C_8 \bar C_\bbf \hat C_{\bu,\bbf}^{1/2}) \} \ln d$.
From \eqref{eq:pigrowth}, \eqref{eq:lvlbd2}, and \eqref{metav}, we thus infer
$$
\hat m(\eta_{k,j}; \bw_{k,j}) \leq C_{11} \bigl( (\ln d)^2  \abs{ \ln\eta_{k,j}} + (\ln d)^3 
\bigr)  \,.
$$

Recall that the decay of best low-rank approximation errors is governed by the inverse $\gamma_\bu^{-1}$ of the growth sequence 
$\gamma_\bu(n)= e^{d_\bu n^{1/b_\bu}}$, see Remark \ref{rem:howtoread}.
Under Assumptions \ref{ass:approximability}\eqref{ass:uapprox}, \eqref{eq:combinedcoarsen_rankest} in Theorem \ref{lmm:combined_coarsening} then yields
$$   
\abs{\rank(\bw_{k,0})}_\infty \leq (d_\bu^{-1} \ln[ ( \kappa_1 \alpha)^{-1} \norm{\bu}_{\Acal_\Hcal(\ga_\bu)} \rho \eta_{k,0}^{-1}])^{b_\bu} \le C(\bu) (\abs{\ln \eta_{k,0} } + \ln d)^{b_\bu} \,,
$$
where we have used in the last step that $\kappa_1^{-1}\lesssim d$.
By Theorem \ref{thm:apply}(iii), setting $\bar R:= \max_\alpha R_\alpha$, which by Assumptions \ref{ass:dim} is bounded independently of $d$, we now obtain
$$ 
 \abs{\rank( \bw_{k,j+1})}_\infty \leq \bigl(\hat m(\eta_{k,j}; \bw_{k,j}) \bigr)^2 \bar R\,  \abs{\rank( \bw_{k,j})}_\infty  + C^{\text{{\rm rank}}}_\bbf \, \abs{\ln \eta_{k,j}}^{b_\bbf} \,.
$$
As a consequence, setting $b:= \max\{b_\bu,b_\bbf\}$, and using again that $I \leq c \ln d$ and hence 
$$
\abs{ \ln\eta_{k,I}} \leq \abs{\ln\eta_{k,0}} + c \ln d \,\abs{\ln \rho}, 
$$
we conclude that
\begin{align} 
\abs{\rank(\bw_{k,I})}_\infty  
  &  \leq C_{12}d^{ (\ln \bar R + 2 \ln C_{11}) c } \,\bigl( (\ln d)^2  (\abs{\ln\eta_{k,0}} + c \ln d \,\abs{\ln \rho}) + (\ln d)^3        
       \bigr)^{2I} \notag \\
     &\qquad  \times ( \abs{\ln\eta_{k,0}} + c \ln d \,\abs{\ln \rho} + \ln d )^{b}  \notag \\   
   & \leq C_{13}d^{ p_2 } (\ln d)^b \,\bigl( (\ln d)^2  \abs{ \ln\eta_{k,0}} + (\ln d)^3   \bigr)^{2I}  \abs{\ln \eta_{k,0}}^{b} \,,
         \label{eq:rankest_full}
\end{align}        
where $p_2 := (\ln \bar R + 2 \ln C_{11} + 2 \ln (1 + c\abs{\ln\rho})) c$. 

In view of Assumptions \ref{ass:approximability}\eqref{ass:rhsops} as well as Theorem \ref{thm:apply}(iv) and (v), the complexity of each inner loop in Algorithm \ref{alg:tensor_opeq_solve} is dominated by that of the hierarchical singular value decompositions used in $\recompress$ and $\coarsen$ (see Remark \ref{rem:CRcomplexity}). Therefore, it is for each $k$, $j$, in view of \eref{eq:wkjsuppbound}, bounded by
\begin{equation*}
  C_{14} \Bigl[  d \, \abs{\rank( \bw_{k,j})}_\infty^4 
  + \abs{\rank( \bw_{k,j})}_\infty^2  
  d^{p_1}  (C_3 d)^{\frac{j}s} C_{\bu,\bbf} \,\eta_{k,j}^{-\frac1s}   \Bigr] \,.
\end{equation*}
Likewise the number of operations for the outer loop with index $k$ is bounded by
\begin{equation*}
   C_{15} \Bigl[  d I \, \abs{\rank( \bw_{k,I})}_\infty^4 
     + \abs{\rank( \bw_{k,I})}_\infty^2  
    d^{p_1}   (C_3 d)^{\frac{I+1}s} C_{\bu,\bbf} \,\eta_{k,I}^{-\frac1s}   \Bigr] \,.
\end{equation*}
The total work for arriving at $\bu_k$ is thus bounded by 
\begin{equation}
\label{eq:opstotal1}
 C_{16} \Bigl[  d I \, \abs{\rank( \bw_{k-1,I})}_\infty^4 k
     + \abs{\rank( \bw_{k-1,I})}_\infty^2  
     d^{p_1}   (C_3 d)^{\frac{c\ln d+1}s} d^{\frac1s} \hat C_{\bu,\bbf}\, \eta_{k-1,I}^{-\frac1s}   \Bigr] \,.
\end{equation}

We need to express the above bounds  in terms of $\varepsilon_k$.
In this regard, note that $k = \log_2\varepsilon_0 - \log_2 \varepsilon_k$, $\eta_{k,0} = \rho \varepsilon_k$, and $\eta_{k-1,I} = \rho^{I+1} 2 \varepsilon_k$. The latter relation yields $\eta_{k-1,I}^{-\frac1s} \leq (2\rho)^{-\frac1s}  d^{c s^{-1} \abs{\ln \rho}} \varepsilon_k^{-\frac1s}$.
From \eqref{eq:rankest_full}, we now obtain first for the ranks
$$  
   \abs{\rank (\bw_{k,I})}_\infty \leq C_{17}d^{p_2 } \,(\ln d)^{b + 6 c \ln d} \abs{\ln \varepsilon_k}^{b + 2 c \ln d} \,.
$$   
Using this in \eqref{eq:opstotal1} gives the bound
$$
  \ops(\bu_{k}) \leq C_{18}\, (\ln d)^{1+b} d^{p_3}\, \,(\ln d)^{24 c \ln d} \,d^{c s^{-1}  \ln d} \, 
  \abs{\ln \varepsilon_k}^{2b + 4 c \ln d} \,\varepsilon_k^{-\frac1s} \,,
$$
where $p_3 := 1+2s^{-1}+cs^{-1}(\abs{\ln\rho}+ \ln C_3)+p_1+4 p_2$.
This completes the proof.
\end{proof}

\begin{remark}\label{dimdependence}
In the present case of a symmetric elliptic operator, an appropriate choice of $\omega$ yields $\abs{\ln \rho} \sim [\operatorname{cond}_2(\bA)]^{-1}$. As a consequence, the bound $I$ for the number of inner iterations scales linearly in $\operatorname{cond}_2(\bA)$.
A violation of our assumption of a $d$-independent bound on $\operatorname{cond}_2(\bA)$ made in Assumptions \ref{ass:dim} 
therefore has a considerable impact on the resulting complexity estimates. 
In particular, $\operatorname{cond}_2(A) \sim d^2$, which is the case in Example \ref{ex:tridiag}, would in fact lead to a complexity estimate with superexponential dependence on $d$.
\end{remark}

\section{Numerical Experiments}\label{sec:num-res}

\subsection{Basic Considerations}

There are two basic choices to be made in a practical realization of Algorithm \ref{alg:tensor_opeq_solve}: the dimension tree $\hdimtree{d}$ for the hierarchical tensor format, and the univariate wavelet basis $\{\psi_\nu \}_{\nu\in\nabla}$. For $\hdimtree{d}$, we use the simplest possible choice \eqref{eq:lindimtree}.

Concerning the choice of wavelets, the available options are limited by the restriction to orthonormal bases (cf.\ Proposition \ref{prop:cond}). A further issue is that, in view of the dependence of the ranks of the approximations of $\Sr^{-1}$ on the maximum active wavelet levels, the compressed application of the rescaled lower-dimensional components $\bA^{(i)}_n$ should increase these maximum levels as little as possible. By classical results on wavelet compression (see, e.g., \cite{Cohen:01}), the wavelets should therefore have high global regularity. In addition, it is desirable that the wavelets are piecewise polynomials. The resulting $\bA^{(i)}_n$ then have very favorable $s^*$-compressibility, exceeding, in particular, the order of the trial functions \cite{Stevenson:02}. 
For all  results presented below, we therefore use orthonormal, continuously differentiable, piecewise polynomial Donovan-Geronimo-Hardin multiwavelets \cite{DGH:99} of polynomial degree\footnote{Note that this is the lowest possible degree for the continuously differentible construction in \cite{DGH:99}.} 6 and approximation order 7. 

\subsection{Improving the Practical Efficiency of {\rm$\apply$}}
 
In a practical realization of the routine $\apply$  we have described above, additional care needs to be taken to keep the ranks arising in the evaluation as low as possible. We now describe a practical procedure that achieves this, retaining the guaranteed output error of the original procedure $\apply$.

We consider the evaluation of $\apply(\bv;\eta)$, where $\bv = \sum_{\kk{k}} a_{\kk{k}} \bigotimes_i \bU^{(i)}_{k_i}$ with $\mathbf{a}$ decomposed further in the hierarchical format. 
As one-dimensional scaling sequences, we choose $\omi{i}{\nu} := \sqrt{a_{ii}\bT^{(i)}_{2,\nu\nu}}$.
For each $i$ and $n_i=2,3,4$, we first determine the matrix entriy indices  $(\nu,\mu)$ required  for the 
approximations of $\bA^{(i)}_{n_i} \pi^{(i)}(\bv)$ with $J(\eta/2)$ as defined in \eqref{etaJ}, and precompute all corresponding $\tilde \bT^{(i)}_{n_i,\nu\mu}$. This gives the components of  $\tilde \bT_{J(\eta/2)} =  \sum_{\kk{n}} c_{\kk{n}} \bigotimes_i \tilde\bT^{(i)}_{n_i}$ such that $\Sr^{-1} \tilde\bT_{J(\eta/2)} \Sr^{-1}$ is a suitable approximation of $\bA$.
Similarly to \eqref{Tv}, we can now determine two separate values $T_0 := \argmin\{T': \supp\bv \subseteq \Lambda_{T'}\}$, $T_1 := \argmin \{ T' : \supp \tilde\bT \bv \subseteq \Lambda_{T'} \}$
and set $m_0 := M( c(\bv) (\eta/2) ; T_0 )$, $m_1 :=  M( c(\bv) (\eta/2) ; T_0 )$, with $M$ defined in \eqref{Nsc} and $c(\bv) \eta/2$ as in \eqref{etaJ}. According to Proposition \ref{prop:weta}, $\bw_{\eta/2} := \Sa{m_1}^{-1} \tilde \bT_{J(\eta/2)} \Sa{m_0}^{-1} \bv$ satisfies $\norm{\bA \bv - \bw_{\eta/2}} \leq \eta/2$. Instead of evaluating $\bw_{\eta/2}$ directly (which, from a practical perspective, could lead to prohibitively high ranks), it is advisable to control the ranks by
additional approximations, which amounts to computing a $\tilde\bw_{\eta/2}$ such that $\norm{\bw_{\eta/2} - \tilde\bw_{\eta/2}} \leq \eta/2$. We shall now describe how $\tilde\bw_{\eta/2}$, which is subsequently used as the output of $\apply(\bv;\eta)$, is obtained.

Recall from Section \ref{ssec:nearsep} that $\Sa{m_r}^{-1}$, $r = 0,1$, can be written in the form 
$$  
 \Sa{m_r}^{-1} = \sum_{\ell = 1}^{\hat m_r}  \Theta_\ell \,,\quad \Theta_\ell := \theta_\ell^{(1)} \otimes\cdots \otimes \theta_\ell^{(d)} \,, 
 $$
where $\hat m_r := 1 + n^+(\delta) + m_r$, and $\theta_\ell^{(i)} = \operatorname{diag} ( \tilde w_\ell^{1/d} e^{-\tilde\alpha_\ell \omi{i}{\nu}^2})_\nu$ with coefficients $\tilde w_\ell, \tilde\alpha_\ell > 0$ given in Theorem \ref{thm:expsum_relerr}.
Note that 
\begin{equation*}
\bw_{\eta/2}  = \sum_{\ell_0=1}^{\hat m_0} \sum_{\ell_1=1}^{\hat m_1}  \Theta_{\ell_1} \tilde\bT_{J(\eta/2)} \Theta_{\ell_0} \bv  \,, \end{equation*}
where the ranks of each summand $\Theta_{\ell_1} \tilde\bT_{J(\eta/2)} \Theta_{\ell_0} \bv$ are bounded by $\max_{\alpha\in \hdimtree{d}} R_\alpha \abs{\rank(\bv)}_\infty$. 

The additional approximations with total error at most $\eta/2$ used in assembling $\bw_{\eta/2}$, which lead to the final output $\tilde\bw_{\eta/2}$, are performed as follows. 
For each $i$ and $n_i, k_i$, we preassemble sparse matrices $\mathbf{W}^{(i)}_{n_i,k_i}$ with entries
$   \mathbf{W}^{(i)}_{n_i,k_i; \nu,\mu} := \tilde\bT^{(i)}_{n_i,\nu\mu} \bU^{(i)}_{k_i,\mu}  $,
and evaluate 
$$ 
   \tau_{\ell_0,\ell_1} :=   \Bignorm{\sum_{\kk{n}, \kk{k}} c_{\kk{n}} a_{\kk{k}} \bigotimes_i \theta^{(i)}_{\ell_1} \mathbf{W}^{(i)}_{n_i,k_i} (\theta^{(i)}_{\ell_0}  \chi_{\supp_i \bv} ) }  \,,
$$ 
where $\chi_{\supp_i \bv}$ denotes the characteristic function of $\supp_i \bv$.
For each $\ell_0$, $\ell_1$, the computation of $\tau_{\ell_0,\ell_1}$ involves the orthogonalization of a hierarchical tensor of relatively low hierarchical ranks.
We now determine a nondecreasing ordering $\hat\tau_q$, $q = 1, \ldots, \hat m_0 \hat m_1$, of these values, with  corresponding pairs
 $(\hat \ell_{0,q}, \hat \ell_{1,q})$ such that $\tau_{\hat \ell_{0,q}, \hat \ell_{1,q}} = \hat \tau_q$ for each $q$.

We first determine the largest $q_0$ such that 
$\sum_{q=1}^{q_0} \hat\tau_q \leq \frac\eta4$,
and discard the parts of the tensor corresponding to $(\hat \ell_{0,q}, \hat \ell_{1,q})$ for $q = 1,\ldots,q_0$. With $q_1 := q_0 + 1$, $q_2 :=  \hat m_0 \hat m_1$, it thus remains to approximate
$$   
 \sum_{q=q_1}^{q_2}   \sum_{\kk{n}, \kk{k}} c_{\kk{n}} a_{\kk{k}} \bigotimes_i \theta^{(i)}_{\hat\ell_{1,q}} \mathbf{W}^{(i)}_{n_i,k_i} \theta^{(i)}_{\hat\ell_{0,q}}  \,.
$$
Here our strategy is to sum these parts in the given order, and apply $\recompress(\cdot; \zeta_q)$ to the intermediate result after each summation; that is, $\zeta_q$ denotes the tolerance used for recompression after adding the term with index $q$. Various different strategies are possible for choosing these $\zeta_q$, with the constraint that $\sum_{q=q_1}^{q_2} \zeta_q \leq \frac\eta4$. Since we start the tensor summation with the smallest contributions, a natural approach for keeping ranks small is to always recompress with a tolerance proportional to an estimate of the relative size of the current intermediate result. This is accomplished by the choice
$$ 
    \zeta_q :=  \frac{\eta\,\sum_{p=q_1}^{q} \hat\tau_p}{4 \sum_{p=q_1}^{q_2} (q_2 + 1 - p) \hat\tau_p }  \,.
$$
Since more complicated choices of $\zeta_q$ (e.g.\ using additional a posteriori information) did not yield a further improvement in our numerical tests, the presented results are based on the above prescription.

It should be noted that this scheme with additional recompressions always preserves convergence, since the prescribed error tolerances for $\apply$ are preserved, but its effect on the computational complexity depends on the rank decrease achieved by the additional truncations.
This, however, is not clear a priori, but in practice the additional recompressions are observed to improve efficiency substantially.

\subsection{A High-Dimensional Poisson Problem}

As a first model example, we consider the Poisson problem $-\Delta u = 1$ on $(0,1)^d$ with homogeneous Dirichlet boundary conditions.
We refer to Example \ref{ex:laplace} concerning the hierarchical tensor representation of $\bT$ in this case. We are, in particular,  interested in 
assessing the $d$-dependence of the computational complexity for achieving a certain $H^1$-error bound. 

\begin{figure}[t]
\centering
\begin{tabular}{ccc}\hspace{-.6cm}
\includegraphics[width=4.2cm]{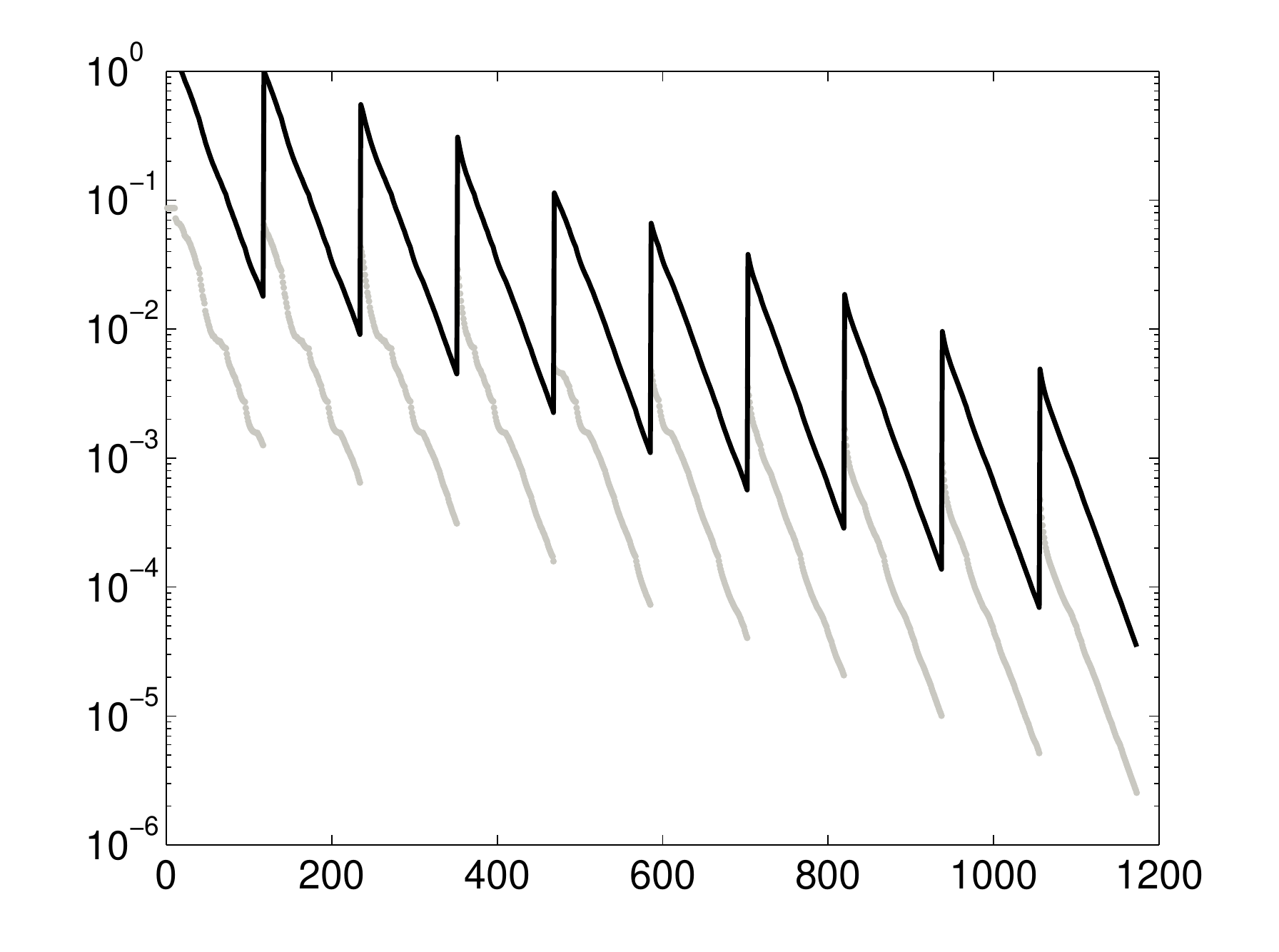} &
\includegraphics[width=4.2cm]{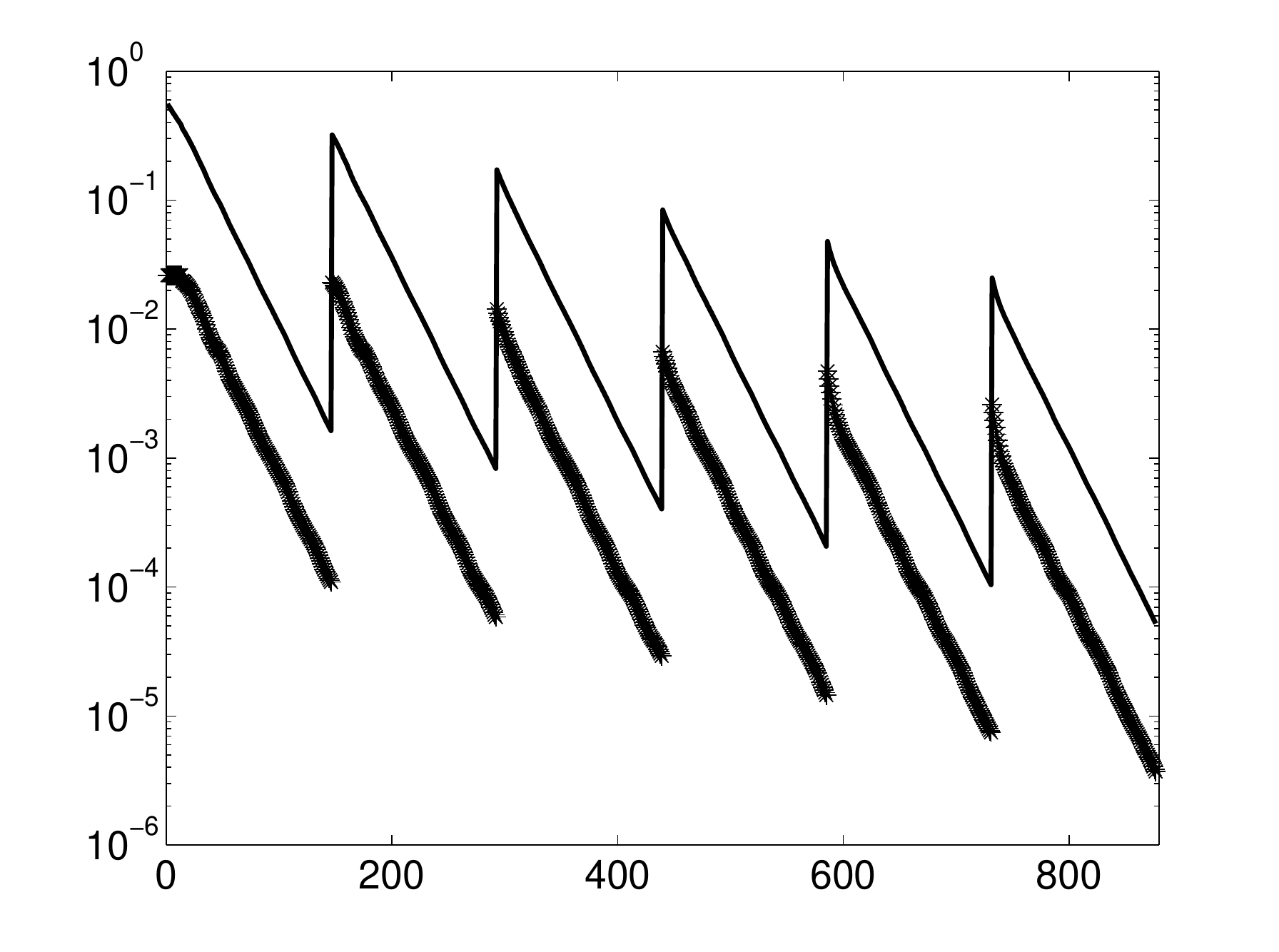} &
\includegraphics[width=4.2cm]{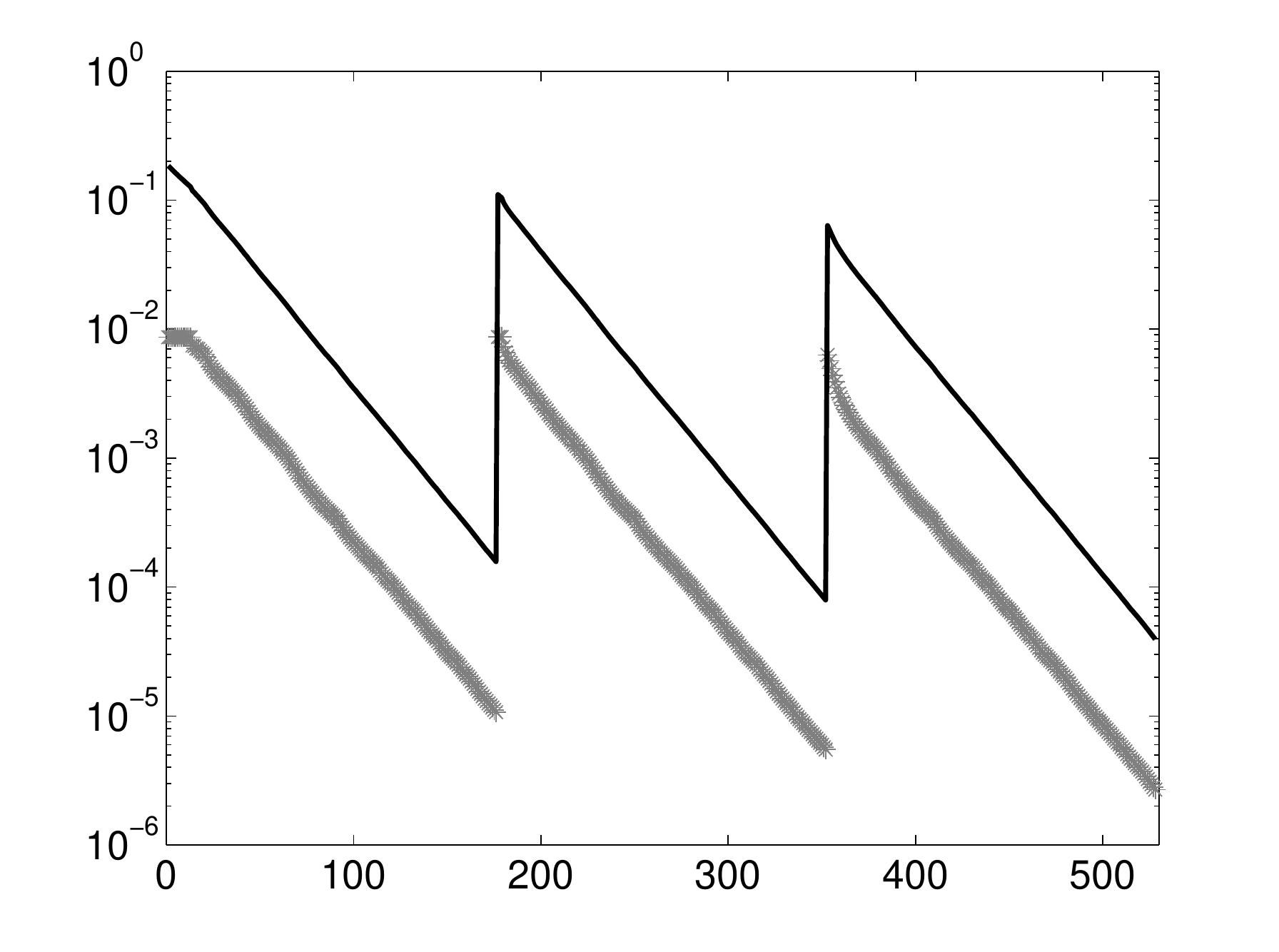}\hspace{-.5cm}
\end{tabular}
\caption{Norms of computed residual estimates (markers) and corresponding error bounds (lines),
in dependence on the total number of inner iterations (horizontal axis), for  $d=\textcolor{lightgray}{\bullet} 4,\ast 16,\textcolor{gray}{\ast} 64$.}
\label{fig:res}
\end{figure}
Figure \ref{fig:res} shows the evolution of the residuals and the corresponding estimates for the $H^1$-error in the course of the iterative scheme. Both residuals and errors behave as expected, with an intermittent increase due to the coarsening and recompression after each completed inner loop. As shown here for three exemplary  values of $d$, a consequence of the $d$-dependence of the choice of the parameter $\kappa_1$ required in our complexity estimates is that the number of iterations within each inner loop increases with $d$. Hence for larger $d$, smaller errors are reached within a lower total number of iterations, but these iterations become increasingly expensive, since the representation complexity of intermediate results is reduced less frequently by coarsening and recompression steps.

\begin{figure}[t]
\centering
\begin{tabular}{cc}\hspace{-1cm}
\includegraphics[width=6cm]{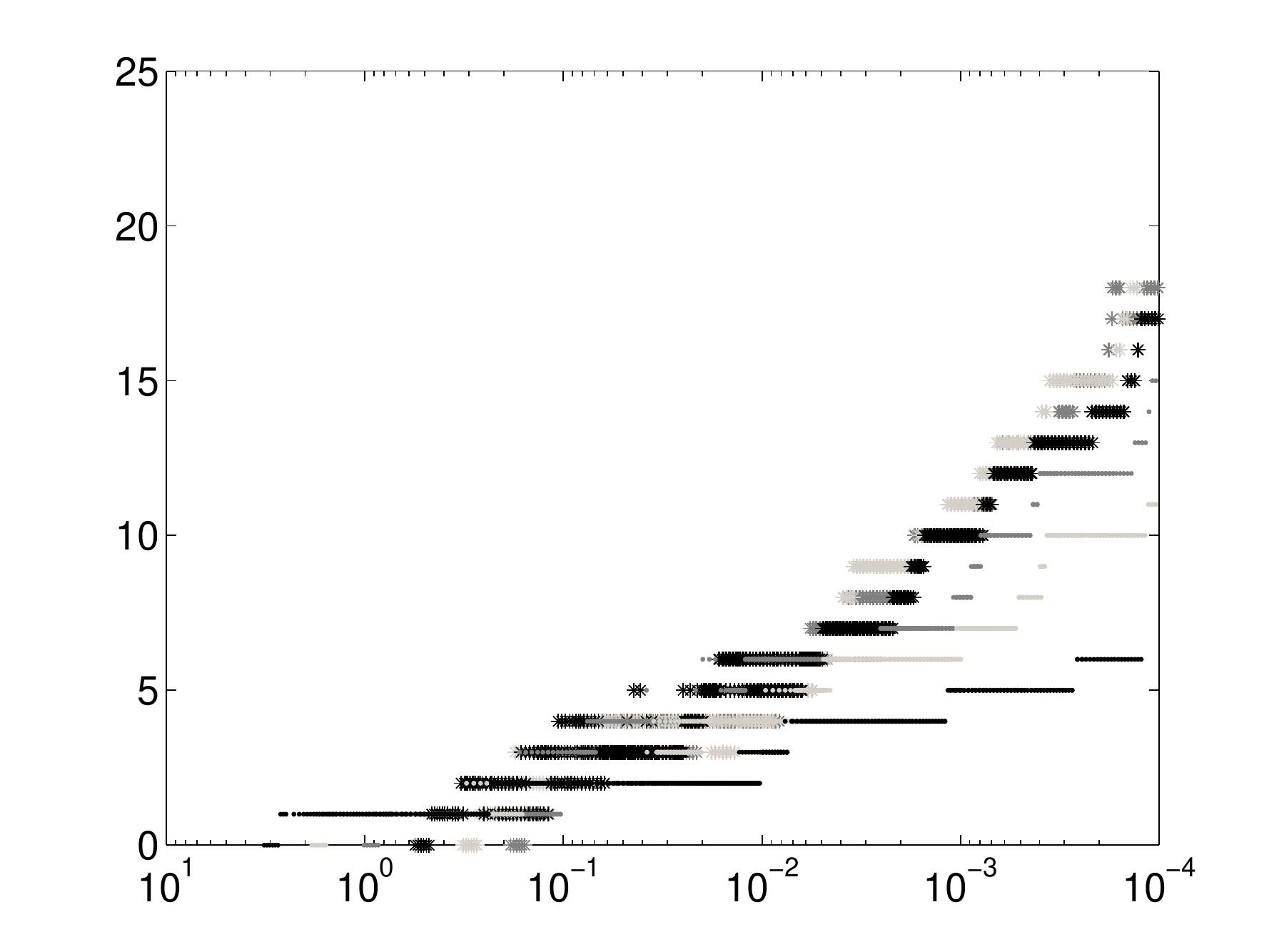} &
\includegraphics[width=6cm]{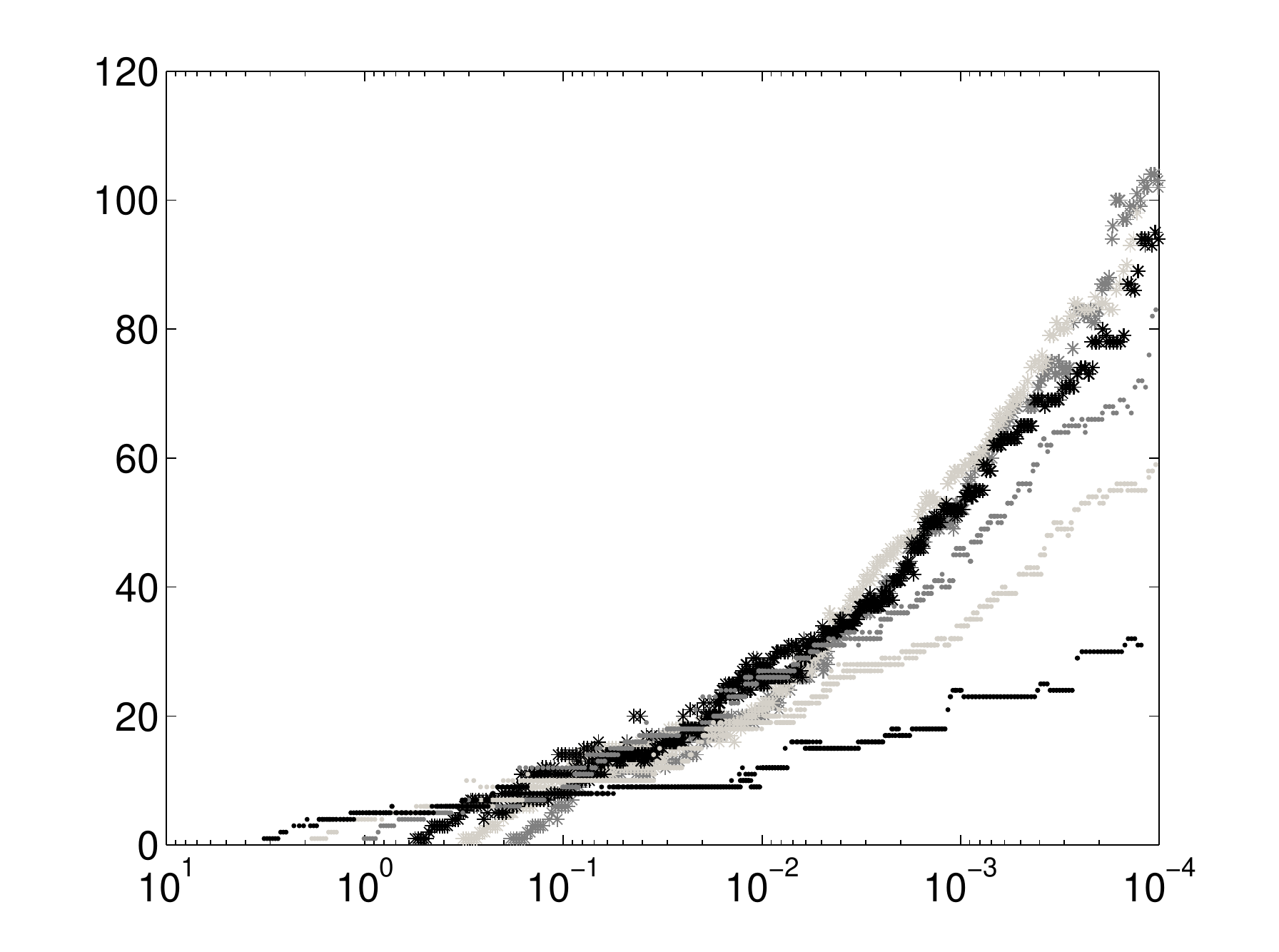}
\end{tabular}
\caption{$\abs{\rank(\bw_{k,j})}_\infty$ (left) and maximum ranks of all intermediates arising in the inner iteration steps (right),
in dependence on current estimate for $\norm{\bu - \bw_{k,j}}$ (horizontal axis), for  $d=\bullet 2,\textcolor{lightgray}{\bullet} 4,\textcolor{gray}{\bullet}8, \ast 16, \textcolor{lightgray}{\ast} 32, \textcolor{gray}{\ast} 64$.}
\label{fig:ranks}
\end{figure}

In Figure \ref{fig:ranks}, we compare the dependence of both the maximum ranks of the iterates and of the intermediate quantities arising in the computation on the $H^1$-error bound for different values of $d$. In view of Remark \ref{rem:CRcomplexity}, these ranks strongly influence the computational cost. We observe only a gradual increase of both types of ranks with decreasing $H^1$-error.
Furthermore, for relatively small values of $d$ we observe an increase of the required ranks with increasing $d$. This is to be expected on the one hand due to \eqref{eq:complexity_rank}, on the other hand as a consequence of the tighter error tolerances e.g.\ in $\apply$ that are required in higher dimensions. However, for larger dimensions such as $d=16,32,64$, the differences between maximum ranks observed at a certain error tolerance for different values of $d$ diminish.

\begin{figure}[t]
\centering
\includegraphics[width=8cm]{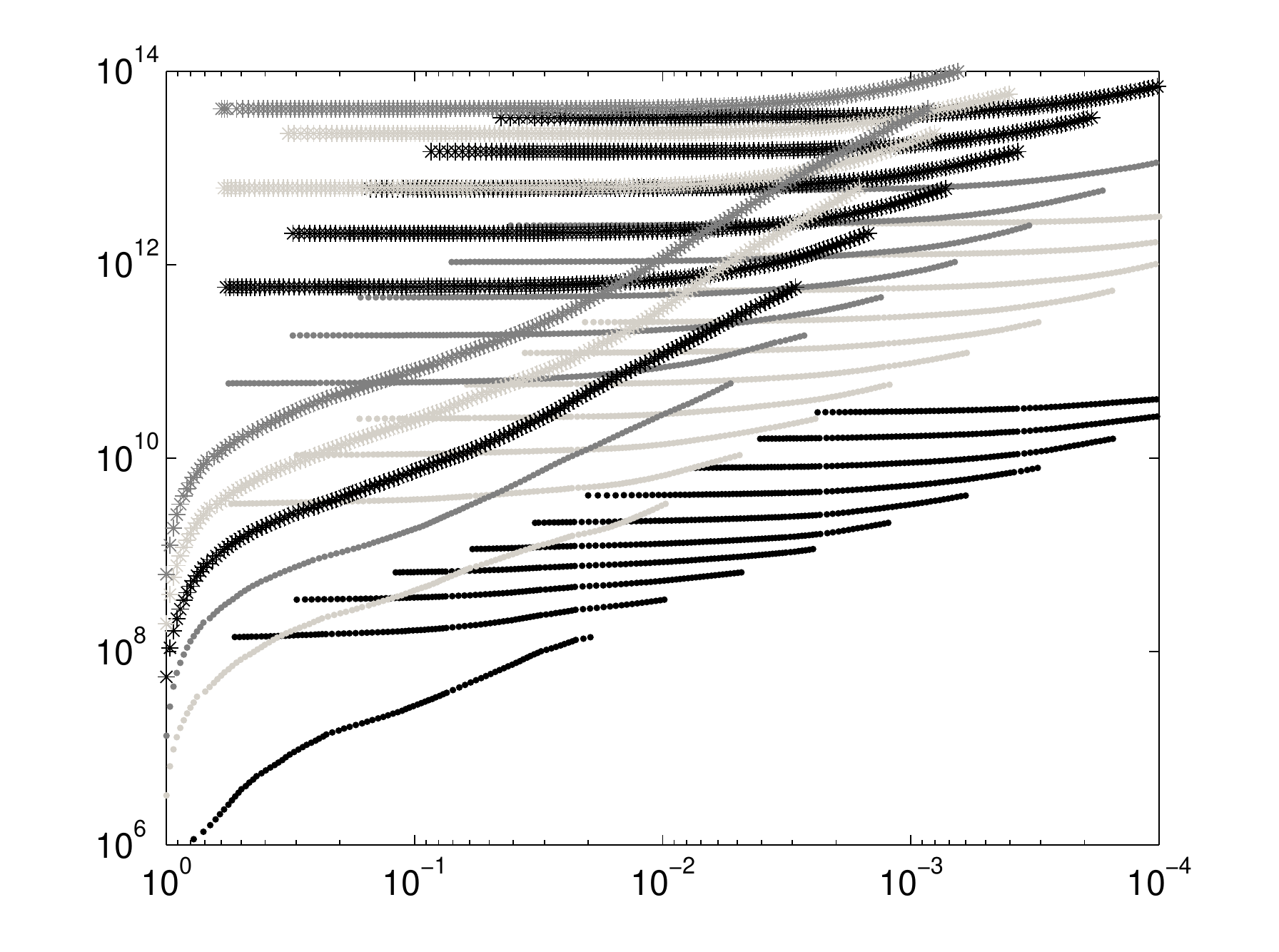}
\caption{Operation count in dependence on the error estimate reduction (horizontal axis), for $d=\bullet 2,\textcolor{lightgray}{\bullet} 4,\textcolor{gray}{\bullet}8, \ast 16, \textcolor{lightgray}{\ast} 32, \textcolor{gray}{\ast} 64$.}
\label{fig:ops}
\end{figure}
In Figure \ref{fig:ops}, the computed estimates for the operation counts\footnote{The given operation counts are obtained using standard estimates (see, e.g., \cite{Hackbusch:12}) for each performed linear algebra operation, and counting the handling of each matrix entry by quadrature (which is $\Ocal(1)$ in our setting) as a single operation. This simplified counting therefore differs from the true number of floating point operations by a certain fixed factor, but does reflect the asymptotic behaviour.}
required to arrive at a \emph{relative} error tolerance are compared for the same values of $d$. For this comparison we use the reduction with respect to the initial error estimate for comparison because, as can also be seen in Figure \ref{fig:ranks}, the norms of $\bbf$, $\bu$ as well as the corresponding initial errors decrease slightly with increasing $d$. For each $d$, similarly to Figure \ref{fig:res}, one observes a characteristic pattern caused by coarsening and recompression steps, where the iteration periodically returns to larger error tolerances.
It is to be noted in particular that the number of operations required for a certain error reduction exhibits a \emph{polynomial} growth in $d$. Thus the method in this case  performs substantially better in practice than  the theoretical complexity guarantees of Theorem \ref{thm:complexity}.

The results can also be compared to those given in \cite[Fig.\ 4]{Dijkema:09} for essentially the same problem\footnote{The only difference is that they impose homogeneous Neumann conditions on certain faces of $\partial(0,1)^d$, and homogeneous Dirichlet on the remaining ones, resulting in symmetry boundary conditions that yield the solution $\hat u|_{(0,1)^d}$, where $\hat u$ solves the  homogeneous Dirichlet problem $-\Delta \hat u = 1$ on $(-1,1)^d$. By a simple scaling argument, one verifies that this problem of approximating $\hat u$ on the single orthant $(0,1)^d$ of $(-1,1)^d$ is (up to a dimension-independent factor) exactly as difficult as the problem that we are considering.}, which are based on direct best $n$-term approximation in a $d$-dimensional tensor product multiwavelet basis. A comparison of the accomplished accuracies   indicates that such a sparse-grid type approximation becomes computationally intractable for large $d$.

\subsection{A Dirichlet Problem with Tridiagonal Diffusion Matrix}

One of the strengths of the proposed method is that, in contrast e.g.\ to the direct application of exponential sum approximations \cite{Grasedyck:04}, it can still be applied when $A$ does \emph{not} have a Laplace-like structure with each summand in the operator acting only on a single variable. For instance, such a structure is not present for $A$ given by \eqref{eq:example_operator} with the tridiagonal diffusion matrix considered in Example \ref{ex:tridiag}, which has values $2$ on the main diagonal and $-1$ on the secondary diagonals. Note that although our scheme can be applied also in this case, the problem does not satisfy the assumptions we have made in our complexity analysis. Specifically, as noted in Remark \ref{dimdependence}, we have $\operatorname{cond}_2(\bA) \sim  d^{2}$. In this sense this example sheds some light on  the role of our assumptions and possible effects of their violation. 
The issues encountered with tensor expansions in this problem are indicated by the following observation. Diagonalizing the diffusion matrix 
transforms the problem   to a rotated domain (which is no longer of product type), where the diffusion tensor becomes diagonal with largest entry uniformly bounded and smallest entry proportional to $d^{-2}$. As a consequence, we have to expect that in the original coordinates,  the solution exhibits
anisotropic structures  that are not aligned with the coordinate axes and become more pronounced with increasing $d$.

This is reflected in the numerical results, where both ranks (Figure \ref{fig:rankstri}) and computational complexity (Figure \ref{fig:opstri}) show a much more rapid increase than for the Poisson problem. Besides the larger approximation ranks, the efficiency of the scheme is also affected by the deterioration of the error reduction rate $\rho$ caused by the dimension-dependent condition number.
\begin{figure}[t]
\centering
\begin{tabular}{cc}\hspace{-.6cm}
\includegraphics[width=6cm]{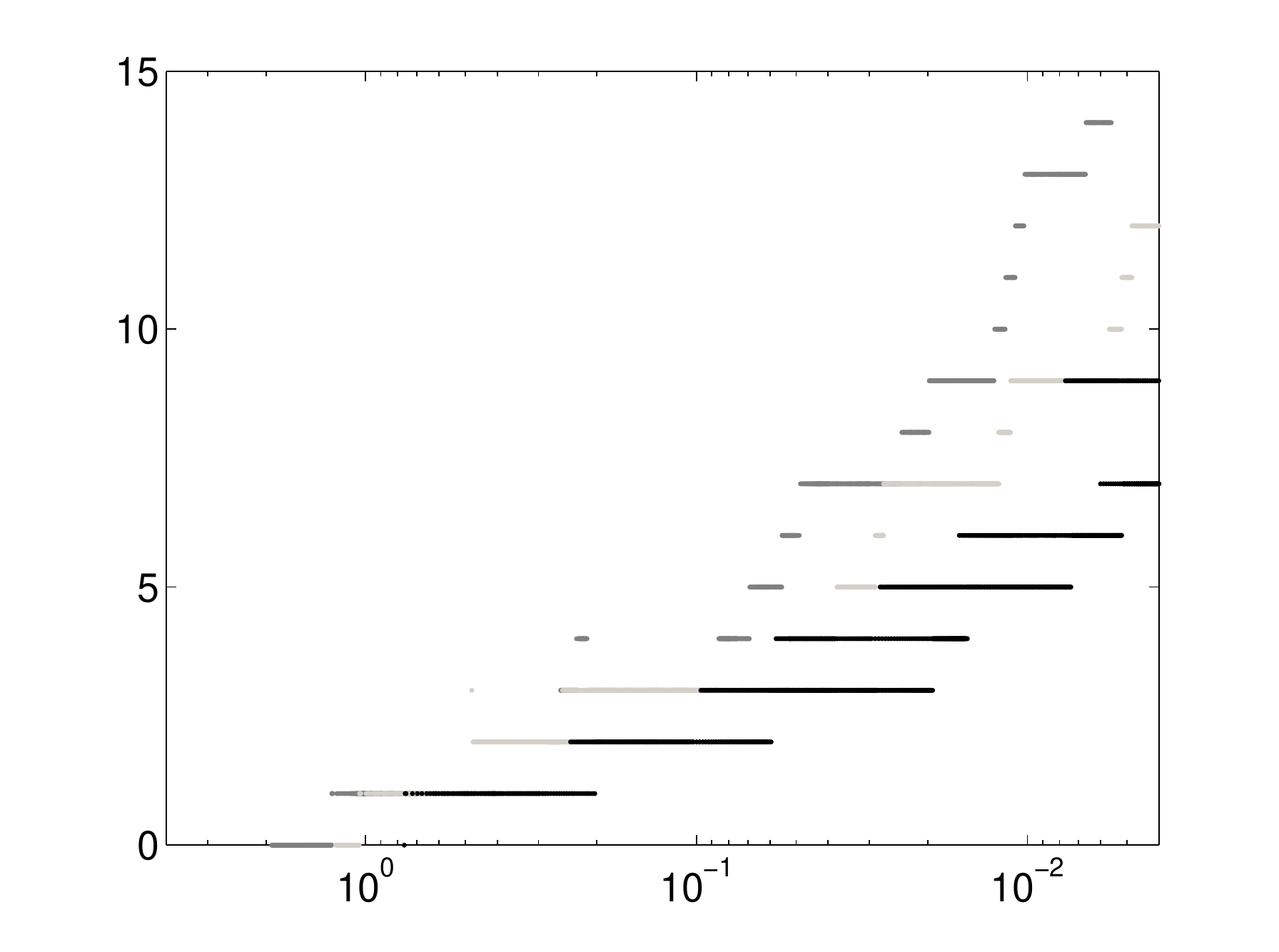} &
\includegraphics[width=6cm]{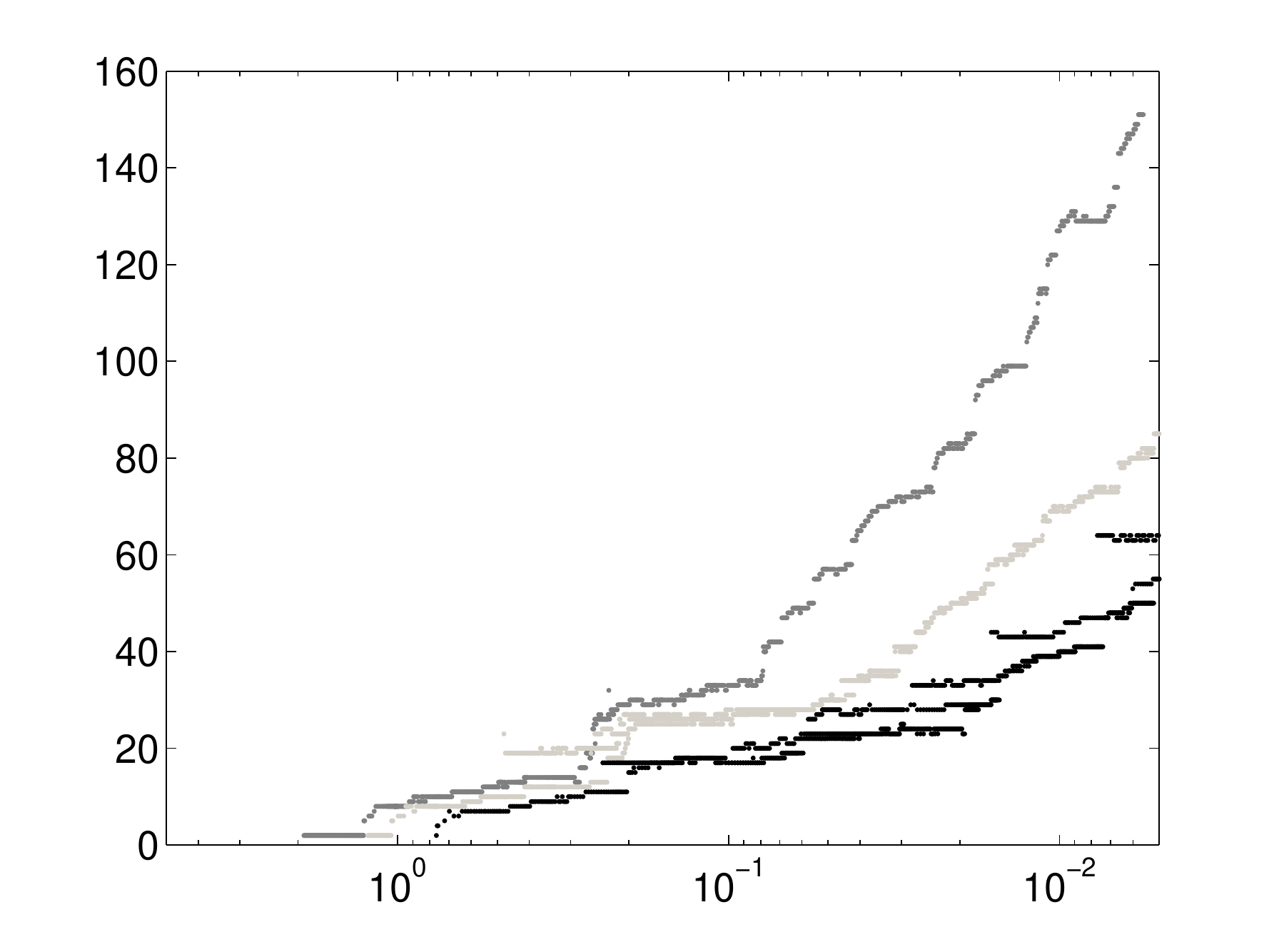}
\end{tabular}
\caption{Tridiagonal diffusion matrix: $\abs{\rank(\bw_{k,j})}_\infty$ (left) and maximum ranks of all intermediates arising in the inner iteration steps (right),
in dependence on current estimates for $\norm{\bu - \bw_{k,j}}$ (horizontal axis), for  $d=\bullet 2,\textcolor{lightgray}{\bullet} 3,\textcolor{gray}{\bullet}4$.}
\label{fig:rankstri}
\end{figure}
\begin{figure}[t]
\centering
\includegraphics[width=8cm]{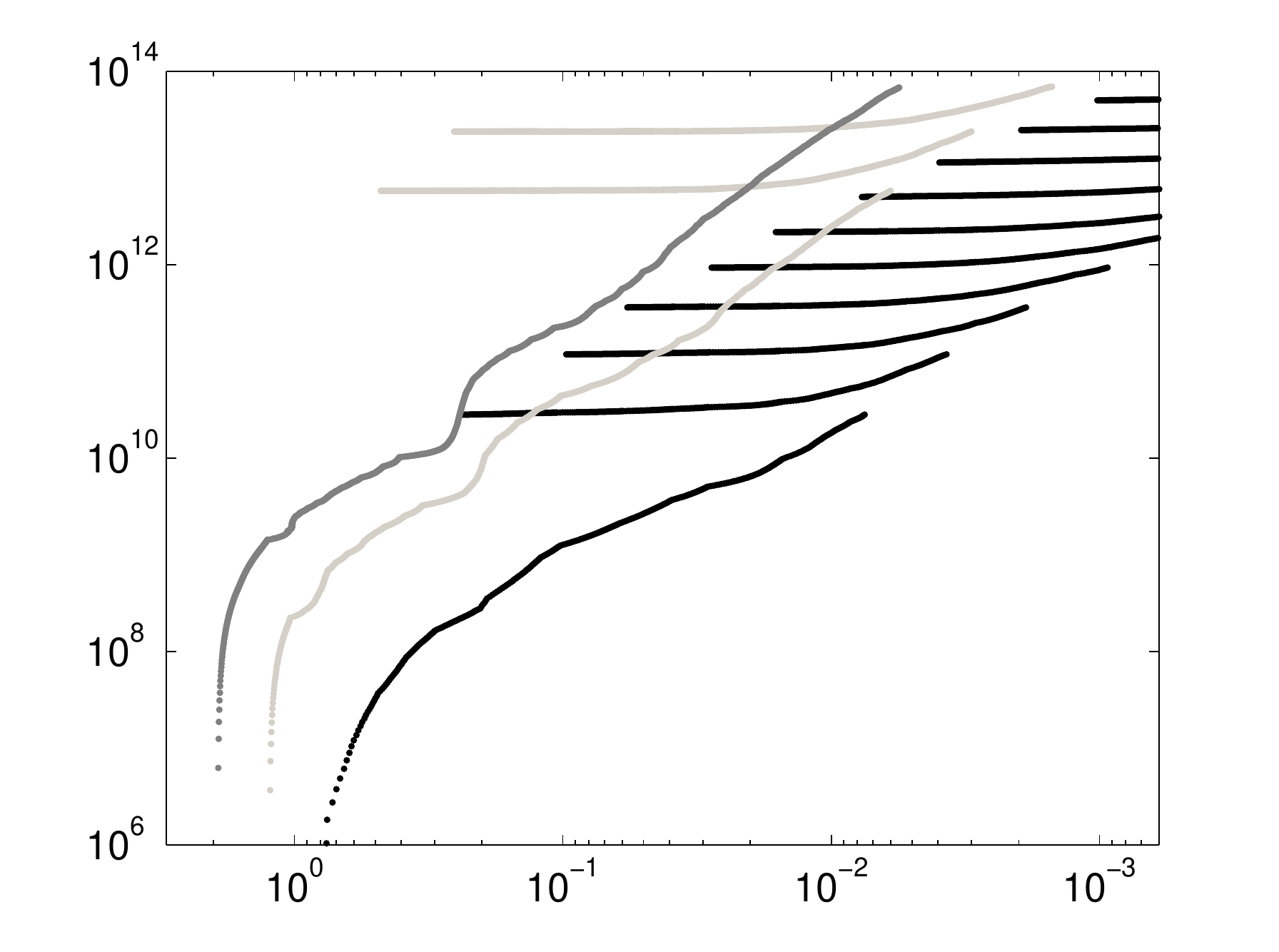}
\caption{Tridiagonal diffusion matrix: operation count in dependence on error estimate reduction (horizontal axis), for $d=\bullet 2,\textcolor{lightgray}{\bullet} 3,\textcolor{gray}{\bullet}4$.}
\label{fig:opstri}
\end{figure}

However, it also needs to be emphasized that the more rapid rank growth is {\em not} solely caused by the non-diagonal diffusion matrix coupling several variables. In fact, there exist other tridiagonal matrices, e.g. with $2$ on the main diagonal and $-\alpha$ with $\alpha \in (0,1)$ on the secondary diagonals, for which the condition number of $\bA$ remains $d$-independent.
A more detailed study of such further model cases will be done elsewhere.

\section{Conclusion}
We have constructed and analyzed an adaptive iterative algorithm for the approximate solution of second order
elliptic boundary value problems on high-dimensional product domains. The algorithm generates for any given target accuracy
$\varepsilon$ an approximation of {\em finite hierarchical rank} that meets the target accuracy with respect to the {\em energy norm},
which to our knowledge is the first result of this type. The analysis brings out several intrinsic obstructions, which originate from the fact
the energy norm is not a cross norm. As a consequence, using corresponding continuity properties to obtain a well-conditioned problem (e.g. by diagonal rescaling of wavelet coefficients as in our case, or by other types of preconditioning) destroys existing explicit low-rank structures. 
Nevertheless, it is shown that under certain benchmark assumptions of the solution, the scheme nearly reproduces minimal ranks and tensor representation sparsity,
without making use of any related a priori knowledge of these assumptions.
Our analysis carefully tracks the influence of the spatial dimension $d$ on the computational complexity. In particular, we have made an effort to formulate the benchmark assumptions in a way that keeps the problems for different spatial dimensions comparable.

The theoretical findings are illustrated and further quantified by numerical experiments for spatial dimensions up to $d=64$.
It can be seen that the actual performance is better than the theoretical predictions.
It should be emphasized that the scheme is {\em not} restricted to Poisson-type problems; however, when dealing with more general diffusion operators, the ranks  are seen to increase significantly faster with decreasing target accuracies.

For simplicity, we have considered in this work the perhaps conceptually simplest iterative form, a perturbed Richardson iteration for the 
infinite dimensional problem in $\ell_2$. Significant quantitative improvements are expected when using instead nested iterations of
adaptively refined Galerkin problems. This will be considered in forthcoming work.

\paragraph{Acknowledgements.} The authors would like to thank Kolja Brix for providing multiwavelet construction data used in the numerical experiments.

\begin{appendix}

\section{Proof of Proposition \ref{prop:cond}}\label{app:cond}

\begin{proof}
First note   that for the original operator $A$, we have
$$   \underline{\lambda}_a \langle (-\Delta )v, v\rangle \leq 
    \langle A v, v\rangle \leq \overline{\lambda}_a \langle ( -\Delta) v, v\rangle\,,\quad v\in \spH{1}_0(\Omega) \,.
$$
By our assumptions on $\{\Psi_\nu\}$, we have on the one hand
$\norm{\sum_{\nu\in\nabla^d} v_\nu \Psi_\nu}_{\spL{2}(\Omega)} = \norm{\bv}$
by $\spL{2}(\Omega)$-orthonormality, and on the other hand, we can now follow the lines of \cite[Section 2]{Dijkema:09} and
sum \eqref{deriv1cond} over $i$ to observe that,  by definition $\norm{\Sc \bv}^2 = \sum_i \norm{\Sc_i \bv}^2$, we obtain
$$
   \underline{\lambda}_1\, \norm{\Sc\bv}^2
    \leq \Bigl\langle (-\Delta) \Bigl( \sum_{\nu\in\nabla^d} v_\nu \Psi_\nu \Bigr), 
     \Bigl( \sum_{\nu\in\nabla^d} v_\nu  \Psi_\nu \Bigr)  \Bigr\rangle 
    \leq  \overline{\lambda}_1 \, \norm{\Sc\bv}^2.
$$ 
Consequently, one has
$$
   \underline{\lambda}_a\underline{\lambda}_1 \, \norm{\bv}^2
    \leq \Bigl\langle A \,\Bigl( \sum_{\nu\in\nabla^d} \omega^{-1}_\nu v_\nu \Psi_\nu \Bigr), 
     \Bigl( \sum_{\nu\in\nabla^d} \omega^{-1}_\nu v_\nu  \Psi_\nu \Bigr)  \Bigr\rangle 
    \leq \overline{\lambda}_a  \overline{\lambda}_1 \, \norm{\bv}^2,\quad \bv\in \spl{2}(\nabla^d).
$$ 
Since
$$  
\Bigl\langle A \,\Bigl( \sum_{\nu\in\nabla^d} \omega^{-1}_\nu v_\nu \Psi_\nu \Bigr), 
     \Bigl( \sum_{\nu\in\nabla^d} \omega^{-1}_\nu v_\nu  \Psi_\nu \Bigr)  \Bigr\rangle 
    = \langle \Sc^{-1} \bT \Sc^{-1} \bv, \bv \rangle   \,,
$$
we arrive at \eqref{cond1}.

As shown in \cite{Dijkema:09}, the dependence on $\overline{\lambda}_a / \underline{\lambda}_a$ can in fact be eliminated in the case of \emph{diagonal} $(a_{ij})$. In fact, if one chooses $\omi{i}{\nu_i} \sim \sqrt{a_{ii} } 2^{\abs{\nu_i}}$, \eqref{deriv1cond} is replaced by
\begin{equation*}
  \underline{\lambda}_1^{(i)} \norm{\Sc_i \bv}^2  \leq  
    a_{ii}\, \Bignorm{\sum_{\nu\in\nabla^d} v_\nu\, \partial_i \Psi_\nu }^2_{\spL{2}(\Omega)} 
      \leq \overline{\lambda}_1^{(i)} \norm{\Sc_i \bv}^2  \,,
\end{equation*}
which holds independently of the diagonal entries $a_{ii}$, and thus summation of these inequalities over $i$ directly yields 
\eqref{cond2} %
in this case.
\end{proof}

\section{Approximation of Right Hand Sides}\label{app:rhs}

As a supplementary discussion, we consider approximations of right hand sides $\bbf$ that satisfy Assumptions \ref{ass:rhs}. 
A first possible model to account for the computational work of providing such approximations is to assume that  $\mathbf{f}$
is in fact already given in a finite hierarchical format with finitely supported mode frames. Then the realization of $\rhs$
simply reduces to applying the reduction operators discussed in Theorem \ref{lmm:combined_coarsening} with appropriate
target tolerances.

As for a second, perhaps more realistic model,
recall that in the problem \eqref{final} under consideration, we have $\bbf = \Sr^{-1} \bg$.
A routine $\rhs$ for constructing an approximation can thus be obtained by combining independent approximations of $\bg$ and $\Sr^{-1}$. 
Assuming that we have sufficient knowledge of the coefficients $g_\nu = \langle \Psi_\nu, f\rangle$, we can use the decay of the coefficients of $\Sr^{-1}\bg$ and a known low-rank structure of $\bg$, combined with some excess regularity $f\in H^{-1+t}(\Omega)$, $t>0$, to find $\tilde n$ and $\tilde\bg$ such that $\norm{\Sr^{-1}\bg - \Sa{\tilde n}^{-1}\tilde\bg}$ is sufficiently small.
We first make this precise under fairly general assumptions in the following proposition, and subsequently give some examples for its application.

\begin{proposition}
\label{prop:rhs}
Assume that  the excess regularity assumptions \eqref{eq:coordwise_regularity}, \eqref{eq:coordwise_reg_f} of order $t>0$ hold,
and that  $\norm{\pi^{(i)}(\Sr^{-1}\bg)}_\As = \norm{\pi^{(i)}(\bbf)}_\As < \infty$. Moreover,  let $\bg$ have known low-rank structure in the following sense: given any finite $\Lambda = \Lambda^{(1)}\times\cdots\times  \Lambda^{(d)}\subset \nabla^d$, then for each $\varepsilon>0$, we have at our disposal a $\bg_\varepsilon$ such that 
\begin{equation}\label{proprhs_assumption}
\norm{\Sr^{-1}(\Restr{\Lambda}\bg - \bg_\varepsilon)} \leq\varepsilon\,, 
\quad
\pi^{(i)}_\nu(\Sr^{-1}\bg_\varepsilon) \le \hat C \pi^{(i)}_\nu (\bbf) \text{ for $\nu\in\Lambda^{(i)}$, $i=1,\ldots,d$},
\end{equation}
with an absolute constant $\hat C$, %
and $\abs{\rank(\bg_\varepsilon)}_\infty \leq C^{\text{{\rm rank}}}_\bg \abs{\ln\varepsilon}^{b_\bg}$ holds for some constants 
$C^{\text{{\rm rank}}}_\bg, b_\bg$, depending only on $\bg$.
Then there exists an absolute constant $C$ such that for any given $\eta > 0$, we can construct $\bbf_\eta$ satisfying  
\beqn
\label{firstclaim}
\norm{\bbf - \bbf_\eta} \leq \eta,\quad  \norm{\pi^{(i)}(\bbf_\eta)}_\As \le C \norm{\pi^{(i)}(\bbf)}_\As ,
\quad  \norm{\Sc_{i}^{t} \bbf_\eta} \leq C  \norm{\Sc_{i}^{t} \bbf},
\quad i=1,\ldots,d,
\eeqn
as well as 
\beqn
\label{secondclaim}
  \abs{\rank(\bbf_\eta)}_\infty \le C \bigl[ C_\bg   + \abs{\ln\eta} \bigr] \abs{\ln \eta}^{b_\bg}, 
  \quad
   \sum_{i=1}^d \#\supp_i \bbf_\eta  \le d C \eta^{-\frac1s} \Bigl( \sum_i \norm{\pi^{(i)}(\bbf)}_\As \Bigr)^{\frac1s} \,.
\eeqn
\end{proposition}

\begin{proof}%
Note first that we may assume $\eta < \norm{\bbf}$, since otherwise $\bbf_\eta := 0$ satisfies our requirements.
We construct $\bbf_\eta$ with the asserted properties in several steps. First we exploit the excess regularity \eqref{eq:coordwise_reg_f} of order $t>0$.
In fact, choosing $\Lambda_k:= \{\nu\in\nabla^d \colon \max_i\abs{\nu_i} \leq k \}$ and defining
$\bg_k := \Restr{\Lambda_k}\bg$, we have, in view of \eref{omiscale}, for some  constant $C$ depending only on $t$,
\begin{eqnarray*}
\norm{\Sr^{-1}(\bg - \bg_k)}^2 %
&\le & 2^{-2kt } \sum_{\nu\notin\Lambda_k} 2^{2tk}(\Sr^{-1}\bg)_\nu^2
 \le C2^{-2kt} \sum_{\nu\notin\Lambda_k}
\omega_\nu^{2t} (\Sr^{-1}\bg)_\nu^2  \\
&\le &  C2^{-2tk } \norm{\Sc^t \bbf}^2.
\end{eqnarray*}
Thus, for any fixed $c_1 >0$, to be specified later,  %
we obtain
\beqn
\label{first}
 \norm{\Sr^{-1}(\bg - \bg_k)}\le c_1 \eta\quad\mbox{when}\quad k\ge k(\eta)= \lceil (t\ln 2)^{-1}\ln(c_1C\|\Sc^t\bbf\|/\eta)\rceil,
 \eeqn
and set $\bg^*:= \bg_{k(\eta)}$. Given $\bg^*$ we can find by assumption \eqref{proprhs_assumption} for any fixed $c_2 >0$ a $\bg_{c_2\eta}$ such that
\beqn
\label{lr_c2}
\|\Sr^{-1}(\bg^* - \bg_{c_2\eta})\|\le c_2\eta,\quad \rank(\bg_{c_2\eta})\lesssim \abs{\ln \eta}^{b_\bg},
\eeqn
with a constant that depends only on $\bg$ and $c_2$. Furthermore, since
$$
\|\pi^{(i)}(\Sr^{-1}\bg^*)\|_{\As}\le \|\pi^{(i)}(\Sr^{-1}\bg)\|_{\As} = \|\pi^{(i)}(\bbf)\|_{\As},\quad i=1,\ldots,d,
$$
we can find $\tilde\Lambda = \tilde\Lambda^{(1)}\times \cdots\times \tilde\Lambda^{(d)}$ with $\tilde\Lambda \subset \Lambda_{k(\eta)}$,  
 such that 
\beqn
\label{c3}
\norm{\Sr^{-1}(\Restr{\tilde\Lambda} \bg^* -\bg^*)} \leq c_3\eta,\quad %
  \sum_i \#\supp_i(\Restr{\tilde\Lambda} \bg^*)  \leq d \,C^{\frac1s} \eta^{-\frac1s} \Bigl( \sum_i \norm{\pi^{(i)}(\bbf)}_\As \Bigr)^{\frac1s} \,,
\eeqn
where $C$ depends only on $c_3$. Defining 
\beqn
\label{fdef}
\bbf_\eta := \Sa{n(\eta)}^{-1}\Restr{\tilde\Lambda}  \bg_{c_2\eta},
\eeqn
one has
\begin{align*}
\|\bbf - \bbf_\eta\| &= \|\Sr^{-1}\bg -  \Sa{n(\eta)}^{-1} \Restr{\tilde\Lambda} \bg_{c_2\eta}\| \\
&\le  \|\Sr^{-1}(\bg - \bg^*)\| + \|\Sr^{-1}(\bg^* - \Restr{\tilde\Lambda} \bg^*)\| +\|\Sr^{-1}\Restr{\tilde\Lambda} (\bg^*-\bg_{c_2\eta} )\|\nonumber\\
&\qquad +\| \Sr^{-1}\Restr{\tilde\Lambda}\bg_{c_2\eta} - \Sa{n(\eta)}^{-1} \Restr{\tilde\Lambda} \bg_{c_2\eta}\|\nonumber\\
&\le  (c_1 + c_3 + c_2)\eta + \|(\id - \Sr \Sa{n(\eta)}^{-1})\Sr^{-1}\Restr{\tilde\Lambda} \bg_{c_2\eta}\|.\nonumber\\
&\le  (c_1 + c_3 + c_2)\eta + \norm{(\id - \Sr \Sa{n(\eta)}^{-1})\Restr{\tilde\Lambda}} \bigl(\| \bbf\| + c_2 \eta),
\end{align*}
where we have used \eref{first}, \eref{lr_c2}, and \eref{c3}. We now fix $c_1=c_2=c_3 = \frac16$. In order to bound 
$\|(\id - \Sr \Sa{n(\eta)}^{-1})\Restr{\tilde\Lambda}\|$, we have to choose $n(\eta)$ large enough to apply \eref{tilde-diff}.
Specifically, we have to find a $T$ such that $\tilde\Lambda\subset \Lambda_T$. Recalling that $\tilde\Lambda \subset \Lambda_{k(\eta)}$,
the highest level occurring in $\tilde\Lambda$ is at most $k(\eta)=  \lceil (t\ln 2)^{-1}\ln(C\|\Sc^t\bbf\|/(6\eta))\rceil$. Hence, by \eref{omiscale},
for all $\nu\in \tilde\Lambda$ one has $\omega_\nu \leq C\sqrt{d}2^{k(\eta)}$, which by \eqref{first} means 
$
 \omega_\nu \leq \sqrt{d} (C \norm{\Sc_i^t   \bbf})^{\frac1t} \eta^{-\frac1t},%
$
where $C$ depends only on $t$. Thus $ \omega_\nu\leq \sqrt{d T } \,\hatomin \leq \sqrt{T} \omin$ holds if 
$$\textstyle\frac12 \displaystyle\ln T =  \ln [\hatomin^{-1}  (C \norm{\Sc_i^t \bbf})^{\frac1t}] + t^{-1} \abs{\ln \eta} .$$ 
Note that by \eqref{SStilde}, $\|(\id - \Sr \Sa{n(\eta)}^{-1})\Restr{\tilde\Lambda}\|\le
 (1-\delta)^{-1} \|(\id - \Sr \Sa{n(\eta)}^{-1})\Restr{\tilde\Lambda}\|$.
 In order to ensure that the latter expression is bounded by $\frac12 \eta/(\|\bbf\| + c_2\eta)$, on account of \eqref{Nsc},
 \eqref{tilde-diff}, and $\eta < \norm{\bbf}$, it suffices to choose
 $$  
n(\eta) \geq M\biggl ( \frac{(1-\delta) \eta}{2 (1 + c_2) \|\bbf\|} \; ; \; \hatomin^{-2} (C \norm{\Sc_i^t \Sr^{-1} \bg})^{\frac2t} \eta^{-\frac2t} \biggr),
 $$
with $M$ defined  in \eqref{Nsc}. In summary, we therefore conclude that with a constant $C=C(t,\delta)$, depending on $t,\delta$,
and a constant $C_\bg$, depending only on $\bg$, we may take
\beqn
\label{neta}
n(\eta) := \bigl\lfloor C(\delta, t) ( C_\bg   + \abs{\ln\eta} ) \bigr\rfloor 
\eeqn
to ensure that $\bbf_\eta$, defined in \eqref{fdef}, satisfies the first relation in \eqref{firstclaim}. 
The second and third relation in \eqref{firstclaim} follow with the second part of the assumption \eqref{proprhs_assumption} and with $(\Sr\Sa{n(\eta)}^{-1})_\nu \leq 1$ for all $\nu$ by \eqref{smaller}.
Since $\supp_i( \Sa{n(\eta)}^{-1}  \bg_{c_2\eta})
\subseteq \supp_i(\Restr{\Lambda^*}\bg)$, the second relation in \eref{secondclaim} follows from the second relation in \eref{c3}.
The first relation in \eref{secondclaim} is a consequence of \eref{neta} and the second equation in \eref{lr_c2}.
 \end{proof}

The assumptions of Proposition \ref{prop:rhs} cover several possible scenarios which  we outline next.
 
\begin{example}
Proposition \ref{prop:rhs} applies if $\rank(\bg) < \infty$ and $\norm{\pi^{(i)}(\Sr^{-1}\bg)}_\As < \infty$. This holds in particular if $f$ can be written in the form $f = \sum_{k=1}^r f^{(1)}_k \otimes \cdots \otimes f^{(d)}_k$ and the coefficients $\langle f^{(i)}_k, \psi_{\nu}\rangle$, $\nu\in\nabla$, have sufficient decay. In our numerical tests, we consider $f\equiv 1$, where this is the case, but the treatment of functionals with $f\notin \spL{2}$ is possible as well.
For instance, for functionals $f$ corresponding to inhomogeneous Neumann boundary data, if we prescribe constant values $c^{(i)}_0, c^{(i)}_1 \in\R$, $i=1,\ldots,d$, on the $2d$ faces of $(0,1)^d$ we obtain
$$  f = \sum_{i=1}^d \bigl( c^{(i)}_0 \operatorname{tr}_{\{x_i = 0\}}  +  c^{(i)}_1 \operatorname{tr}_{\{x_i = 1\}} \bigr)  \,. $$
Since each arising trace operator $\operatorname{tr}$ has the form of a point evaluation in a single variable tensorized with the identity in the remaining variables, the resulting coefficients $\bg$ can be represented with hierarchical rank 2 similarly to Example \ref{ex:laplace}. Non-constant Neumann boundary data can be treated similarly, provided that they have suitable tensor structure.
\end{example}

\begin{example}
If $f$ is such that the corresponding coefficients $\bg$ are not of finite rank, we additionally need some means to generate low-rank approximations on given finite sets of basis indices.
In principle, given a suitable index set $\Lambda$, if we can only evaluate the coefficients $\bg_\nu$ for $\nu\in\Lambda$, one could use $\mathcal{H}$SVD truncation of the resulting \emph{full} tensor $\bbf = \Sr^{-1}\bg$ on $\Lambda$ to directly construct $\bbf_\eta$ satisfying \eqref{firstclaim}, \eqref{secondclaim} (where the second inequality in \eqref{firstclaim} follows from \eqref{Ppieta}). Due to the costs of computationally constructing a $\mathcal{H}$SVD of a full tensor, this strategy is practically applicable only in the special situation that such a decomposition can be obtained more cheaply by some different (e.g.\ semi-analytical) means. 
In case that an $\mathcal{H}$SVD of $\bbf$ is not practically available, one may need to resort to more problem-specific low-rank approximations $\bg_\varepsilon$ that possibly do not have such orthogonality properties; for instance, for a number of important classes of functions, suitable approximations can be obtained by exponential sum expansions similarly to those considered for different purposes in Section \ref{ssec:nearsep}. In this case, one needs to ensure by construction of $\bg_\varepsilon$ that $\|\pi^{(i)}(\Sr^{-1}\bg_\varepsilon)\|_\As \le \hat C\|\pi^{(i)}(\bbf)\|_\As$ is satisfied, in other words, the low-rank approximation should not destroy the approximate sparsity of $\bg$. A sufficient condition for this to hold is that each entry $g_\nu$ for $\nu\in\Lambda$ is approximated with a bounded \emph{relative} error tolerance.
\end{example}

\begin{remark}
\label{rem:rhs_ops}
If the coefficients in the tensor representation of $\bg_\varepsilon$ in Proposition \ref{prop:rhs} can be produced directly at unit cost, for instance based on analytical knowledge of $f$,   the number of operations required to construct $\bbf_\eta$ can be estimated by
$$
  \ops(\rhs(\eta)) \lesssim d \bigl[ \bigl(C_\bg + \abs{\ln\eta}\bigr) \abs{\ln\eta}^{b_\bg}  \bigr]^3 
    + d\, \Bigl( \sum_i \norm{\pi^{(i)}(\bbf)}_\As \Bigr)^{\frac1s}  \bigl(C_\bg + \abs{\ln\eta}\bigr) \abs{\ln\eta}^{b_\bg} \eta^{-\frac1s}  \,.
$$
\end{remark}

\end{appendix}

\bibliography{bd_alr}

\end{document}